\documentclass[1996/10/24]{amsart}

\usepackage{color}

\allowdisplaybreaks[3]

\usepackage{amsfonts}
\usepackage{epsf,latexsym,amsfonts,amsbsy,mathrsfs}
\usepackage{amsmath, amssymb, amsthm,amscd,amsxtra}
\usepackage[]{graphicx,subfigure}
\usepackage{epstopdf}
\usepackage{float}
\usepackage{multirow}
\usepackage{epsfig}
\usepackage{tikz}
\usetikzlibrary{positioning} 
\usepackage{xcolor}
\usepackage{marginnote}
\usepackage{sidenotes}

\usepackage{changes}
\usepackage{lipsum}

\newtheorem{theorem}{Theorem}[section]
\newtheorem{lemma}[theorem]{Lemma}

\theoremstyle{definition}

\newtheoremstyle{italicremark}
{3pt}
{3pt}
{\itshape}
{}
{\bfseries}
{.}
{ }
{}

\theoremstyle{italicremark}
\newtheorem{remark}[theorem]{Remark}

\numberwithin{equation}{section}

\newcommand{\ignore}[1]{}

\newtheorem{rem}{Remark}[section]

\begin{document}
	
	\title[]{ bound-preserving Adaptive Time-Stepping Method with Energy Stability for Simulating Compressible Gas Flow in Poroelastic Media}
	
	\author{Huangxin Chen}
	\address{School of Mathematical Sciences and Fujian Provincial Key Laboratory on Mathematical Modeling and
		High Performance Scientific Computing, Xiamen University, Fujian, 361005, China}
	\email{chx@xmu.edu.cn}
	\author{Yuxiang Chen}
	\address{School of Mathematical Sciences and Fujian Provincial Key Laboratory on Mathematical Modeling and
		High Performance Scientific Computing, Xiamen University, Fujian, 361005, China}
	\email{chenyuxiang@stu.xmu.edu.cn}
	\author{Jisheng Kou}
	\address{State Key Laboratory of Intelligent Deep Metal Mining and Equipment,  Shaoxing University, Shaoxing 312000, Zhejiang, China} 
	\email{jishengkou@163.com}
	\author{Shuyu Sun}
	\address{School of Mathematical Sciences, Tongji University, Shanghai 200092, China.}
	\email{suns@tongji.edu.cn}
	\subjclass[2010]{}
	

	\date{}
	
	\begin{abstract}
		{In this paper, we present an efficient numerical method to address a thermodynamically consistent gas flow model in porous media involving compressible gas and deformable rock. The accurate modeling of gas flow in porous media often poses significant challenges due to their inherent nonlinearity, the coupling between gas and rock dynamics, and the need to preserve physical principles such as mass conservation, energy dissipation and molar density boundedness. The system is further complicated by the need to balance computational efficiency with the accuracy and stability of the numerical scheme. To tackle these challenges, we adopt a stabilization approach that is able to preserve the original energy dissipation while achieving linear energy-stable numerical schemes. We also prove the convergence of the adopted linear iterative method. At each time step, the stabilization parameter is adaptively updated using a simple and explicit formula to ensure compliance with the original energy dissipation law. The proposed method uses adaptive time stepping to improve computational efficiency while maintaining solution accuracy and boundedness. The adaptive time step size is calculated  explicitly at each iteration, ensuring stability and allowing for efficient handling of highly dynamic scenarios.  A mixed finite element method combined with an upwind scheme is employed as spatial discretization   to ensure mass conservation and stability. Finally, we conduct a series of numerical experiments to validate the performance and robustness of the proposed numerical method.}
	\end{abstract}
	
	\maketitle
	
	\section{Introduction}
	Porous media flow processes often involve complex physical phenomena. Traditional analytical methods are difficult to accurately describe these complex phenomena, but numerical simulations can flexibly handle these complex interactions by establishing accurate mathematical models. With the development of computing technology and advancement of experimental methods, researchers continue to propose more accurate models and efficient numerical methods to gradually solve complex flow problems. Continued research in this area will provide important support for energy development, environmental protection and industrial optimization \cite{Chen2006,Firoo1999}. The phenomenon of gas flow inside porous media has become a popular research topic in the field of modeling and simulation \cite{El2018, Guo2018, Mikyska2014}. In porous media, the classic flow model mainly describes the penetration and flow behavior of gas in porous media. These models are based on classical fluid mechanics theory and describe the flow characteristics of gases through a series of assumptions and simplifications. When simulating gas flow in porous media, energy stability and rock compressibility are two factors that must be considered. They respectively affect the thermodynamic behavior of gas flow and the structural response of porous media. Biot's poroelasticity theory is used to explain the deformation of solid skeletons \cite{Biot1941}. This theory first obtains the displacement of the solid by solving the solid mechanics equation, and then calculates the change in porosity based on these displacements. Energy stability requires the model to adhere to the law of energy dissipation, which is derived from the second law of thermodynamics. It states that, in an isothermal, closed system, total free energy decreases over time during spontaneous processes until the system reaches a stable or equilibrium state. In this paper, we consider the compressible single-phase gas flow model in porous media that satisfies the law of energy dissipation, as proposed in paper \cite{Chencmame2024}. The model explains the effect of rock deformation on porosity by introducing the pore elasticity equation and the state equation for porosity variation. The Peng-Robinson equation of state (PR-EoS) \cite{Robinson1976} is used to characterize the relationship between molar density and pressure at a given temperature. Additionally, the chemical potential gradient is used to replace the pressure gradient as the main driving force. Finally, the free energy of the rock skeleton is introduced to account for the compressibility of the rock.

	Due to our physical model strictly adhering to the principles of thermodynamic consistency, the numerical methods used for its implementation must also satisfy this critical property, as violations can lead to non-physical and unstable results \cite{Chen2006,Kou2018,kousun2018,Casas-Vazquez2008}. Established approaches range from the rigorously nonlinear convex splitting method \cite{Eyre1998,Qiao2014}, which guarantees stability at the cost of solving nonlinear systems, to more efficient linear strategies. The latter category includes versatile stabilization methods \cite{Ju2021,Li2023,Shen2018} and structure-preserving exponential time-differencing (ETD) schemes \cite{Beylkin1998,Cox2002}, which offer improved computational efficiency.     Furthermore, the invariant energy quadratization (IEQ) \cite{Ju2017, Wang2017}  and scalar auxiliary variable (SAV) \cite{Shen2018, Shen2019} approaches provide linear and easily implementable energy-stable schemes, though the discrete energy they produce typically differs from the original functional. Another notable method in this domain is the energy factorization (EF) approach \cite{kou2020, kou2022}, which also operates within an efficient linear framework while successfully preserving the dissipation characteristics of the original energy. Building upon this foundation, our work adopts the novel stabilization method proposed by Kou et al. \cite{kou2023}, which achieves a linear energy-stable scheme while exactly preserving the original energy functional.
    
Beyond energy stability, the positivity and boundedness of molar density are fundamental and essential physical constraints in the numerical simulation of compressible gas flow in porous media. Designing numerical schemes that respect these bounds is crucial for obtaining physically reasonable solutions. The exploration of provably bound-preserving schemes has led to various strategies, including the Lagrange multiplier approach \cite{Cheng2022, IICheng2022}, the variational approach  \cite{Joshi2018, Yang2017}, the post-processing approach \cite{Zhang2010}, the cut-off approach \cite{Yang2022cutoff}, and the nonlinear convex splitting approach \cite{Wise2012, Dong2021, Shen2021}. Recent advances have further produced high-order methods that simultaneously ensure energy stability and solution boundedness. These include stabilized linear schemes requiring only linear \cite{Tang2016,Shen2016}, auxiliary variable methods like SAV combined with cut-off operations \cite{Akrivis2019,Shen2019}.  Building on the foundation of exponential integrators, a significant advancement was achieved by Ju et al. \cite{Ju2022,JuJCS2022}, who ingeniously combined SAV with stabilized first-order ETD and ETDRK2 methods. This hybrid approach successfully constructed schemes that simultaneously preserve the original energy-dissipation law and the maximum bound principle (MBP) for a class of Allen-Cahn type gradient flows. However, in the specific context of gas flow in porous media, many overshoot/undershoot correction techniques may compromise numerical accuracy and critical properties like local mass conservation. To address this challenge, the present paper adopts an adaptive time-stepping strategy to maintain the molar density bounds. Specifically, we employ the novel stabilization method from \cite{kou2023} to separate the discretized chemical potential into explicit and semi-implicit stability terms. This allows us to explicitly handle the velocity in the mass conservation equation during linear iterations, thereby deriving an explicit formula to adaptively compute the time step at each iteration, ensuring both thermodynamic consistency and solution boundedness.  Meanwhile, a mixed finite element method with upwind schemes is used for spatial discretization to guarantee mass conservation, while the momentum equation of poroelasticity is solved using a discontinuous Galerkin method to prevent the locking phenomenon \cite{Phillips2008}.
	
The paper is organized as follows. In section 2, we introduce the formulation of gas flow model with rock compressibility. In section 3, we propose the fully discrete numerical scheme. In section 4, we analyze the selection criteria for the time-adaptive algorithm. In section 5, the convergence of the iterative method and the energy stability are rigorously proven. In section 6, numerical experiments are presented to validate the proposed methods.
	
	\section{Mathematical model}
	In this section, we present the mathematical model for a thermodynamically consistent gas flow system in porous media, which accounts for the interaction between compressible gas and deformable rock. The model is designed to ensure thermodynamic consistency, capturing the essential physical processes such as gas flow, rock deformation, and their mutual coupling. 
	
	The formulation of the gas flow model with compressible rock in porous media is presented as follows (cf. \cite{Chencmame2024})
	\begin{subequations}\label{eq-con-1}
		\begin{align}
			&\nabla\cdot\mathbf{\sigma}(\boldsymbol{u}_{s}, p) = 0, \qquad\qquad{\rm in}~\Omega_{t}=:\Omega\times(0,t),\label{eq-1-1}\\
			&\frac{\partial (\phi c)}{\partial t} + \nabla\cdot(\boldsymbol{u}_{f}c) = 0 ,\quad \,    {\rm in}~\Omega_{t},\label{eq-1-2-c}\\
			&\boldsymbol{u}_{f} = -\lambda(\phi)c\nabla \mu, ~\qquad \    \ \ \ {\rm in}~\Omega_{t},\\
			&p = c\mu(c) - f(c), \qquad\ \ \ \ \    {\rm in}~\Omega_{t},\label{eq-1-2-e}\\
			&\frac{\partial \phi}{\partial t} = \frac{1}{N}\frac{\partial p}{\partial t} + \alpha  \nabla\cdot \textbf{v}_s,  \quad \    {\rm in}~\Omega_{t}.\label{eq-1-6}
		\end{align}
	\end{subequations}
	Here, the mobility $\lambda(\phi)$ is defined as
	
	\[
	\lambda(\phi) = \frac{\kappa(\phi)}{\nu},
	\]
	where $\nu$ is the viscosity of the gas, $\phi$ is the porosity and $\kappa(\phi)$ represents the permeability by the Kozeny-Carman model (cf. \cite{Chen2006}) as
	
	\[
	\kappa(\phi) = \kappa^{0} \left(\frac{\phi}{\phi_r}\right)^{3} \left(\frac{1-\phi_r}{1-\phi}\right)^{2}.
	\]
	Here, $\kappa^0$ is the initial intrinsic permeability, $\phi_r$ is the porosity at the reference pressure.

	The stress tensor $\boldsymbol{\mathbf{\sigma}}(\boldsymbol{u}_{s}, p)$ is given by	
	\begin{align*}
	\boldsymbol{\sigma}(\boldsymbol{u}_{s}, p) = \boldsymbol{\sigma}_e - \alpha p \boldsymbol{I},  \quad
\boldsymbol{\sigma}_e =  2\eta \boldsymbol{\varepsilon}(\boldsymbol{u}_s) + \gamma \operatorname{div} (\boldsymbol{u}_s) \boldsymbol{I}, 
	\end{align*}
	where $\boldsymbol{\varepsilon}(\boldsymbol{u}_s) = \frac{1}{2}\left(\nabla \boldsymbol{u}_s + (\nabla \boldsymbol{u}_s)^{T}\right)$ is the strain tensor, $\boldsymbol{\sigma}_e$ is the effective stress tensor, $\boldsymbol{I}$ is the unit tensor, $\boldsymbol{u}_s$ is the displacement of the solid phase, $\boldsymbol{v}_s = \frac{\partial \boldsymbol{u}_s}{\partial t}$ is the velocity of the solid phase, $\eta$ and $\gamma$ are the Lamé parameters, $\alpha$ is the Biot's coefficient, $N = \frac{(\alpha - \phi) }{K_s}$ is the Biot's modulus, $K_s$ is the bulk modulus of the solid grains, $c$ is the molar density of the gas.
	
	The pressure $p = p(c)$ is given by
	\begin{align}\label{eq-pres-peng}
		p=\frac{c R T}{1-\beta c}-\frac{b c^2}{1+2 \beta c-\beta^2 c^2}.
	\end{align}
	
	The Helmholtz free energy density $f(c)$ and the chemical potential $\mu(c)$ are determined by the Peng-Robinson equation of state. These quantities are expressed as
	\begin{subequations}\label{eq-f-1}
		\begin{align}
			&f(c)=f_{i d e}(c)+f_{r e p}(c)+f_{a t t}(c),\\
			&f_{\text {ide }}(c)=c R T \ln (c),\\
			&f_{\text {rep }}(c)=-c R T \ln (1-\beta c),\\
			&f_{a t t}(c)=\frac{b(T) c}{2 \sqrt{2} \beta} \ln \left(\frac{1+(1-\sqrt{2}) \beta c}{1+(1+\sqrt{2}) \beta c}\right),
		\end{align}
	\end{subequations}
	and
	\begin{align}
		\mu(c)=f^{\prime}(c).
	\end{align}	
	
	This formulation integrates the key physical mechanisms, including poroelasticity, gas compressibility and non-ideal gas behavior, which provides a robust framework for modeling gas flow in deformable porous media. The use of the Peng-Robinson equation of state ensures precise representation of the gas phase's thermodynamic properties \cite{Robinson1976}. In equation \eqref{eq-f-1}, the term $f_{\text{ide}}$ represents the contribution from an ideal, homogeneous gas, whereas $f_{\text{rep}}$ and $f_{\text{att}}$ describe the effects arising from intermolecular repulsive and attractive forces, respectively. We introduce the critical temperature $T_c$ and define the reduced temperature as $T_r = T / T_c$. The parameters $b$ and $\beta$ are determined based on the fluid’s critical properties and its acentric factor.
	$$
	b=0.45724 \frac{R^2 T_c^2}{P_c}\left(1+m\left(1-\sqrt{T_r}\right)\right)^2, \quad \beta=0.07780 \frac{R T_c}{P_c},
	$$
	where $P_c$ is the critical pressure and $m$ relies on the acentric factor $\omega$
	$$
	\begin{array}{c}
		m=0.37464+1.54226 \omega-0.26992 \omega^2,\quad \omega \leq 0.49, \\
		m=0.379642+1.485030 \omega-0.164423 \omega^2+0.016666 \omega^3,\quad \omega>0.49 .
	\end{array}
	$$

	In order to close the system, the following boundary conditions are imposed as
	\begin{align}
		\mathbf{\sigma}(\boldsymbol{u}_s, p)\cdot\boldsymbol{n} = 0, \qquad {\rm on}~\partial\Omega,\label{eq-BD-1}\\
		\boldsymbol{u}_{f}\cdot\boldsymbol{n} = 0, \qquad {\rm on}~\partial\Omega.
	\end{align}
	
	The total free energy of \eqref{eq-con-1} and \eqref{eq-BD-1} in $\Omega$ can be defined as 
	\begin{align}\label{eq-energy_1}
		E(t) = 	\int_{\Omega}\left(\phi f(c) +\frac{1}{2}\mathbf{\sigma}_{e}(\boldsymbol{u}_{s}):\varepsilon(\boldsymbol{u}_{s}) + \frac{1}{2N} p^2\right) ~d \boldsymbol{x}.
	\end{align}
	\begin{remark}
		It is noteworthy that the total free energy of the system includes not only the fluid free energy and the elastic strain energy of the solid skeleton, but also an additional term $\frac{1}{2N} p^2$. This term accounts for the energy stored due to the compressibility of the solid framework under pore pressure, and is interpreted as the compressive storage energy.
		Under isothermal conditions, the variation of the solid density $\rho_s$ is governed by the following equation of state  \cite{Lewis1998}
		\begin{align}\label{eq-comeq}
			\frac{1}{\rho_s} \frac{\partial \rho_s}{\partial t} = \frac{1}{1 - \phi} \left[\frac{1}{N} \frac{\partial p}{\partial t} - (1 - \alpha) \nabla \cdot \boldsymbol{v}_s \right].
		\end{align}
		
		This equation shows that changes in solid density are influenced not only by volumetric strain, but also directly by pore pressure. Therefore, the energy term $\frac{1}{2N} p^2$ reflects the stored energy resulting from the compressibility of the solid matrix. Substituting Equation \eqref{eq-comeq} into the mass conservation equation for the solid yields the porosity equation \eqref{eq-1-6} (cf. \cite{Lewis1998}). When the solid phase approaches incompressibility, the bulk modulus $N$ tends to infinity, and consequently the term $\frac{1}{2N} p^2$ becomes negligible. This implies that the influence of pore pressure on the solid density can be disregarded. Under this assumption, the total free energy of the system can be simplified to the sum of the fluid free energy and the elastic strain energy of the solid skeleton. A systematic formulation of the free energy under the incompressibility assumption for the solid phase is provided in Reference \cite{kou2023}.
	\end{remark}
	\section{Energy stable and bounds-preserving numerical method}
	In this section, we first introduce the time semi-discrete scheme for the problem. Following this, we describe the discrete function spaces and propose the fully discrete scheme based on these spaces. Finally, we provide a detailed analysis of the scheme, including proofs of the boundedness of molar density, the convergence of the iterative method, and the energy stability of the system. These results collectively establish the theoretical foundation and numerical robustness of the proposed approach.
	\subsection{Stable time semi-discrete scheme}	
	Let $X^{n+1}$ denote the  approximation of $X$ at time $t = t_{n+1}$, where X represents any of the variables $\phi, \boldsymbol{u}_s, \boldsymbol{u}_f, p, c,$ or  $\mu$. To construct the time semi-discrete scheme, we employ a novel stabilization method \cite{kou2023}, which ensures numerical stability and accuracy in the temporal discretization. The stabilization terms are carefully designed to maintain the physical consistency of the system while improving computational efficiency.
	
	Using the novel stabilization strategy, we define the following stabilized discrete chemical potential
	\begin{align}\label{eq-stable}
		\mu^{n+1} = \mu(c^{n}) + \theta_n R T \zeta^{n+1},
	\end{align}
    where $\theta_n$ is an adaptive stabilization parameter, and its selection criterion is given in Theorem 5.3. $\zeta^{n+1}$	represents the stabilization term, which depends on the discrete approximations of the concentration $c^{n+1}$ at time $t_{n+1}$:
	\begin{align}
		\zeta^{n+1} := \frac{c^{n+1}-c^{n}}{c^{n}(1-\beta c^{n})^2}.
	\end{align}
	This stabilization term is designed to enhance numerical stability, and ensure the preservation of energy dissipation.
	
	\begin{subequations}\label{eq-Semi-1}
		\begin{align}
			& - \nabla \cdot \mathbf{\sigma}(\boldsymbol{u}_{s}^{n+1}, p^{n+1}) = 0,\label{eq-2-2}\\
			& D_{\tau}(\phi^{n+1}c^{n+1}) + \nabla\cdot(\boldsymbol{u}_{f}^{n+1}c^{n}) = 0,\label{eq-2-3}\\
			& \boldsymbol{u}_{f}^{n+1} = -\lambda(\phi^{n}) c^{n} \nabla \mu^{n+1},\label{eq-2-4}\\
			& p^{n+1} = c^{n}\mu^{n+1} - f(c^{n}),\label{eq-2-5}\\
			& D_{\tau}\phi^{n+1} = \frac{1}{N}D_{\tau}p^{n+1} + \alpha \nabla\cdot D_{\tau} \boldsymbol{u}_{s}^{n+1},\label{eq-2-6}
		\end{align}
	\end{subequations}
	where $D_{\tau} X^{n+1} = \frac{X^{n+1} - X^{n}}{\tau}$.

	\subsection{{Full discretization}}
	Let $\mathcal{K}_h$ denote a family of non-degenerate, quasi-uniform triangulations of $\Omega$, with the diameter of an element $K \in \mathcal{K}_h$ denoted by $h_K$. We define the standard finite element space of $d$-vectors whose components are piecewise linear polynomials as follows
	\begin{align*}
		\mathcal{V}_h:=\left\{\psi \in L^2(\Omega)^d:\psi|_K \in \mathbb{P}_1(K)^d, \forall K \in \mathcal{K}_h\right\},
	\end{align*}
	where $\mathbb{P}_1(K)$ represents the space of linear polynomials on the element $K$. The set of all faces (for $d = 3$) or edges (for $d = 2$) of the triangulation $\mathcal{K}_h$ is denoted by $\mathcal{E}_h$, while the set of interior edges or faces is denoted by $\mathcal{E}_h^I$. For two neighboring elements $K_i, K_j \in \mathcal{K}_h$, their common face or edge is denoted by $e = \partial K_i \cap \partial K_j \in \mathcal{E}_h^I$, and the outward unit normal vector $\boldsymbol{n}_e$ is oriented from $K_i$ to $K_j$.
	
	We define the average and jump of a function $\psi \in \mathcal{V}_h$ on an interior edge or face $e$ as follows
	$$
	\{\psi\}:=\frac{1}{2}((\psi|_{K_i})|_e+(\psi|_{K_j})|_e), \quad[\psi]:=\left.\left(\left.\psi\right|_{K_i}\right)\right|_e-\left(\psi|_{K_j}\right)|_e,
	$$
	where $\psi|_{K_i}$ denotes the restriction of $\psi$ to the element $K_i$.
	
	For any domain $D \subseteq \Omega$, and for scalar functions $\vartheta_1, \vartheta_2$ or vector functions $\boldsymbol{\vartheta}_1, \boldsymbol{\vartheta}_2$, we define the inner products as follows
	$$(\vartheta_1, \vartheta_2)_D = \int_D \vartheta_1\vartheta_2 ~d \boldsymbol{x}, \qquad (\boldsymbol{\vartheta}_1, \boldsymbol{\vartheta}_2)_D = \int_D \boldsymbol{\vartheta}_1\cdot \boldsymbol{\vartheta}_2 ~d \boldsymbol{x}.$$
	The $L^2$-norm on a domain $D$ is denoted by $\|\cdot\|_{L^2(D)}$, while the inner product on an edge or face $e \in \mathcal{E}_h$ is denoted by $\langle \cdot, \cdot \rangle_e$, and the associated norm by $\|\cdot\|_{L^2(e)}$.
	
	For dealing with the convection term, we adopt the upwind strategy. The upwind value of $c_h^n$ on an interior edge $e$ is defined as
	\begin{align}
		c^{n*}_{h}= \begin{cases}c^{n}_{h}|_{K_i}, & \boldsymbol{u}_{f,h}^{n} \cdot \boldsymbol{n}_{e} \geq 0, \\ c^{n}_{h}|_{K_j}, & \boldsymbol{u}_{f,h}^{n} \cdot \boldsymbol{n}_{e} <0. \end{cases}
	\end{align}
	Next, we introduce the lowest-order Raviart-Thomas finite element space $RT_0$. On a simplicial mesh, the space $RT_0$ is given by
	$$RT_0 = [\mathbb{P}_0]^d + \boldsymbol{x} \mathbb{P}_0,$$
	where $\mathbb{P}_0$ denotes the piecewise constant space. The corresponding finite element spaces are defined as follows
	\begin{align*}
		\mathcal{U}_h=\left\{\mathbf{v}_h \in H(\operatorname{div}, \Omega):\left.\mathbf{v}_h\right|_K \in R T_0(K), \forall K \in \mathcal{K}_h\right\},
	\end{align*}
	\begin{align*}	
		\mathcal{Q}_h=\left\{q_h \in L^2(\Omega):\left.q_h\right|_K \in \mathbb{P}_0(K), \forall K \in \mathcal{K}_h\right\},
	\end{align*}
	where $H(\operatorname{div}, \Omega)$ is the space of vector fields $\mathbf{v} \in [L^2(\Omega)]^d$ such that $\nabla \cdot \mathbf{v} \in L^2(\Omega)$.
	Finally, we define the subspace $\mathcal{U}_h^0 \subset \mathcal{U}_h$ as
	$$\mathcal{U}^{0}_h = \left\{\mathbf{v}\in\mathcal{U}_h: \mathbf{v}\cdot\boldsymbol{n} = 0 ~on~ \partial\Omega\right\}.$$
	
	For any $\textbf{v}_h \in \mathcal{V}_h, \mathbf{w}_h \in \mathcal{U}^0_h, q_h,z_h,\varphi_h\in \mathcal{Q}_h$, we find $\boldsymbol{u}^{n+1}_{s,h}\in \mathcal{V}_h,\boldsymbol{u}_{f,h}^{n+1}\in \mathcal{U}_h, c^{n+1}_{h},\phi^{n+1}_{h},p^{n+1}_h \in \mathcal{Q}_h $ such that
	\begin{subequations}\label{eq-full-1}
		\begin{align}
			&\mathcal{A}(\boldsymbol{u}^{n+1}_{s,h},  p_h^{n+1},\textbf{v}_h) = 0,  \label{eq-2-2-1}\\
			&(D_{\tau} (\phi^{n+1}_{h}c^{n+1}_{h}), q_h) + 
			\sum_{e\in \mathcal{E}_h^I}\langle c^{n*}_h\boldsymbol{u}_{f,h}^{n+1}\cdot\boldsymbol{n}, [q_h]\rangle_e + \sum_{e\in \mathcal{E}_h^I}\frac{\varsigma_1}{h_e}\langle[\mu^{n+1}_{h}], [q_h] \rangle_e= 0,\label{eq-2-2-2}\\
			&(\lambda^{-1}(\phi_{h}^{n}) \boldsymbol{u}_{f,h}^{n+1},\textbf{w}_h) = \sum_{e\in \mathcal{E}_h^I}\langle [\mu_h^{n+1}],c^{n*}_h\textbf{w}_{h}\cdot\boldsymbol{n}\rangle_e ,\label{eq-2-2-3}\\
			&(p^{n+1}_h, z_h) = (c_h^{n}\mu_h^{n+1} - f(c_h^{n}), z_h),\label{eq-2-2-4}\\
			&(D_{\tau} \phi^{n+1}_h,\varphi_h) = \frac{1}{N}(D_{\tau} p^{n+1}_h, \varphi_h) + \alpha(D_{\tau} (\nabla\cdot \boldsymbol{u}_{s,h}^{n+1}), \varphi_h) -\alpha\sum_{e \in \mathcal{E}_h^I}\langle\{\varphi_h\boldsymbol{n}_e\}, [D_{\tau}\boldsymbol{u}^{n+1}_{s,h}]\rangle_e,\label{eq-2-2-5}
		\end{align}
	\end{subequations}
	where $\mathcal{A}$ is the bilinear form defined as 
	\begin{align}
		\mathcal{A}(\boldsymbol{u}_{s,h}, p_h,\textbf{v}_h):= &\sum_{K \in \mathcal{K}_h}(\mathbf{\sigma}_e(\boldsymbol{u}_{s,h}), \varepsilon(\textbf{v}_h))_K-\sum_{e \in \mathcal{E}^{I}_h}\langle\{\mathbf{\sigma}_e(\boldsymbol{u}_{s,h}) \boldsymbol{n}_e\}, [\textbf{v}_h]\rangle_e \\\nonumber
		&- \alpha\sum_{K \in \mathcal{K}_h}(p_h, \nabla\cdot \mathbf{v}_h)_{K}+\alpha\sum_{e \in \mathcal{E}^{I}_h}\langle\{p_h\boldsymbol{n}_e\}, [\textbf{v}_h]\rangle_e\\\nonumber
		&-\sum_{e \in \mathcal{E}^{I}_h}\langle [\boldsymbol{u}_{s,h}],\left\{\mathbf{\sigma}_{e}(\textbf{v}_h) \boldsymbol{n}_e\right\}\rangle_e+\sum_{e \in \mathcal{E}^{I}_h} \frac{\varsigma_2}{h_e}\langle[\boldsymbol{u}_{s,h}], [\textbf{v}_h]\rangle_e.
	\end{align}
     Here, $\varsigma_1$ and $\varsigma_2$ are penalty parameters to be determined (see the Numerical Experiments section), and the inner product $(\cdot,\cdot) = \sum\limits_{K \in \mathcal{T}_h} (\cdot,\cdot)_K$ denotes the global inner product over the computational mesh $\mathcal{T}_h$, the notation $h_e$ simply means the length of $e$ in 2D and the area of $e$ in 3D.
      
	Since  \eqref{eq-2-2-2}  is nonlinear, we use the linear iteration method to solve the equations \eqref{eq-full-1}: 
	\begin{align}
		&\mathcal{A}(\textbf{u}^{n+1,l+1}_{s,h}, p_{h}^{n+1, l+1}, \textbf{v}_h) = 0,  \label{eq-3-2-1}\\
		&(\frac{\phi^{n+1,l}_{h}c^{n+1,l+1}_{h}-\phi^{n}_{h}c^{n}_{h}}{\tau}, q_h) + 
		\sum_{e\in \mathcal{E}_h^I}\langle c^{n*}_h\textbf{u}_{f,h}^{n+1,l}\cdot\boldsymbol{n}, [q_h]\rangle_e + \sum_{e\in \mathcal{E}_h^I}\frac{\varsigma_1}{h_e}\langle[\mu^{n+1,l+1}_{h}], [q_h] \rangle_e= 0,\label{eq-3-2-2}\\
		&(\lambda^{-1}(\phi_{h}^{n}) \textbf{u}_{f,h}^{n+1,l+1},\textbf{w}_h) = \sum_{e\in \mathcal{E}_h^I}\langle [\mu_h^{n+1,l+1}],c^{n*}_h\textbf{w}_{h}\cdot\boldsymbol{n}\rangle_e ,\label{eq-3-2-3}\\
		&(p^{n+1,l+1}_h, z_h) = (c_h^{n}\mu_h^{n+1,l+1} - f(c_h^{n}), z_h),\label{eq-3-2-4}\\
		&(D_{\tau} \phi^{n+1,l+1}_h,\varphi_h) = \frac{1}{N}(D_{\tau} p^{n+1,l+1}_h, \varphi_h) + \alpha(D_{\tau} (\nabla\cdot \textbf{u}_{s,h}^{n+1,l+1}), \varphi_h)\label{eq-3-2-5}\\\nonumber
		&-\alpha\sum_{e \in \mathcal{E}^{I}_h}\langle\{\varphi_h\boldsymbol{n}_e\}, [D_{\tau}\textbf{u}^{n+1, l+1}_{s,h}]\rangle_e.
	\end{align}
	\section{Bound-preserving Adaptive Time-Stepping strategies}
	Next, we will conduct some theoretical analysis to prove the convergence of the iterative method we adopted, the boundedness of the molar density, and the energy stability. These analyses are crucial for understanding the effectiveness and reliability of our numerical scheme. Specifically, we will firstly prove that the molar density remains bounded throughout the entire computational process. By performing a rigorous analysis of the boundedness of the molar density, we can ensure that, during the numerical solution process, the molar density stays within the physically acceptable range. Then, we demonstrate that during each iteration, the iterative method progressively converges to a stable solution, ensuring that no divergence occurs during the numerical computation.  Finally, we will demonstrate the stability of the proposed numerical scheme in terms of energy conservation. 
	\begin{lemma}\label{le-1}
		For two given real constants $a, b$ and $a < b$, for any $c_h \in \mathcal{Q}_{h}$, we define $c_{h,-} = \min(c_h+a, 0)$, $c_{h,+} = \max(c_h-b, 0)$. We have the following inequalities hold
		\begin{align}
			\sum_{e\in \mathcal{E}^{I}_h}\langle[c_h], [c_{h, -}]\rangle_e \geq \sum_{e\in \mathcal{E}^{I}_h} \langle[c_{h, -}], [c_{h, -}]\rangle_e,\label{eq-108-1}\\
			\sum_{e\in \mathcal{E}^{I}_h}\langle[c_h], [c_{h, +}]\rangle_e \geq \sum_{e\in \mathcal{E}^{I}_h} \langle[c_{h, +}], [c_{h, +}]\rangle_e.\label{eq-108-2}
		\end{align}
	\end{lemma}
	\begin{proof}
		By the definition $c_{h,-}$, we deduce that																																								
		\begin{align}
			&\sum_{e\in \mathcal{E}^{I}_h}\langle[c_h], [c_{h, -}]\rangle_e \\\nonumber
			&= \sum_{e\in \mathcal{E}^{I}_h}\left(\langle c_h|_{K}+a, c_{h, -}|_{K}\rangle_e - \langle c_h|_{K}+a, c_{h, -}|_{K'}\rangle_e\right.\\\nonumber
			&\left.\quad- \langle c_h|_{K'}+a, c_{h, -}|_{K}\rangle_e + \langle c_h|_{K'}+a, c_{h, -}|_{K'}\rangle_e\right)\\\nonumber
			&\geq \sum_{e\in \mathcal{E}^{I}_h}\left(\langle c_{h, -}|_{K}, c_{h, -}|_{K}\rangle_e - \langle c_{h, -}|_{K}, c_{h, -}|_{K'}\rangle_e\right.\\\nonumber
			&\left.\quad- \langle c_{h, -}|_{K'}, c_{h, -}|_{K}\rangle_e + \langle c_{h, -}|_{K'}, c_{h, -}|_{K'}\rangle_e\right)\\\nonumber
			& = \sum_{e\in \mathcal{E}^{I}_h}\langle[c_{h, -}], [c_{h, -}]\rangle_e,
		\end{align}
where $K'$ is the neighboring element that shares the edge (or face) $e$ with the element $K$.
		Similarly, we can also derive \eqref{eq-108-2}.
	\end{proof}
	\begin{lemma}\label{theo-1}
		Assume that $0 <\varrho_0 \leq  \beta c^n_h \leq \varrho< 1$ and the boundary condition \eqref{eq-BD-1} holds. For $n \geq 0$ and given constants $0 < \delta_1 < 1$ and $0 < \delta_2 < 1$, if the time step size $\tau^{l}_n$ satisfies
		\begin{align}\label{eq-tau}
			\tau^{l}_n = \min_{K\in\mathcal{K}_h}\left(\tau_1, \tau_2, \tau_{max}\right),
		\end{align}
		where 
		$$\tau_1 := \frac{\left(\phi^{n+1,l}_{h} c^{n}_{h}\left(1-\beta c^{n}_{h}\right)^2\delta_1 - (\phi^{n+1,l}_{h}-\phi^{n}_{h})c^{n}_{h}\right)|K|}{\sum\limits_{e\in \partial K_{\boldsymbol{u}_{f,h}}^{+}} c^{n}_h\boldsymbol{u}_{f,h}^{n+1, l}\cdot\boldsymbol{n}|e| + \sum\limits_{e\in \partial K_{\mu}^{+}}\frac{\varsigma_1}{h_e}[\mu(c^{n}_{h})]|e|+ \epsilon},$$\\
		$$\tau_2 := \frac{\left(\phi^{n+1,l}_{h} c^{n}_{h}\left(1-\beta c^{n}_{h}\right)^2\delta_2 + (\phi^{n+1,l}_{h}-\phi^{n}_{h})c^{n}_{h}\right)|K|}{-\left(\sum\limits_{e\in \partial K_{\boldsymbol{u}_{f,h}}^{-}} c^{n*}_h\boldsymbol{u}_{f,h}^{n+1, l}\cdot\boldsymbol{n}|e| + \sum\limits_{e\in \partial K_{\mu}^{-}}\frac{\varsigma_1}{h_e}[\mu(c^{n}_{h})]|e|\right)+ \epsilon},$$
		$\epsilon > 0$ is a very small constant to avoid zero denominator, $\tau_{max}$ is the allowed maximum time step size to
		guarantee the accuracy of numerical solutions and $\partial K_{\boldsymbol{u}_{f,h}}^{+} = \{e \in \partial K: \boldsymbol{u}_{f,h}^{n+1}\cdot\boldsymbol{n}|_e > 0, \forall K \in \mathcal{K}_{h}\}$, $\partial K_{\boldsymbol{u}_{f,h}}^{-} = \{e\in\partial K: \boldsymbol{u}_{f,h}^{n+1}\cdot\boldsymbol{n}|_e < 0, \forall K \in \mathcal{K}_{h}\}$, $\partial K_{\mu}^{-} = \{e\in\partial K: [\mu(c^{n}_{h})] < 0, \forall K \in \mathcal{K}_{h}\}, \partial K_{\mu}^{+} = \{e\in\partial K: [\mu(c^{n}_{h})] > 0, \forall K \in \mathcal{K}_{h}\}$. Then $c^{n+1, l+1}_h$ satisfies
		$$
		0<(1 - \delta_1\left(1-\beta c^{n}_{h}\right)^{2})c^{n}_{h}\leq c_{h}^{n+1,l+1} \leq (1 + \delta_2\left(1 - \beta c^{n}_{h}\right)^{2})c^{n}_{h}<\frac{1}{\beta}.
		$$
	\end{lemma}
	\begin{proof} 
		
		Let us define $\zeta_{h, -}^{n+1, l+1} = \min(\zeta_{h}^{n+1, l+1} + \delta_1, 0)$ and $\zeta_{h, +}^{n+1, l+1} = \max(\zeta_{h}^{n+1, l+1} - \delta_2, 0)$. Obviously $\zeta_{h, -}^{n+1, l+1} \leq 0 $ and $\zeta_{h, +}^{n+1, l+1} \geq 0$.
		
		Taking $q_h = \zeta_{h, -}^{n+1,l+1}$ in \eqref{eq-3-2-2} and using \eqref{eq-stable}, we can obtain
		\begin{align}\label{eq-616-1}
			&\frac{1}{\tau^{l}_n}\left(\phi^{n+1,l}_{h}(c_h^{n+1,l+1}-c^{n}_{h}), \zeta_{h, -}^{n+1, l+1}\right) + \frac{1}{\tau^{l}_n}\left((\phi^{n+1,l}_{h}-\phi^{n}_{h})c^{n}_{h}, \zeta_{h, -}^{n+1, l+1}\right)\\\nonumber
			&\quad+ \sum_{e\in \mathcal{E}^{I}_h} \langle c^{n*}_h\boldsymbol{u}_{f,h}^{n+1, l}\cdot\boldsymbol{n}, [\zeta_{h, -}^{n+1, l+1}]\rangle_{e} + \sum_{e\in \mathcal{E}^{I}_h}\frac{\varsigma_1}{h_e}\langle[\mu(c^{n}_{h})], [\zeta_{h, -}^{n+1, l+1}]\rangle_e\\\nonumber
			&\quad + \theta_n R T \sum_{e\in \mathcal{E}^{I}_h}\frac{\varsigma_1}{h_e}\langle[\zeta_{h}^{n+1, l+1}], [\zeta_{h, -}^{n+1, l+1}]\rangle_e = 0.
		\end{align}
		In terms of the definition of $\zeta_{h, -}^{n+1, l+1}$, we deduce that
		\begin{align}\label{eq-616-2}
			\left(\phi^{n+1,l}_{h}(c_h^{n+1,l+1}-c^{n}_{h}), \zeta_{h, -}^{n+1, l+1}\right) = & \left(\phi^{n+1,l}_{h} c^{n}_{h}\left(1-\beta c^{n}_{h}\right)^{2} \zeta_{h}^{n+1,l+1},  \zeta_{h, -}^{n+1, l+1}\right) \\\nonumber
			= & \left(\phi^{n+1,l}_{h} c^{n}_{h}\left(1-\beta c^{n}_{h}\right)^{2}\left(\zeta_{h}^{n+1,l+1}+\delta_1\right), \zeta_{h, -}^{n+1, l+1}\right)\\\nonumber
			& -\left(\phi^{n+1,l}_{h} c^{n}_{h}\left(1-\beta c^{n}_{h}\right)^{2} \delta_1, \zeta_{h, -}^{n+1, l+1}\right)\\\nonumber
			= & \left(\phi^{n+1,l}_{h} c^{n}_{h}\left(1-\beta c^{n}_{h}\right)^{2}  \zeta_{h, -}^{n+1, l+1}, \zeta_{h, -}^{n+1, l+1}\right)\\\nonumber
			& - \left(\phi^{n+1,l}_{h} c^{n}_{h}\left(1-\beta c^{n}_{h}\right)^{2}  \delta_1,  \zeta_{h, -}^{n+1, l+1}\right).
		\end{align}
		Taking into account \eqref{eq-tau}, $\zeta_{h, -}^{n+1,l+1} \leq 0$, and Lemma \ref{le-1},  we can get  
		
		\begin{align}\label{eq-616-3}
			&  \sum_{K\in \mathcal{K}_h }\sum_{e\in \partial K }\langle c^{n*}_h\boldsymbol{u}_{f,h}^{n+1, l}\cdot\boldsymbol{n}, \zeta_{h, -}^{n+1, l+1}\rangle_{e} -\frac{1}{\tau^{l}_n}\sum_{K\in \mathcal{K}_h }\left(\phi^{n+1, l}_{h} c^{n}_{h}\left(1-\beta c^{n}_{h}\right)^{2} \delta_1, \zeta_{h, -}^{n+1, l+1}\right) \\\nonumber
			&\quad + \frac{1}{\tau^{l}_n}\sum_{K\in \mathcal{K}_h }\left((\phi^{n+1,l}_{h}-\phi^{n}_{h})c^{n}_{h}, \zeta_{h, -}^{n+1, l+1}\right) + \sum_{K\in \mathcal{K}_h }\sum_{e\in \partial K }\frac{\varsigma_1}{h_e}\langle[\mu(c^{n}_{h})], \zeta_{h, -}^{n+1, l+1}\rangle_e\\\nonumber
			& =\frac{1}{\tau^{l}_n}\sum_{K\in \mathcal{K}_h }\left((\phi^{n+1,l}_{h}-\phi^{n}_{h})c^{n}_{h}, \zeta_{h, -}^{n+1, l+1}\right) - \frac{1}{\tau^{l}_n}\sum_{K\in \mathcal{K}_h }\left(\phi^{n+1, l}_{h} c^{n}_{h}\left(1-\beta c^{n}_{h}\right)^{2} \delta_1, \zeta_{h, -}^{n+1, l+1}\right) \\\nonumber
			&\quad+ \sum_{K\in \mathcal{K}_h }\sum_{e\in \partial K_{\boldsymbol{u}_{f,h}}^{-}}\langle c^{n*}_h\boldsymbol{u}_{f,h}^{n+1, l}\cdot\boldsymbol{n}, \zeta_{h, -}^{n+1, l+1}\rangle_{e} + \sum_{K\in \mathcal{K}_h }\sum_{e\in \partial K_{\boldsymbol{u}_{f,h}}^{+}}\langle c^{n*}_h\boldsymbol{u}_{f,h}^{n+1, l}\cdot\boldsymbol{n}, \zeta_{h, -}^{n+1, l+1}\rangle_{e}\\\nonumber
			&\quad + \sum_{K\in \mathcal{K}_h }\sum_{e\in \partial K_{\mu}^{-}}\frac{\varsigma_1}{h_e}\langle[\mu(c^{n}_{h})], \zeta_{h, -}^{n+1, l+1}\rangle_e + \sum_{K\in \mathcal{K}_h }\sum_{e\in \partial K_{\mu}^{+}}\frac{\varsigma_1}{h_e}\langle[\mu(c^{n}_{h})], \zeta_{h, -}^{n+1, l+1}\rangle_e\\\nonumber
			& =\sum_{K\in \mathcal{K}_h }\frac{\left((\phi^{n+1,l}_{h}-\phi^{n}_{h})c^{n}_{h}-\phi^{n+1,l}_{h} c^{n}_{h}\left(1-\beta c^{n}_{h}\right)^2\delta_1\right) \zeta_{h, -}^{n+1, l+1}|K|}{\tau^{l}_n} \\\nonumber
			&\quad +  \sum_{K\in \mathcal{K}_h }\sum_{e\in \partial K_{\boldsymbol{u}_{f,h}}^{-}}c^{n*}_h\boldsymbol{u}_{f,h}^{n+1, l}\cdot\boldsymbol{n} \zeta_{h, -}^{n+1, l+1}|e| +  \sum_{K\in \mathcal{K}_h }\sum_{e\in \partial K_{\boldsymbol{u}_{f,h}}^{+}} c^{n}_h\boldsymbol{u}_{f,h}^{n+1, l}\cdot\boldsymbol{n}\zeta_{h, -}^{n+1, l+1}|e|\\\nonumber
			&\quad +  \sum_{K\in \mathcal{K}_h }\sum_{e\in \partial K_{\mu}^{-}}\frac{\varsigma_1}{h_e}[\mu(c^{n}_{h})]\zeta_{h, -}^{n+1, l+1}|e| +  \sum_{K\in \mathcal{K}_h }\sum_{e\in \partial K_{\mu}^{+}}\frac{\varsigma_1}{h_e}[\mu(c^{n}_{h})]\zeta_{h, -}^{n+1, l+1}|e|\geq 0.
		\end{align}
		
		Here we assume that $\phi^{n+1,l}_{h}\left(1-\beta c^{n}_{h}\right)^2\delta_1 > (\phi^{n+1,l}_{h}-\phi^{n}_{h})$, which means 
		\begin{align}\label{eq-1012-2}
			\frac{\phi^{n+1,l}_{h}}{\phi^{n}_{h}}<\frac{1}{1-\left(1-\beta c^{n}_{h}\right)^2\delta_1}.
		\end{align}
		Combining \eqref{eq-616-1}-\eqref{eq-616-3}, we get
		\begin{align}
			\frac{1}{\tau^{l}_n}\left(\phi^{n+1,l}_{h} c^{n}_{h}\left(1-\beta c^{n}_{h}\right)^{2}  \zeta_{h, -}^{n+1, l+1}, \zeta_{h, -}^{n+1, l+1}\right) + \theta_n R T\sum_{e\in \mathcal{E}^{I}_h}\frac{\varsigma_1}{h_e}\langle[\zeta_{h, -}^{n+1, l+1}], [\zeta_{h, -}^{n+1, l+1}]\rangle_e \leq 0.
		\end{align}
		Due to $0< c^{n}_{h} < \frac{1}{\beta} $, we obtain
		\begin{align}
			\zeta_{h, -}^{n+1, l+1} \equiv 0,
		\end{align}
		and
		\begin{align}
			c_h^{n+1, l+1} \geq (1 - \delta_1\left(1-\beta c^{n}_{h}\right)^{2})c^{n}_{h}>0.
		\end{align}
		Taking $q_h = \zeta_{h, +}^{n+1,l+1}$ in \eqref{eq-3-2-2}, we can obtain
		\begin{align}
			\frac{1}{\tau^{l}_n}\left(\phi^{n+1,l}_{h}(c_h^{n+1,l+1}-c^{n}_{h}), \zeta_{h, +}^{n+1,l+1}\right) + \frac{1}{\tau^{l}_n}\left((\phi^{n+1,l}_{h}-\phi^{n}_{h})c^{n}_{h}, \zeta_{h, +}^{n+1,l+1}\right)\\\nonumber
			+ \sum_{e\in \mathcal{E}^{I}_h} \langle c^{n*}_h\boldsymbol{u}_{f,h}^{n+1, l}\cdot\boldsymbol{n}, \zeta_{h, +}^{n+1,l+1}\rangle_{e} + \sum_{e\in \mathcal{E}^{I}_h}\frac{\varsigma_1}{h_e}\langle[\mu^{n+1}_{h}], [\zeta_{h, +}^{n+1,l+1}]\rangle_e= 0.
		\end{align}
		In terms of the definition of $\zeta_{h, +}^{n+1, l+1}$, we deduce that
		\begin{align}\label{eq-617-1}
			\left(\phi^{n+1,l}_{h}(c_h^{n+1,l+1}-c^{n}_{h}), \zeta_{h, +}^{n+1, l+1}\right) = & \left(\phi^{n+1,l}_{h} c^{n}_{h}\left(1-\beta c^{n}_{h}\right)^{2} \zeta_{h}^{n+1,l+1},  \zeta_{h, +}^{n+1, l+1}\right) \\\nonumber
			= & \left(\phi^{n+1,l}_{h} c^{n}_{h}\left(1-\beta c^{n}_{h}\right)^{2}\left(\zeta_{h}^{n+1,l+1}- \delta_2\right), \zeta_{h, +}^{n+1, l+1}\right)\\\nonumber
			& + \left(\phi^{n+1,l}_{h} c^{n}_{h}\left(1-\beta c^{n}_{h}\right)^{2} \delta_2, \zeta_{h, +}^{n+1, l+1}\right)\\\nonumber
			= & \left(\phi^{n+1,l}_{h} c^{n}_{h}\left(1-\beta c^{n}_{h}\right)^{2}  \zeta_{h, +}^{n+1, l+1}, \zeta_{h, +}^{n+1, l+1}\right)\\\nonumber
			& + \left(\phi^{n+1,l}_{h} c^{n}_{h}\left(1-\beta c^{n}_{h}\right)^{2}  \delta_2,  \zeta_{h, +}^{n+1, l+1}\right).
		\end{align}
		By \eqref{eq-tau}, $\zeta_{h, +}^{n+1, l+1} \geq 0$ and Lemma \ref{le-1}, we get
		\begin{align}\label{eq-617-2}
			&  \sum_{K\in \mathcal{K}_h }\sum_{e\in \partial K }\langle c^{n*}_h\boldsymbol{u}_{f,h}^{n+1, l}\cdot\boldsymbol{n}, \zeta_{h, +}^{n+1, l+1}\rangle_{e} + \frac{1}{\tau^{l}_n}\sum_{K\in \mathcal{K}_h }\left(\phi^{n+1, l}_{h} c^{n}_{h}\left(1-\beta c^{n}_{h}\right)^{2} \delta_2, \zeta_{h, +}^{n+1, l+1}\right) \\\nonumber
			&\quad + \frac{1}{\tau^{l}_n}\sum_{K\in \mathcal{K}_h }\left((\phi^{n+1,l}_{h}-\phi^{n}_{h})c^{n}_{h}, \zeta_{h, +}^{n+1, l+1}\right) + \sum_{K\in \mathcal{K}_h }\sum_{e\in \partial K }\frac{\varsigma_1}{h_e}\langle[\mu(c^{n}_{h})], \zeta_{h, +}^{n+1, l+1}\rangle_e\\\nonumber
			& =\frac{1}{\tau^{l}_n}\sum_{K\in \mathcal{K}_h }\left((\phi^{n+1,l}_{h}-\phi^{n}_{h})c^{n}_{h}, \zeta_{h, +}^{n+1, l+1}\right) + \frac{1}{\tau^{l}_n}\sum_{K\in \mathcal{K}_h }\left(\phi^{n+1, l}_{h} c^{n}_{h}\left(1-\beta c^{n}_{h}\right)^{2} \delta_2, \zeta_{h, +}^{n+1, l+1}\right) \\\nonumber
			&\quad+ \sum_{K\in \mathcal{K}_h }\sum_{e\in \partial K_{\boldsymbol{u}_{f,h}}^{-}}\langle c^{n*}_h\boldsymbol{u}_{f,h}^{n+1, l}\cdot\boldsymbol{n}, \zeta_{h, +}^{n+1, l+1}\rangle_{e} + \sum_{K\in \mathcal{K}_h }\sum_{e\in \partial K_{\boldsymbol{u}_{f,h}}^{+}}\langle 	c^{n*}_h\boldsymbol{u}_{f,h}^{n+1, l}\cdot\boldsymbol{n}, \zeta_{h, +}^{n+1, l+1}\rangle_{e}\\\nonumber
			&\quad + \sum_{K\in \mathcal{K}_h }\sum_{e\in \partial K_{\mu}^{-}}\frac{\varsigma_1}{h_e}\langle[\mu(c^{n}_{h})], \zeta_{h, +}^{n+1, l+1}\rangle_e + \sum_{K\in \mathcal{K}_h }\sum_{e\in \partial K_{\mu}^{+}}\frac{\varsigma_1}{h_e}\langle[\mu(c^{n}_{h})], \zeta_{h, +}^{n+1, l+1}\rangle_e\\\nonumber
			& =\sum_{K\in \mathcal{K}_h }\frac{\left((\phi^{n+1,l}_{h}-\phi^{n}_{h})c^{n}_{h}+\phi^{n+1,l}_{h} c^{n}_{h}\left(1-\beta c^{n}_{h}\right)^{2}\delta_2\right) \zeta_{h, +}^{n+1, l+1}|K|}{\tau^{l}_n} \\\nonumber
			&\quad +  \sum_{K\in \mathcal{K}_h }\sum_{e\in \partial K_{\boldsymbol{u}_{f,h}}^{-}}c^{n*}_h\boldsymbol{u}_{f,h}^{n+1, l}\cdot\boldsymbol{n} \zeta_{h, +}^{n+1, l+1}|e| +  \sum_{K\in \mathcal{K}_h }\sum_{e\in \partial K_{\boldsymbol{u}_{f,h}}^{+}} c^{n}_h\boldsymbol{u}_{f,h}^{n+1, l}\cdot\boldsymbol{n}\zeta_{h, +}^{n+1, l+1}|e|\\\nonumber
			&\quad +  \sum_{K\in \mathcal{K}_h }\sum_{e\in \partial K_{\mu}^{-}}\frac{\varsigma_1}{h_e}[\mu(c^{n}_{h})]\zeta_{h, +}^{n+1, l+1}|e| +  \sum_{K\in \mathcal{K}_h }\sum_{e\in \partial K_{\mu}^{+}}\frac{\varsigma_1}{h_e}[\mu(c^{n}_{h})]\zeta_{h, +}^{n+1, l+1}|e|\geq 0.
		\end{align}
		Here we need $\phi^{n+1,l}_{h}\left(1-\beta c^{n}_{h}\right)^2\delta_2 > (\phi^{n}_{h} - \phi^{n+1,l}_{h})$, which means 
		\begin{align}\label{eq-1012-3}
			\frac{\phi^{n+1,l}_{h}}{\phi^{n}_{h}} > \frac{1}{(1 + \left(1-\beta c^{n}_{h}\right)^2\delta_2)}.
		\end{align}
		
		Due to  \eqref{eq-617-1}-\eqref{eq-617-2}, we can obtain
		\begin{align}
			\left(\phi^{n+1,l}_{h} c^{n}_{h}\left(1-\beta c^{n}_{h}\right)^{2}\zeta_{h, +}^{n+1, l+1},\zeta_{h, +}^{n+1, l+1}\right) + \theta_n R T\sum_{e\in \mathcal{E}^{I}_h}\langle[\zeta_{h, +}^{n+1, l+1}], [\zeta_{h, +}^{n+1, l+1}]\rangle_e\leq 0,
		\end{align}	 
		which directly yields
		$$
		\zeta_{h, +}^{n+1, l+1} = 0.  
		$$
		Then, we can get
		\begin{align}
			c_h^{n+1, l+1} \leq (1 + \delta_2\left(1-\beta c^{n}_{h}\right)^{2})c^{n}_{h}.
		\end{align}
		By $0 < c^{n}_{h} < \frac{1}{\beta}$ and $0 < \delta_2 < 1$, we deduce that
		\begin{align}
			\beta c_h^{n+1, l+1}& = (1 + \delta_2\left(1-\beta c^{n}_{h}\right)^{2})\beta c^{n}_{h}\\\nonumber
			& = \beta c^{n}_{h}  + \delta_2\left(1-\beta c^{n}_{h}\right)^{2}\beta c^{n}_{h}\\\nonumber  
			& < \beta c^{n}_{h} + 1 - \beta c^{n}_{h} = 1.
		\end{align}
		The proof is completed. 
	\end{proof}
	\begin{rem}
		In the above proof, we use the following inequality
		$$\frac{1}{(1 + \left(1-\beta c^{n}_{h}\right)^2\delta_2)}< \frac{\phi^{n+1,l}_{h}}{\phi^{n}_{h}} <\frac{1}{1-\left(1-\beta c^{n}_{h}\right)^2\delta_1}, $$ 
		when $l = 0$, we can get $\phi^{n+1,0} = \phi^{n}$, the above assumption is obviously true. Next we will prove that when $0 \leq c^{n+1,l}_{h} \leq 1$ is established, the appeal assumption holds.
	\end{rem}
	\begin{theorem}
		According to Lemma~\ref{theo-1}, we assume that the inequality $0 < \varrho_0 \leq \beta c^n_h \leq \varrho < 1$ holds, where $\varrho_0$ and $\varrho$ are positive constants. Furthermore, let the parameters $\gamma$ and $\varsigma_1$ be chosen sufficiently large. Under these assumptions, the function $\phi_{h}^{n+1, l+1}$ is bounded by
		$$
		0 < \phi_{h}^{n} - C_{\epsilon} \leq \phi_{ h}^{n+1, l+1} \leq \phi_{h}^{n} + C_{\epsilon} < 1,
		$$
		where the parameter $C_{\epsilon}$ satisfies the condition
		\begin{align}
			C_{\epsilon} < \min\left(\frac{\left(1-\varrho\right)^2\delta_2C_{\phi, min}}{(1 + \left(1-\varrho\right)^2\delta_2)}, \frac{\left(1-\varrho\right)^2\delta_1C_{\phi, min}}{(1 - \left(1-\varrho\right)^2\delta_1)},  1-C_{\phi, max}\right).
		\end{align}
		Here, $C_{\phi, min}$ and $C_{\phi, max}$  are the maximum and minimum values of $\phi^{n}$ respectively.
	\end{theorem}
	\begin{proof}
	We begin by considering the difference:
	$\mathcal{A}(\mathbf{u}^{n+1,l+1}_{s,h}, p_{h}^{n+1, l+1}, \mathbf{v}_h) - \mathcal{A}(\mathbf{u}^{n}_{s,h}, p_{h}^{n}, \mathbf{v}_h)$ = 0.
	Then, by selecting $\mathbf{v}_h = \mathbf{u}_{s,h}^{n+1,l+1} - \mathbf{u}_{s,h}^{n}$ as the test function, we obtain:
		\begin{align}\label{eq-20251110-0}
			&2\eta\sum_{K \in \mathcal{K}_h} \|\varepsilon(\boldsymbol{u}_{s,h}^{n+1,l+1} - \boldsymbol{u}_{s,h}^{n}) \|^{2}_{L^2(K)} + \gamma\sum_{K \in \mathcal{K}_h} \|\nabla\cdot (\boldsymbol{u}_{s,h}^{n+1,l+1} - \boldsymbol{u}_{s,h}^{n})\|^{2}_{L^2(K)}\\\nonumber
			&\quad+\sum_{e \in \mathcal{E}^{I}_h}\frac{\varsigma_2}{h_e}\|[\boldsymbol{u}_{s,h}^{n+1,l+1} - \boldsymbol{u}_{s,h}^{n}]\|^{2}_{L^2(e)}
			-2\sum_{e \in \mathcal{E}^{I}_h}\langle\{\mathbf{\sigma}_e(\boldsymbol{u}_{s,h}^{n+1,l+1} - \boldsymbol{u}_{s,h}^{n}) \boldsymbol{n}_e\}, [\boldsymbol{u}_{s,h}^{n+1,l+1} - \boldsymbol{u}_{s,h}^{n}]\rangle_e 
			\\\nonumber
			&\quad- \alpha\sum_{K \in \mathcal{K}_h}(p_{h}^{n+1,l+1} - p_{h}^{n}, \nabla\cdot (\boldsymbol{u}_{s,h}^{n+1,l+1} - \boldsymbol{u}_{s,h}^{n}))_{K}\\\nonumber
			&\quad+\alpha\sum_{e \in\mathcal{E}^{I}_h}\langle\{(p_{h}^{n+1,l+1} - p_{h}^{n})\boldsymbol{n}_e\}, [\boldsymbol{u}_{s,h}^{n+1,l+1} - \boldsymbol{u}_{s,h}^{n}]\rangle_e \\\nonumber
			&= 0, 
	\end{align}
    Due to the trace inequality, Cauchy-Schwarz inequality and the Young's inequality, we get
	\begin{align}\label{eq-20251110-1}
			&2\sum_{e \in \mathcal{E}^{I}_h}\langle\{\mathbf{\sigma}_e(\boldsymbol{u}_{s,h}^{n+1,l+1} - \boldsymbol{u}_{s,h}^{n}) \boldsymbol{n}_e\}, [\boldsymbol{u}_{s,h}^{n+1,l+1} - \boldsymbol{u}_{s,h}^{n}]\rangle_e\\\nonumber
			&\leq 2\sum_{e \in \mathcal{E}^{I}_h}\|\{\mathbf{\sigma}_e(\boldsymbol{u}_{s,h}^{n+1,l+1} - \boldsymbol{u}_{s,h}^{n}) \boldsymbol{n}_e\}\|_{L^{2}(e)}\|[\boldsymbol{u}_{s,h}^{n+1,l+1} - \boldsymbol{u}_{s,h}^{n}]\|_{L^{2}(e)}\\\nonumber
			&\leq 2(\sum_{e \in \mathcal{E}^{I}_h}\|\{\mathbf{\sigma}_e(\boldsymbol{u}_{s,h}^{n+1,l+1} - \boldsymbol{u}_{s,h}^{n}) \boldsymbol{n}_e\}\|^{2}_{L^{2}(e)})^{\frac{1}{2}}(\sum_{e \in \mathcal{E}^{I}_h}\|[\boldsymbol{u}_{s,h}^{n+1,l+1} - \boldsymbol{u}_{s,h}^{n}]\|^{2}_{L^{2}(e)})^{\frac{1}{2}}\\\nonumber 
			&\leq 2(\frac{C}{2h_e}\sum_{K \in \mathcal{K}_h}(\mathbf{\sigma}_e(\boldsymbol{u}_{s,h}^{n+1,l+1} - \boldsymbol{u}_{s,h}^{n}), \epsilon(\boldsymbol{u}_{s,h}^{n+1,l+1} - \boldsymbol{u}_{s,h}^{n}))_{K})^{\frac{1}{2}}(\sum_{e \in \mathcal{E}^{I}_h}\|[\boldsymbol{u}_{s,h}^{n+1,l+1} - \boldsymbol{u}_{s,h}^{n}]\|^{2}_{L^{2}(e)})^{\frac{1}{2}}\\\nonumber
			&\leq \frac{1}{2}\sum_{K \in \mathcal{K}_h}(\mathbf{\sigma}_e(\boldsymbol{u}_{s,h}^{n+1,l+1} - \boldsymbol{u}_{s,h}^{n}), \epsilon(\boldsymbol{u}_{s,h}^{n+1,l+1} - \boldsymbol{u}_{s,h}^{n}))_{K} + \frac{C}{h_e}\sum_{e \in \mathcal{E}^{I}_h}\|[\boldsymbol{u}_{s,h}^{n+1,l+1} - \boldsymbol{u}_{s,h}^{n}]\|^{2}_{L^{2}(e)}, 
	\end{align}
    \begin{align}\label{eq-20251110-2}
	&\alpha\sum_{e \in\mathcal{E}^{I}_h}\langle\{(p_{h}^{n+1,l+1} - p_{h}^{n})\boldsymbol{n}_e\}, [\boldsymbol{u}_{s,h}^{n+1,l+1} - \boldsymbol{u}_{s,h}^{n}]\rangle_e\\\nonumber
	&\leq \alpha\sum_{e \in\mathcal{E}^{I}_h}\|\{(p_{h}^{n+1,l+1} - p_{h}^{n})\boldsymbol{n}_e\}\|_{L^{2}(e)}\|[\boldsymbol{u}_{s,h}^{n+1,l+1} - \boldsymbol{u}_{s,h}^{n}]\|_{L^{2}(e)}\\\nonumber
	&\leq \alpha(\sum_{e \in\mathcal{E}^{I}_h}\|\{(p_{h}^{n+1,l+1} - p_{h}^{n})\boldsymbol{n}_e\}\|^{2}_{L^{2}(e)})^{\frac{1}{2}}(\sum_{e \in \mathcal{E}^{I}_h}\|[\boldsymbol{u}_{s,h}^{n+1,l+1} - \boldsymbol{u}_{s,h}^{n}]\|^{2}_{L^{2}(e)})^{\frac{1}{2}}\\\nonumber
	&\leq \alpha(\frac{C}{2h_e}\sum_{K \in \mathcal{K}_h}\|p_{h}^{n+1,l+1} - p_{h}^{n}\|^{2}_{L^{2}(K)})^{\frac{1}{2}}(\sum_{e \in \mathcal{E}^{I}_h}\|[\boldsymbol{u}_{s,h}^{n+1,l+1} - \boldsymbol{u}_{s,h}^{n}]\|^{2}_{L^{2}(e)})^{\frac{1}{2}}\\\nonumber
	&\leq \frac{\alpha}{2}\sum_{K \in \mathcal{K}_h}\|p_{h}^{n+1,l+1} - p_{h}^{n}\|^{2}_{L^{2}(K)} + \frac{\alpha C}{h_e}\sum_{e \in \mathcal{E}^{I}_h}\|[\boldsymbol{u}_{s,h}^{n+1,l+1} - \boldsymbol{u}_{s,h}^{n}]\|^{2}_{L^{2}(e)},
\end{align}
and
\begin{align}\label{eq-20251110-3}
	&\alpha\sum_{K \in \mathcal{K}_h}(p_{h}^{n+1,l+1} - p_{h}^{n}, \nabla\cdot (\boldsymbol{u}_{s,h}^{n+1,l+1} - \boldsymbol{u}_{s,h}^{n}))_{K}\\\nonumber
	&\leq \frac{\alpha}{2}\sum_{K \in \mathcal{K}_h}\|p_{h}^{n+1,l+1} - p_{h}^{n}\|^{2}_{L^{2}(K)} + \frac{\alpha}{2}\sum_{K \in \mathcal{K}_h}\|\nabla\cdot (\boldsymbol{u}_{s,h}^{n+1,l+1} - \boldsymbol{u}_{s,h}^{n})\|^{2}_{L^{2}(K)}.
\end{align}
		Here we have used the following trace inequality (cf. \cite{Schwab1998}).
		\begin{align}
			\left\|\mathbf{\sigma}_e(\mathbf{v}_h) \boldsymbol{n}\right\|_{L^2\left(e\right)}^2 \leq C h_e^{-1} \int_{K}(\mathbf{\sigma}_e(\mathbf{v}_h): \varepsilon(\mathbf{v}_h)) d \boldsymbol{x},\quad  \forall \mathbf{v}_h \in \mathbb{P}_r\left(K \right)^d,
		\end{align}
		where $ \mathbb{P}_r$ is the $r$th-order polynomial space.
		
		By substituting formulas \eqref{eq-20251110-1}, \eqref{eq-20251110-2}, and \eqref{eq-20251110-3} into \eqref{eq-20251110-0}, we obtain
		\begin{align}
			&\eta\sum_{K \in \mathcal{K}_h} \|\varepsilon(\boldsymbol{u}_{s,h}^{n+1,l+1} - \boldsymbol{u}_{s,h}^{n})\|^{2}_{L^2(K)} + \frac{\gamma -\alpha}{2}\sum_{K \in \mathcal{K}_h} \|\nabla\cdot (\boldsymbol{u}_{s,h}^{n+1,l+1} - \boldsymbol{u}_{s,h}^{n})\|^{2}_{L^2(K)}\label{eq-615-1} \\\nonumber
			&\quad+(\varsigma_2-(1+\alpha)C)\sum_{e \in \mathcal{E}^{I}_h}\frac{1}{h_e}\|[\boldsymbol{u}_{s,h}^{n+1,l+1} - \boldsymbol{u}_{s,h}^{n}]\|^{2}_{L^2(e)}\\\nonumber
			&\leq \alpha\sum_{K \in \mathcal{K}_h}\|p_{h}^{n+1,l+1} - p_{h}^{n}\|^{2}_{L^2(K)}.
		\end{align}
		
		Due to \eqref{eq-615-1}, we get
		\begin{align}
			\sum_{K \in \mathcal{K}_h} \|\nabla\cdot \boldsymbol{u}_{s,h}^{n+1,l+1}-\boldsymbol{u}_{s,h}^{n}\|^{2}_{L^2(K)}\leq \frac{2\alpha}{\gamma-\alpha}\sum_{K \in \mathcal{K}_h}\|p_{h}^{n+1,l+1}- p_{h}^{n}\|^{2}_{L^2(K)},\label{eq-615-2}\\
			\sum_{e \in \mathcal{E}^{I}_h}\frac{1}{h_e}\|[\boldsymbol{u}_{s,h}^{n+1,l+1} - \boldsymbol{u}_{s,h}^{n}]\|^{2}_{L^2(e)}\leq	\frac{\alpha}{\varsigma_2-(1+\alpha)C}\sum_{K \in \mathcal{K}_h}\|p_{h}^{n+1,l+1} - p_{h}^{n}\|^{2}_{L^2(K)}.\label{eq-615-3}
		\end{align}
    In equation \eqref{eq-3-2-4}, we take the test function $q_h$ to be 1 on the element $K$ and 0 on all other elements, we get
		\begin{align}\label{eq-p-1}
			p_h^{n+1,l+1} &= c_h^{n}\mu_h^{n+1,l+1} - f(c_h^{n}) \\\nonumber
			&= \frac{RTc_h^{n+1,l+1}}{1-\beta c_h^{n}}+\frac{bc_h^{n}}{2\sqrt{2}\beta}\left(\frac{(1-\sqrt{2})\beta c_h^{n+1,l+1}}{1+(1-\sqrt{2})\beta c_h^{n}}-\frac{(1+\sqrt{2})\beta c_h^{n}}{1+(1+\sqrt{2})\beta c_h^{n}}\right).
		\end{align} 
		Due to \eqref{eq-p-1} and $0<c_h^n$, $c_h^{n+1,l+1}<\frac{1}{ \beta}$, we get
		\begin{align}\label{eq-p-max}
			\|p_h^{n+1,l+1}\|_{\infty} &\leq \frac{RT}{1-\varrho} + \frac{b}{2\sqrt{2}}\left(\frac{(\sqrt{2}-1)}{1 + (1-\sqrt{2})} + (1+\sqrt{2})\right)\\\nonumber
			&= \frac{RT}{1-\varrho} + \frac{(2\sqrt{2}-1)b}{4\sqrt{2} - 1}:= C_p,
		\end{align}
	where $\|\cdot\|_{\infty} = \max\limits_{K \in \mathcal{K}_h}|\cdot| $.
	
		Letting $\varphi_h = \phi_h^{n+1,l+1}-\phi_h^{n}$, we obtain 
		\begin{align}\label{eq-1017-1}
			\sum_{K \in \mathcal{K}_h}\|\phi_h^{n+1,l+1}-\phi_h^{n}\|^{2}_{L^2(K)}& = \frac{1}{N}\sum_{K \in \mathcal{K}_h}(p_h^{n+1,l+1} - p_h^{n}, \phi_h^{n+1,l+1}-\phi_h^{n})_{K}\\\nonumber 
			&\quad + \alpha \sum_{K \in \mathcal{K}_h}(\nabla\cdot \textbf{u}_{s,h}^{n+1,l+1}-\nabla\cdot \textbf{u}_{s,h}^{n}, \phi_h^{n+1,l+1}-\phi_h^{n})_{K} \\\nonumber
			&\quad -\alpha\sum_{e \in \partial K}\langle[\textbf{u}^{n+1,l+1}_{s,h}], \{\phi_h^{n+1,l+1}-\phi_h^{n}\}\mathbf{n}_e \rangle_e  \\\nonumber
			&\quad + \alpha\sum_{e \in \partial K}\langle [\textbf{u}^{n}_{s,h}], \{\phi_h^{n+1,l+1}-\phi_h^{n}\}\mathbf{n}_e \rangle_e.
		\end{align}
		By \eqref{eq-1017-1}, we get
		\begin{align}
			&\sum_{K \in \mathcal{K}_h}\|\phi_h^{n+1,l+1}-\phi_h^{n}\|^{2}_{L^2(K)} \\\nonumber
			&\leq \frac{1}{N}\left(\sum_{K \in \mathcal{K}_h}\|p_{h}^{n+1,l+1} - p_{h}^{n}\|^{2}_{L^2(K)}\right)^{\frac{1}{2}}\left(\sum_{K \in \mathcal{K}_h}\|\phi_h^{n+1,l+1}-\phi_h^{n}\|^{2}_{L^2(K)}\right)^{\frac{1}{2}}\\\nonumber
			&\quad +\alpha \left(\sum_{K \in \mathcal{K}_h}\|\nabla\cdot \textbf{u}_{s,h}^{n+1,l+1} - \nabla\cdot \textbf{u}_{s,h}^{n}\|^{2}_{L^2(K)}\right)^{\frac{1}{2}}\left(\sum_{K \in \mathcal{K}_h}\|\phi_h^{n+1,l+1}-\phi_h^{n}\|^{2}_{L^2(K)}\right)^{\frac{1}{2}}\\\nonumber
			&\quad + \frac{\alpha C}{2}\left(\sum_{e \in \partial K}h^{-1}_{e}\|[\textbf{u}^{n+1,l+1}_{s,h} - \textbf{u}^{n}_{s,h}]\|^{2}_{L^2(e)}\right)^{\frac{1}{2}}\left(\sum_{K \in 	\mathcal{K}_h}\|\phi_h^{n+1,l+1}-\phi_h^{n}\|^{2}_{L^2(K)}\right)^{\frac{1}{2}}.
		\end{align}
		Eliminating the same terms on both sides of the above inequality, we further obtain
		\begin{align}\label{eq-1023-1}
			&\left(\sum_{K \in \mathcal{K}_h}\|\phi_h^{n+1,l+1}-\phi_h^{n}\|^{2}_{L^2(K)}\right)^{\frac{1}{2}}\\\nonumber
			&\leq  \frac{1}{N}\left(\sum_{K \in \mathcal{K}_h}\|p_{h}^{n+1,l+1} - p_{h}^{n}\|^{2}_{L^2(K)}\right)^{\frac{1}{2}} +\alpha \left(\sum_{K \in \mathcal{K}_h}\|\nabla\cdot (\textbf{u}_{s,h}^{n+1,l+1} - \textbf{u}_{s,h}^{n})\|^{2}_{L^2(K)}\right)^{\frac{1}{2}}\\\nonumber
			&\quad + \frac{\alpha C}{2}\left(\sum_{e \in \partial K}h^{-1}_{e}\|[\textbf{u}^{n+1,l+1}_{s,h} - \textbf{u}^{n}_{s,h}]\|^{2}_{L^2(e)}\right)^{\frac{1}{2}}.
		\end{align}
		Due to the inverse inequality, \eqref{eq-615-2},\eqref{eq-615-3}, \eqref{eq-p-max} and \eqref{eq-1023-1}, we can get
		\begin{align}\label{eq-1018-1}
			& \max_{K \in \mathcal{K}_h} \|\phi_h^{n+1,l+1}-\phi_h^{n}\|_{L^\infty(K)}=\|\phi_h^{n+1,l+1}-\phi_h^{n}\|_{L^\infty(K_0)} \\\nonumber
			&\leq  \frac{C}{h^{\frac{d}{2}}}\|\phi_h^{n+1,l+1}-\phi_h^{n}\|_{L^2(K_0)}\\\nonumber
			&\leq \frac{C C_{x}}{h^{\frac{d}{2} } \sqrt{C_N}}\left(\sum\limits_{K \in \mathcal{K}_h}\|\phi_h^{n+1,l+1}-\phi_h^{n}\|^{2}_{L^2(K)} \right)^{\frac{1}{2}}\\\nonumber
			&\leq \frac{C}{h^{\frac{d}{2}}}\frac{C_{x}}{\sqrt{C_N}}\left(\frac{2\alpha^{2}}{\gamma-\alpha} + \frac{\alpha C}{2} \frac{\alpha}{\varsigma_2-(1+\alpha)C}+\frac{1}{N}\right) \left(\sum_{K \in \mathcal{K}_h}\|p_{h}^{n+1,l+1} - p_{h}^{n}\|^{2}_{L^2(K)}\right)^{\frac{1}{2}} \\\nonumber
			&\leq C C_{x}\left(\frac{2\alpha^{2}}{\gamma-\alpha} + \frac{\alpha C}{2} \frac{\alpha}{\varsigma_2-(1+\alpha)C}+\frac{1}{N}\right)C_p,
		\end{align}
		where $K_0$ is the unit with the maximum value of $\|\phi_h^{n+1,l+1}-\phi_h^{n}\|$, $C_N$ is the number of elements of $\mathcal{K}_h$ and $C_{x} \leq C_N$.
		
		By \eqref{eq-1012-2} and \eqref{eq-1012-3}, we can define $C_{\epsilon}$  to satisfy
		
		\begin{align}
			C_{\epsilon} < \min\left(\left(1 - \frac{1}{(1 + \left(1-\beta c^{n}_{h}\right)^2\delta_2)}\right)\phi^{n}_{h}, \left(\frac{1}{(1 - \left(1-\beta c^{n}_{h}\right)^2\delta_1)}-1\right)\phi^{n}_{h},  1-\phi^{n}_{h}\right).
		\end{align}
		Due to the assumption $\varrho_0 \leq  \beta c^{n}_{h} \leq \varrho$, $C_{\phi, min} \leq  \phi^{n}_{h} \leq C_{\phi, max}$, we further get
		\begin{align}
			C_{\epsilon} < \min\left(\frac{\left(1-\varrho\right)^2\delta_2C_{\phi, min}}{(1 + \left(1-\varrho\right)^2\delta_2)}, \frac{\left(1-\varrho\right)^2\delta_1C_{\phi, min}}{(1 - \left(1-\varrho\right)^2\delta_1)},  1-C_{\phi, max}\right).
		\end{align}
		Due to the definition of $C_\epsilon$ and \eqref{eq-1018-1}, we can get that $\gamma$ and $\varsigma_2$ should be sufficiently large to ensure that the following inequality holds, 
		\begin{align}
			C(\frac{\alpha^{2}}{2\gamma} + \frac{\alpha C}{2} \frac{\alpha}{\varsigma_2-(1+\alpha)C})C_pC_N \leq C_{\epsilon}.
		\end{align}
		Enumerating  all $l = 0,1,2,\cdots L$  in the aforementioned proof, we can prove $ 0 < \phi^{n+1,l+1}_h < 1 $ and $0<c^{n+1,l+1}_{h}<\frac{1}{\beta}$, for any $l= 0,1,2,\cdots, L$, where $L$ is the number of iterations. Now we complete the proof.
		
	\end{proof}
	\section{Convergence of the iterative method and energy stability}
	Next, we will use the Banach fixed-point theorem to prove the convergence of the above iterative method. 
	We now introduce a mapping $\mathcal{F}_h$ such that $c_{h}^{n+1,l+1}$ = $\mathcal{F}_h(c_{h}^{n+1,l})$, where $c_{h}^{n+1,l+1}$ along with $\textbf{u}_{s,h}^{n+1,l+1}$ $p_{h}^{n+1, l+1}$, $\textbf{u}_{f,h}^{n+1, l+1}$, $\phi^{n+1,l+1}_{h}$, $\mu_{h}^{n+1,l+1}$ solves the equations \eqref{eq-3-2-1}-\eqref{eq-3-2-5}.
	To prove $\mathcal{F}_h$ is a contraction map, we introduce the following iterative error equation.
	Let $e_{c_{h}}^{n+1,l+1}$, $e_{p_{h}}^{n+1,l+1}$, $e_{\mu_{h}}^{n+1,l+1}$, $e_{\boldsymbol{u}_{s,h}}^{n+1,l+1}$, $e_{\boldsymbol{u}_{f,h}}^{n+1,l+1}$ be the corresponding difference between the $l$-th layer and $(l+1)$-th layer solutions of \eqref{eq-3-2-1}-\eqref{eq-3-2-5}, which satisfy the following variational problem: find $(e_{c_{h}}^{n+1,l+1}, e_{p_{h}}^{n+1,l+1}, e_{\mu_{h}}^{n+1,l+1}, e_{\boldsymbol{u}_{s,h}}^{n+1,l+1}, e_{\boldsymbol{u}_{f,h}}^{n+1,l+1}) \in \mathcal{Q}_h\times\mathcal{Q}_h\times\mathcal{Q}_h\times\mathcal{V}_h\times\mathcal{U}_h^{0}$ there hold
	\begin{subequations}
		\begin{align}
			&\mathcal{A}(e_{\boldsymbol{u}_{s,h}}^{n+1,l+1}, e_{p_{h}}^{n+1,l+1}, \textbf{v}_h) = 0,\label{eq-4-3-6}\\
			&(\frac{e_{\phi_{h}}^{n+1, l}c^{n+1,l}_{h}+\phi_{h}^{n+1,l}e_{c_{h}}^{n+1,l+1}}{\tau}, q_h) + 
			\sum_{e\in \mathcal{E}_h^I}\langle c^{n*}_he_{\boldsymbol{u}_{f, h}}^{n+1, l}\cdot\boldsymbol{n}_e, [q_h]\rangle_e\label{eq-4-3-7}\\\nonumber
			&\qquad\qquad\qquad\qquad\qquad+ \sum_{e\in \mathcal{E}_h^I}\frac{\varsigma_1}{h_e}\langle [e_{\mu_{h}}^{n+1,l+1}], [q_h]\rangle_e  = 0, \\	
			&(\lambda^{-1}(\phi_{h}^{n}) e_{\boldsymbol{u}_{f, h}}^{n+1, l+1}, \boldsymbol{w}_{h}) = \sum_{e\in \mathcal{E}_h^I}\langle [e_{\mu_{h}}^{n+1,l+1}],c^{n*}_h \boldsymbol{w}_{h}\cdot\boldsymbol{n}_e\rangle_e ,\label{eq-4-3-8}\\
			&(e_{\phi_{ h}}^{n+1, l+1}, \varphi_h) = \frac{1}{N}(e_{p_{h}}^{n+1,l+1}, \varphi_h) + \alpha (\nabla\cdot e_{\boldsymbol{u}_{s, h}}^{n+1, l+1}, \varphi_h)-\alpha \sum_{e \in \mathcal{E}^{I}_h}\langle\{\varphi_h\boldsymbol{n}_e\}, [e_{\boldsymbol{u}_{s,h}}^{n+1,l+1}]\rangle_e,\label{eq-4-33-9}
		\end{align}
	\end{subequations}
	where 
	\begin{align}
		&e_{\mu_{h}}^{n+1,l+1} = \frac{\theta_n R T}{ c_h^n(1-\beta c_h^n)^2}e_{c_{h}}^{n+1,l+1},\label{eq-4-3-mu}\\
		&e_{p_{h}}^{n+1,l+1} = c_h^{n}e_{\mu_{h}}^{n+1,l+1}.\label{eq-4-3-p}
	\end{align}
	\begin{lemma}\label{lem-4-c}
		Assume that $0 < \varrho_0 \leq  \beta c^n_h \leq \varrho < 1$ and the parameters $\gamma, \varsigma_1, \varsigma_2$ and $N$ are chosen to be large enough,  we can get $\mathcal{F}_h: c^{n+1,l}_h\rightarrow c^{n+1,l+1}_h$ is a contraction, i.e., there exists a constant $C_{con} \in (0, 1)$ such that
		\begin{align}
			&\frac{3C_{\phi, min}}{4}\sum_{K \in \mathcal{K}_h}\|\mathcal{F}_h(c_{h}^{n+1,l}) - \mathcal{F}_h(c_{h}^{n+1,l-1})\|^{2}_{L^2(K)}\\\nonumber
			&+ (\varsigma_1C_{\mu, min} - \frac{\varrho^{2} C}{\beta^{2} C_{\lambda}})\sum_{e\in \mathcal{E}_h^I}\frac{1}{h_e}\|[\mathcal{F}_h(c_{h}^{n+1,l}) - \mathcal{F}_h(c_{h}^{n+1,l-1})]\|^{2}_{L^2(e)} \\\nonumber
			&< C_{con}\left(\frac{3C_{\phi, min}}{4}\sum_{K \in \mathcal{K}_h}\|c_{h}^{n+1,l} - c_{h}^{n+1,l-1}\|^{2}_{L^2(K)}
			\right.\\\nonumber
			&  \left.+ (\varsigma_1C_{\mu, min} - \frac{\varrho^{2} C}{\beta^{2} C_{\lambda}})\sum_{e\in \mathcal{E}_h^I}\frac{1}{h_e}\|[c_{h}^{n+1,l} - c_{h}^{n+1,l-1}]\|^{2}_{L^2(e)}\right),
		\end{align}
		where $C_{\lambda}$ is the lower bound of $\lambda^{-1}(\phi^{n})$, $C_{\mu, min}$ and $C_{\mu, max}$ are the upper and lower bounds of $\frac{\theta_n R T}{ c_h^n(1-\beta c_h^n)^2}$, respectively.
	\end{lemma}
	\begin{proof}
		Let $q_h = e_{c_{h}}^{n+1,l+1}$ in \eqref{eq-4-3-7} and  $\boldsymbol{w}_{h} = e_{\boldsymbol{u}_{f, h}}^{n+1,l}$ in \eqref{eq-4-3-8}, we can get
		\begin{align}
			&(\frac{e_{\phi_{h}}^{n+1, l}c^{n+1,l}_{h}+\phi_{h}^{n+1,l}e_{c_{h}}^{n+1,l+1}}{\tau}, e_{c_{h}}^{n+1,l+1}) + \sum_{e\in \mathcal{E}_h^I}\langle c^{n*}_h e_{\boldsymbol{u}_{f, h}}^{n+1, l}\cdot\boldsymbol{n}, [e_{c_{h}}^{n+1,l+1}]\rangle_e\\\nonumber 
			&\qquad\qquad\qquad\qquad+ \sum_{e\in \mathcal{E}_h^I}\frac{\varsigma_1}{h_e}\langle [e_{\mu_{h}}^{n+1,l+1}], [e_{c_{h}}^{n+1,l+1}]\rangle_e  = 0,\\
			&(\lambda^{-1}(\phi_{h}^{n}) e_{\boldsymbol{u}_{f, h}}^{n+1, l}, e_{\boldsymbol{u}_{f, h}}^{n+1,l}) = \sum_{e\in \mathcal{E}_h^I}\langle [e_{\mu_{h}}^{n+1,l}],c^{n*}_h e_{\boldsymbol{u}_{f, h}}^{n+1,l}\cdot\boldsymbol{n}_e\rangle_e.\label{eq-250407}
		\end{align}
		By the Cauchy-Schwarz inequality, the Young's inequality and $\varrho_0 \leq  \beta c^n_h \leq \varrho$, we obtain
		\begin{align}\label{eq-109-1}
			(e_{\phi_{h}}^{n+1, l}c^{n+1,l}_{h}, e_{c_{h}}^{n+1,l+1})&\leq \frac{\varrho^{2}}{\beta^2C_{\phi, min}}\sum_{K \in \mathcal{K}_h}\| e_{\phi_{h}}^{n+1, l}\|^{2}_{L^2(K)} \\\nonumber
			&\qquad+ \frac{C_{\phi, min}}{4}\sum_{K \in \mathcal{K}_h} \|e_{c_{h}}^{n+1,l+1}\|^{2}_{L^2(K)},
		\end{align}
		Using the trace inequality and the bound $\varrho_0 \leq \beta c^n_h \leq \varrho$, we can get
		\begin{align}	
			\sum_{e\in \mathcal{E}_h^I}\langle [e_{c_{h}}^{n+1,l+1}],c^{n*}_h e_{\boldsymbol{u}_{f,h}}^{n+1,l}\cdot\boldsymbol{n}\rangle_e
			&\leq \frac{C_{\lambda}}{4}\sum_{K\in \mathcal{K}_h}\| e_{\boldsymbol{u}_{f,h}}^{n+1,l}\|^{2}_{L^{2}(K)} \\\nonumber
			&\qquad+ \frac{\varrho^{2} C}{\beta^{2} C_{\lambda}}\sum_{e\in \mathcal{E}_h^I}\frac{1}{h_e}\|[e_{c_{h}}^{n+1,l+1}]\|^{2}_{L^{2}(e)}.
		\end{align}
		Similarly, we have
		\begin{align}\label{eq-250407-2}
			\sum_{e\in \mathcal{E}_h^I}\langle [e_{\mu_{h}}^{n+1,l}],c^{n*}_h e_{\boldsymbol{u}_{f,h}}^{n+1,l}\cdot\boldsymbol{n}\rangle_e
			&\leq \frac{C_{\lambda}}{4}\sum_{K\in \mathcal{K}_h}\| e_{\boldsymbol{u}_{f,h}}^{n+1,l}\|^{2}_{L^{2}(K)}\\\nonumber
			&\qquad + \frac{\varrho^{2} C^{2}_{\mu, max}C}{\beta^{2} C_{\lambda}}\sum_{e\in \mathcal{E}_h^I}\frac{1}{h_e}\|[e_{c_{h}}^{n+1,l}]\|^{2}_{L^{2}(e)}.
		\end{align}	
		By \eqref{eq-250407} and \eqref{eq-250407-2}, we deduce
		\begin{align}	
			\frac{3C_{\lambda}}{4}\sum_{K\in \mathcal{K}_h}\| e_{\boldsymbol{u}_{f,h}}^{n+1,l}\|^{2}_{L^{2}(K)}&\leq \frac{\varrho^{2} C^{2}_{\mu, max}C}{\beta^{2} C_{\lambda}}\sum_{e\in \mathcal{E}_h^I}\frac{1}{h_e}\|[e_{c_{h}}^{n+1,l}]\|^{2}_{L^{2}(e)}.\label{eq-109-2}
		\end{align}
		Combining \eqref{eq-109-1}-\eqref{eq-109-2}, we obtain
		\begin{align}\label{eq-109-3}
			\frac{3C_{\phi, min}}{4}\sum_{K \in \mathcal{K}_h}\|e_{c_{h}}^{n+1,l+1}\|^{2}_{L^2(K)} + (\varsigma_1 C_{\mu, min}- \frac{\varrho^{2} C}{\beta^{2} C_{\lambda}})\sum_{e\in \mathcal{E}_h^I}\frac{1}{h_e}\|[e_{c_{h}}^{n+1,l+1}]\|^{2}_{L^2(e)}\\\nonumber
			\leq \frac{\varrho^{2}}{\beta^2C_{\phi, min}}\sum_{K \in \mathcal{K}_h}\| e_{\phi_{h}}^{n+1, l}\|^{2}_{L^2(K)} + \frac{\varrho^{2} C^{2}_{\mu, max}C}{\beta^{2} C_{\lambda}}\sum_{e\in \mathcal{E}_h^I}\frac{1}{h_e}\|[e_{c_{h}}^{n+1,l}]\|^{2}_{L^{2}(e)}.
		\end{align}
	   Letting $\varphi_{h}=e_{\phi_{h}}^{n+1, l}$ in \eqref{eq-4-33-9}, we have
		\begin{align}
			(e_{\phi_{ h}}^{n+1, l}, e_{\phi_{ h}}^{n+1, l}) = \frac{1}{N}(e_{p_h}^{n+1,l}, e_{\phi_{ h}}^{n+1, l}) &+ \alpha (\nabla\cdot e_{\boldsymbol{u}_{s, h}}^{n+1, l}, e_{\phi_{ h}}^{n+1, l})\label{eq-4-3-9}\\\nonumber
			&-\alpha \sum_{e \in \mathcal{E}^{I}_h}\langle\{e_{\phi_{ h}}^{n+1, l}\boldsymbol{n}_e\}, [e_{\boldsymbol{u}_{s,h}}^{n+1,l}]\rangle_e.
		\end{align}	
		As a result of the Cauchy-Schwarz inequality, the Young's inequality,  \eqref{eq-4-3-p} and $\varrho_0 \leq  \beta c^n_h \leq \varrho$, we have
		\begin{align}\label{eq-1010-1}
			\alpha (\nabla\cdot e_{\boldsymbol{u}_{s, h}}^{n+1, l}, e_{\phi_{ h}}^{n+1, l}) \leq \frac{\alpha }{2}\sum_{K \in \mathcal{K}_h}\|\nabla\cdot e_{\boldsymbol{u}_{s,h}}^{n+1,l}\|^{2}_{L^2(K)} 
			+ \frac{\alpha }{2}\sum_{K \in \mathcal{K}_h}\|e_{\phi_{ h}}^{n+1, l}\|^{2}_{L^2(K)},
		\end{align}
		\begin{align}\label{eq-1010-22}
			\frac{1}{N}(e_{p_h}^{n+1,l}, e_{\phi_{ h}}^{n+1, l})&\leq \frac{1}{2N}\sum_{K \in \mathcal{K}_h}\|e_{p_h}^{n+1,l}\|^{2}_{L^2(K)} + \frac{1}{2N}\sum_{K \in \mathcal{K}_h}\|e_{\phi_{ h}}^{n+1, l}\|^{2}_{L^2(K)}\\\nonumber
			&\leq \frac{(\varrho C_{\mu, \max})^2}{2N\beta^2}\sum_{K \in \mathcal{K}_h}\|e_{c_h}^{n+1,l}\|^{2}_{L^2(K)} + \frac{1}{2N}\sum_{K \in \mathcal{K}_h}\|e_{\phi_{ h}}^{n+1, l}\|^{2}_{L^2(K)}.
		\end{align}
		By the trace inequality, we obtain
		\begin{align}\label{eq-1010-2}
			\alpha \sum_{e \in \mathcal{E}^{I}_h}\langle\{e_{\phi_{ h}}^{n+1, l}\boldsymbol{n}_e\}, [e_{\boldsymbol{u}_{s,h}}^{n+1,l}]\rangle_e\leq \frac{\alpha }{2}\sum_{K \in \mathcal{K}_h}\|e_{\phi_{ h}}^{n+1, l}\|^{2}_{L^{2}(K)} + \alpha  C\sum_{e \in \mathcal{E}^{I}_h}\frac{1}{h_e}\| [e_{\boldsymbol{u}_{s,h}}^{n+1,l}]\|^{2}_{L^{2}(e)}.
		\end{align}
		
		On account of \eqref{eq-1010-1}-\eqref{eq-1010-2}, we get the following inequality
		
		\begin{align}\label{eq-1010-3}
			\sum_{K \in \mathcal{K}_h}\|e_{\phi_{h}}^{n+1, l}\|^{2}_{L^2(K)}&\leq \frac{(\varrho C_{\mu, \max})^2}{2N\beta^2(1-\alpha - \frac{1}{2N})}\|e_{c_h}^{n+1,l}\|^{2}_{L^2(K)} + \frac{\alpha}{2(1-\alpha - \frac{1}{2N})}\sum_{K \in \mathcal{K}_h}\|\nabla\cdot e_{\boldsymbol{u}_{s,h}}^{n+1,l}\|^{2}_{L^2(K)}\\\nonumber
			&\qquad+\frac{\alpha C}{(1-\alpha - \frac{1}{2N})}\sum_{e \in \mathcal{E}^{I}_h}\frac{1}{h_e}\| [e_{\boldsymbol{u}_{s,h}}^{n+1,l}]\|^{2}_{L^{2}(e)}.
		\end{align}
		Let $\textbf{v}_h = e_{\boldsymbol{u}_{s,h}}^{n+1,l}$  in \eqref{eq-4-3-6}, we get 
		\begin{align}\label{eq-1010-4}
			&\sum_{K \in \mathcal{K}_h}(\mathbf{\sigma}_e(e_{\boldsymbol{u}_{s,h}}^{n+1,l}), \varepsilon(e_{\boldsymbol{u}_{s,h}}^{n+1,l}))_K-\sum_{e \in \mathcal{E}^{I}_h}\langle\{\mathbf{\sigma}_e(e_{\boldsymbol{u}_{s,h}}^{n+1,l}) \boldsymbol{n}_e\}, [e_{\boldsymbol{u}_{s,h}}^{n+1,l}]\rangle_e \\\nonumber
			&\qquad - \alpha\sum_{K \in \mathcal{K}_h}(e_{p_h}^{n+1,l}, \nabla\cdot e_{\boldsymbol{u}_{s,h}}^{n+1,l})_{K}+\alpha\sum_{e \in \mathcal{E}^{I}_h}\langle\{e_{p_h}^{n+1,l}\boldsymbol{n}_e\}, [e_{\boldsymbol{u}_{s,h}}^{n+1,l}]\rangle_e\\\nonumber
			&\qquad-\sum_{e \in \mathcal{E}^{I}_h}\langle [e_{\boldsymbol{u}_{s,h}}^{n+1,l}],\{\mathbf{\sigma}_{e}(e_{\boldsymbol{u}_{s,h}}^{n+1,l}) \boldsymbol{n}_e\}\rangle_e+\sum_{e \in \mathcal{E}^{I}_h} \frac{\varsigma_2}{h_e}\langle[e_{\boldsymbol{u}_{s,h}}^{n+1,l}], [e_{\boldsymbol{u}_{s,h}}^{n+1,l}]\rangle_e\\\nonumber
			&= 0.
		\end{align}
		By the Cauchy-Schwarz inequality, Young's inequality, trace inequality and $\varrho_0$ $\leq$  $\beta c^n_h$ $\leq$ $\varrho$, we have
		\begin{align}
			&2\sum_{e \in \mathcal{E}^{I}_h}\langle\{\mathbf{\sigma}_e(e_{\boldsymbol{u}_{s,h}}^{n+1,l}) \boldsymbol{n}_e\}, [e_{\boldsymbol{u}_{s,h}}^{n+1,l}]\rangle_e\label{eq-1010-5}\\\nonumber
			&\leq \frac{1}{2}\sum_{K \in \mathcal{K}_h} (\mathbf{\sigma}_e(e_{\boldsymbol{u}_{s,h}}^{n+1,l}), \varepsilon(e_{\boldsymbol{u}_{s,h}}^{n+1,l}))_{K} 
			+ C\sum_{e \in \mathcal{E}^{I}_h}\frac{1}{h_e}\|[e_{\boldsymbol{u}_{s,h}}^{n+1,l}]\|^{2}_{L^{2}(e)},\\
			&\alpha\sum_{K \in \mathcal{K}_h}(e_{p_{h}}^{n+1,l}, \nabla\cdot e_{\boldsymbol{u}_{s,h}}^{n+1,l})_{K}\label{eq-1010-6}\\\nonumber
			&\leq\alpha\sum_{K\in \mathcal{K}_h}\|c^{n}_he_{\mu_{h}}^{n+1,l}\|_{L^{2}(K)}\|\nabla\cdot e_{\boldsymbol{u}_{s,h}}^{n+1,l}\|_{L^{2}(K)}\\\nonumber
			&\leq \frac{\alpha (\varrho C_{\mu, \max})^2}{2\gamma\beta^2}\sum_{K \in \mathcal{K}_h}\| e_{c_{h}}^{n+1,l}\|^{2}_{L^{2}(K)} + \frac{\alpha\gamma}{2}\sum_{K \in \mathcal{K}_h}\|\nabla\cdot e_{\boldsymbol{u}_{s,h}}^{n+1,l}\|^{2}_{L^{2}(K)},\\
			&\alpha\sum_{e \in\mathcal{E}^{I}_h}\langle\{e_{p_{h}}^{n+1,l}\boldsymbol{n}_e\}, [e_{\boldsymbol{u}_{s,h}}^{n+1,l}]\rangle_e\\\nonumber
			&\leq \frac{\alpha C^{2}_{\mu, \max}\varrho^2}{2\beta^2}\sum_{K \in \mathcal{K}_h}\|e_{c_{h}}^{n+1,l}\|^{2}_{L^{2}(K)} + \alpha C \sum_{e \in \mathcal{E}^{I}_h}\frac{1}{h_e }\| [e_{\boldsymbol{u}_{s,h}}^{n+1,l}]\|^{2}_{L^{2}(e)}.
		\end{align}
		Substituting \eqref{eq-1010-5} and \eqref{eq-1010-6} into \eqref{eq-1010-4}, we get
		\begin{align}\label{eq-1012-1}
			&\eta\sum_{K \in \mathcal{K}_h} \|\varepsilon(e_{\boldsymbol{u}_{s,h}}^{n+1,l})\|^{2}_{L^2(K)} + (\frac{1-\alpha}{2})\gamma\sum_{K \in \mathcal{K}_h} \|\nabla\cdot(e_{\boldsymbol{u}_{s,h}}^{n+1,l})\|^{2}_{L^2(K)}\\\nonumber
			&\qquad+(\varsigma_2-(1+\alpha)C)\sum_{e \in \mathcal{E}^{I}_h}\frac{1}{h_e}\|[e_{\boldsymbol{u}_{s,h}}^{n+1,l}]\|^{2}_{L^2(e)} \\\nonumber
			&\leq \left(\frac{\alpha (1 + \gamma)(\varrho C_{\mu, \max})^2}{2\gamma\beta^2} \right)\sum_{K \in \mathcal{K}_h}\| e_{c_{h}}^{n+1,l}\|^{2}_{L^{2}(K)}.
		\end{align} 
		By \eqref{eq-1012-1}, it is straightforward to derive the following two inequalities:
		\begin{align}\label{eq-1010-7}
			\sum_{K \in \mathcal{K}_h} \|\nabla\cdot(e_{\boldsymbol{u}_{s,h}}^{n+1,l})\|^{2}_{L^2(K)}\leq \frac{1}{(\frac{1}{2}-\alpha)\gamma}\left(\frac{\alpha (1 + \gamma)(\varrho C_{\mu, \max})^2}{2\gamma\beta^2} \right)\sum_{K \in \mathcal{K}_h}\| e_{c_{h}}^{n+1,l}\|^{2}_{L^{2}(K)},
		\end{align}
		and
		\begin{align}\label{eq-1010-8}
			\sum_{e \in \mathcal{E}^{I}_h}\frac{1}{h_e}\|[e_{\boldsymbol{u}_{s,h}}^{n+1,l}]\|^{2}_{L^2(e)}\leq \frac{1}{(\varsigma_2-(1+\alpha)C)}\left(\frac{\alpha (1 + \gamma)(\varrho C_{\mu, \max})^2}{2\gamma\beta^2} \right)\sum_{K \in \mathcal{K}_h}\| e_{c_{h}}^{n+1,l}\|^{2}_{L^{2}(K)}.                                     
		\end{align}
		Combining \eqref{eq-1010-3}, \eqref{eq-1010-7} and \eqref{eq-1010-8}, we obtain
		\begin{align}\label{eq-1010-9}
			&\sum_{K \in \mathcal{K}_h}\|e_{\phi_{h}}^{n+1, l}\|^{2}_{L^2(K)}\\\nonumber
			&\leq \frac{1}{(1-\alpha - \frac{1}{2N})}\left((\frac{1}{(\frac{1}{2}-\alpha)\gamma}+\frac{\alpha C}{(\varsigma_2-(1+\alpha)C)})\frac{\alpha (1+\gamma)(\varrho C_{\mu, \max})^2}{2\gamma\beta^2} \right.\\\nonumber
			&\left.\quad+ \frac{(\varrho C_{\mu, \max})^2}{2N\beta^2} \right)\sum_{K \in \mathcal{K}_h}\| e_{c_{h}}^{n+1,l}\|^{2}_{L^{2}(K)}.
		\end{align}
		Substituting \eqref{eq-1010-9} into \eqref{eq-109-3}, we can get
		\begin{align}
			&\frac{3C_{\phi, min}}{4}\sum_{K \in \mathcal{K}_h}\|e_{c_{h}}^{n+1,l+1}\|^{2}_{L^2(K)} + (\varsigma_1C_{\mu, min} - \frac{\varrho^{2} C}{\beta^{2} C_{\lambda}})\sum_{e\in \mathcal{E}_h^I}\frac{1}{h_e}\|[e_{c_{h}}^{n+1,l+1}]\|^{2}_{L^2(e)}\\\nonumber
			&\leq\frac{\varrho^{2}}{\beta^2C_{\phi, min}}\frac{(C_{\mu, \max}\varrho)^2}{2\beta^2} \frac{1}{(1-\alpha-\frac{1}{2N})}\left((\frac{\alpha}{(1 - 2\alpha)\gamma}\right.\\\nonumber
			&\left.\quad+\frac{\alpha C}{( \varsigma_2-(1+\alpha)C)})(\frac{\alpha (1+\gamma)}{\gamma}) +  \frac{1}{N}\right)\sum_{K \in \mathcal{K}_h}\| e_{c_{h}}^{n+1,l}\|^{2}_{L^{2}(K)}\\\nonumber
			&\quad+ \frac{\varrho^{2} C^{2}_{\mu, max}C}{\beta^{2} C_{\lambda}}\sum_{e\in \mathcal{E}_h^I}\frac{1}{h_e}\|[e_{c_{h}}^{n+1,l}]\|^{2}_{L^{2}(e)},
		\end{align}
		where $0<\alpha<1$, and $\gamma$, $\varsigma_1$, $\varsigma_2$, $ N$ should be chosen to be sufficiently large such that
		
		\begin{align*}
			\frac{3C_{\phi, \min}}{4} 
			&> \frac{\varrho^{2}}{\beta^2 C_{\phi, \min}} \frac{(C_{\mu, \max} \varrho)^2}{2\beta^2} \frac{1}{\left(1 - \alpha - \frac{1}{2N}\right)} \\
			&\quad \left[ 
			\left( 
			\frac{\alpha}{(1 - 2\alpha)\gamma} 
			+ \frac{\alpha C}{\varsigma_2 - (1 + \alpha)C} 
			\right) 
			\frac{\alpha (1 + \gamma)}{\gamma} 
			+ \frac{1}{N} 
			\right]
		\end{align*}
		and		
		\[
		\left( \varsigma_1 C_{\mu, \min} - \frac{\varrho^2 C}{\beta^2 C_{\lambda}} \right) 
		> \frac{\varrho^2 C_{\mu, \max}^2 C}{\beta^2 C_{\lambda}}.
		\]
		Now we define the constant \( C_{\text{cont}} \) as
		
		\begin{align*}
			C_{\text{cont}} = \min \Bigg( 
			&\frac{
				\frac{\varrho^{2}}{\beta^2 C_{\phi, \min}} 
				\frac{(C_{\mu, \max} \varrho)^2}{2\beta^2} 
				\frac{1}{\left(1 - \alpha - \frac{1}{2N}\right)} 
				\left[ 
				\left( 
				\frac{\alpha}{(1 - 2\alpha)\gamma} 
				+ \frac{\alpha C}{\varsigma_2 - (1 + \alpha)C} 
				\right) 
				\frac{\alpha(1 + \gamma)}{\gamma} 
				+ \frac{1}{N} 
				\right]
			}{\frac{3C_{\phi, \min}}{4}}, \\
			&\quad 
			\frac{
				\frac{\varrho^2 C_{\mu, \max}^2 C}{\beta^2 C_{\lambda}} 
			}{
				\varsigma_1 C_{\mu, \min} - \frac{\varrho^2 C}{\beta^2 C_{\lambda}} 
			}
			\Bigg) < 1,
		\end{align*} 
		then we can get the contraction estimate
		\begin{align}
			&\frac{3C_{\phi, min}}{4}\sum_{K \in \mathcal{K}_h}\|e_{c_{h}}^{n+1,l+1}\|^{2}_{L^2(K)} + (\varsigma_1C_{\mu, min} - \frac{\varrho^{2} C}{\beta^{2} C_{\lambda}})\sum_{e\in \mathcal{E}_h^I}\frac{1}{h_e}\|[e_{c_{h}}^{n+1,l+1}]\|^{2}_{L^2(e)} \\\nonumber
			&< C_{con}\left(\frac{3C_{\phi, min}}{4}\sum_{K \in \mathcal{K}_h}\|e_{c_{h}}^{n+1,l}\|^{2}_{L^2(K)} + (\varsigma_1C_{\mu, min} - \frac{\varrho^{2} C}{\beta^{2} C_{\lambda}})\sum_{e\in \mathcal{E}_h^I}\frac{1}{h_e}\|[e_{c_{h}}^{n+1,l}]\|^{2}_{L^2(e)}\right).
		\end{align}
		The proof is completed.
	\end{proof}
	\begin{theorem}
		There exists a unique solution ($\boldsymbol{u}_{s,h}^{n+1}, p_{h}^{n+1}, \boldsymbol{u}_{f,h}^{n+1}, c^{n+1}_{h}, \mu_{h}^{n+1}, \phi_{ h}^{n+1}$) which satisfies the nonlinear system \eqref{eq-full-1}.
	\end{theorem}
	\begin{proof}
		Lemma \ref{lem-4-c} indicates that the mapping $\mathcal{F}_h$ is contraction. Moreover, the discrete solution spaces are obviously compact sets. Thus the existence of discrete solutions of the nonlinear system \eqref{eq-full-1} can be obtained by the Banach fixed-point theorem, \eqref{eq-109-2}, \eqref{eq-1012-1} and \eqref{eq-1010-9}, 
		\begin{align*}
			& \boldsymbol{u}_{s,h}^{n+1,l+1} \rightarrow \boldsymbol{u}_{s,h}^{n+1} \quad in \quad \mathcal{V}_h,\\
			& \boldsymbol{u}_{f,h}^{n+1,l+1} \rightarrow \boldsymbol{u}_{f,h}^{n+1} \quad in \quad\mathcal{U}_h^{0},\\
			& p_{h}^{n+1,l+1}, \phi_{h}^{n+1,l+1}, c^{n+1, l+1}_{h}  \rightarrow p_{h}^{n+1}, \phi_{h}^{n+1}, c^{n+1}_{h} \quad in \quad \mathcal{Q}_h.
		\end{align*}
		The proof is completed.
	\end{proof}
	\begin{theorem}\label{th-4-2} 
		
		Assume that the boundary condition \eqref{eq-BD-1} holds, $0<\epsilon\leq c^{n}_{h}\leq\frac{\varrho}{\beta} < \frac{1}{\beta}$. The stabilization parameter $\theta_n$ is taken as follow 
		\begin{align}\label{eq-theta}
			\theta_n = \max_{K \in \mathcal{K}_{h}} \left\{1, \frac{\left(1-\beta c_h^{n}\right)^2}{\chi_1^n\left(1-\chi_1^n \beta c_h^{n}\right)^2}, \frac{\left(1-\beta c_h^{n}\right)^2}{\chi_2^n\left(1-\chi_2^n \beta c_h^{n}\right)^2}\right\},
		\end{align}
		where $\chi_1^{n} = 1 - \delta_1\left(1-\beta c^{n}_{h}\right)^{2}$, $\chi_2^{n} = 1 + \delta_2\left(1 - \beta c^{n}_{h}\right)^{2}$. The total free energy generated by the scheme \eqref{eq-full-1} is dissipated as
		$$D_{\tau} E^{n+1}_{h}  \leq 0,$$
		where 
		\begin{align}
			E^{n+1}_h =& \sum_{K \in \mathcal{K}_h}\int_{K}\left(\phi^{n+1}_{h}f(c^{n+1}_h)+\frac{1}{2}\mathbf{\sigma}_{e}(\boldsymbol{u}^{n+1}_{s,h}):\varepsilon(\boldsymbol{u}^{n+1}_{s,h})+\frac{1}{2N}|p^{n+1}_{h}|^2\right) ~ d \boldsymbol{x}\\\nonumber
			&+\sum_{e \in \mathcal{E}_{h}^I} \frac{\varsigma_1}{2h_e}\langle[\boldsymbol{u}^{n+1}_{s,h}], [ \boldsymbol{u}^{n+1}_{s,h}]\rangle_e
			-\sum_{e \in \mathcal{E}_{h}^I}\langle\{\mathbf{\sigma}_e(\boldsymbol{u}^{n+1}_{s,h}) \boldsymbol{n}_e\}, [\boldsymbol{u}^{n+1}_{s,h}]\rangle_e.
		\end{align}
	\end{theorem}
	\begin{proof}
		By the Taylor expansion and assuming that $\xi$ is a number between $c_h^{n}$ and $c_h^{n+1}$, we have
		\begin{align}
			f\left(c_h^{n+1}\right)-f\left(c_h^{n}\right)=\mu\left(c_h^{n}\right)\left(c_h^{n+1}-c_h^{n}\right)+\frac{f^{\prime \prime}(\xi)}{2}\left(c_h^{n+1}-c_h^{n}\right)^2.
		\end{align}
		In view of the stabilized chemical potential \eqref{eq-stable}, we further obtain
		\begin{align}\label{eq-3-4}
			f\left(c_h^{n+1}\right)-f\left(c_h^{n}\right) & =\mu^{n+1}_h\left(c_h^{n+1}-c_h^{n}\right)+\left(\frac{f^{\prime \prime}(\xi)}{2}-\frac{\theta_n R T}{c_h^{n}\left(1-\beta c_h^{n}\right)^2}\right)\left(c_h^{n+1}-c_h^{n}\right)^2 .
		\end{align}
		From the expression of $f$, we can obtain that$f^{\prime \prime}(\xi)$ can be split into two parts as
		$$
		f^{\prime \prime}(\xi)=f_{i r}^{\prime \prime}(\xi)+f_{a t t}^{\prime \prime}(\xi),
		$$
		where
		\begin{align}
			f_{i r}^{\prime \prime}(\xi)	&= \frac{R T}{\xi(1-\beta \xi)^2}>0, \label{eq-1019-1}\\
			f_{a t t}^{\prime \prime} 	&= \frac{a(T)}{2\sqrt{2}\beta}\left(\frac{(1-\sqrt{2})\beta}{1+(1-\sqrt{2})\beta c} - \frac{(1+\sqrt{2})\beta}{(1 + (1 + \sqrt{2})\beta c)}\right)\\\nonumber 
			&\quad + \frac{a(T)}{2\sqrt{2}\beta}\left(\frac{(1-\sqrt{2})\beta}{(1+(1-\sqrt{2})\beta c)^2} - \frac{(1+\sqrt{2})\beta}{(1 + (1 + \sqrt{2})\beta c)^2}\right) < 0.
		\end{align}
		Differentiating \eqref{eq-1019-1} with respect to $c$, we obtain
		$$
		f_{i r}^{\prime \prime \prime}(\xi)=-\frac{R T(1-\beta \xi)(1-3 \beta \xi)}{\xi^2(1-\beta \xi)^4}=-\frac{R T(1-3 \beta \xi)}{\xi^2(1-\beta \xi)^3},
		$$
		we deduce that
		\begin{align}
			& f_{i r}^{\prime \prime \prime}(\xi) \leq 0, \text { for } \xi \leq \frac{1}{3 \beta}, \label{eq-1019-2}\\
			& f_{i r}^{\prime \prime \prime}(\xi)>0, \text { for } \xi>\frac{1}{3 \beta}\label{eq-1019-3}.
		\end{align}
		
		By \eqref{eq-1019-1}, \eqref{eq-1019-2} and \eqref{eq-1019-3}, we get
		$$
		c_h^{n}\left(1-\beta c_h^{n}\right)^2 \frac{f_{i r}^{\prime \prime}(\xi)}{2} \leq \frac{R T}{2} \sup \left\{1, \frac{c_h^{n}\left(1-\beta c_h^{n}\right)^2}{c_h^{n+1}\left(1-\beta c_h^{n+1}\right)^2}\right\} .
		$$
		Due to the relationship $\chi_1^n c_h^{n} \leq c_h^{n+1} \leq \chi_2^n c_h^{n}$, we get the following estimate
		\begin{align}\label{eq-1019-4}
			\frac{1}{R T} c_h^{n}\left(1-\beta c_h^{n}\right)^2\frac{f_{i r}^{\prime \prime}(\xi)}{2} \leq \frac{1}{2}\sup \left\{1, \frac{\left(1-\beta c_h^{n}\right)^2}{\chi_1^n\left(1-\chi_1^n \beta c_h^{n}\right)^2}, \frac{\left(1-\beta c_h^{n}\right)^2}{\chi_2^n\left(1-\chi_2^n \beta c_h^{n}\right)^2}\right\}=\theta_n.
		\end{align}

		By \eqref{eq-1019-4} and the fact $f_{a t t}^{\prime \prime}(\xi)<0$, we obtain that $c_h^{n}\left(1-\beta c_h^{n}\right)^2 f^{\prime \prime}(\xi) \leq \theta_n R T$. Thus, \eqref{eq-3-4} becomes
		\begin{align}\label{eq-1019-5}
			f\left(c_h^{n+1}\right)-f\left(c_h^{n}\right) \leq \mu^{n+1}_{h}\left(c_h^{n+1}-c_h^{n}\right).
		\end{align}
		Due to \eqref{eq-2-2-2}, \eqref{eq-2-2-3} and \eqref{eq-1019-5}, we obtain that
		\begin{align}\label{eq-2-2-13}
			& \frac{1}{\tau}\sum_{K \in \mathcal{K}_h}\int_{K}\left((\phi_h^{n+1} f(c_h^{n+1})-\phi_h^n f(c_h^n))+(\phi_h^{n+1}-\phi_h^{n}, p_h^{n+1})\right) ~d\boldsymbol{x}\\\nonumber
			\leq &\frac{1}{\tau}\sum_{K \in \mathcal{K}_h}\int_{K} \left(f(c_h^n)(\phi_h^{n+1}-\phi_h^n)+\phi_h^{n+1} \mu_h^{n+1}(c_h^{n+1}-c_h^n)\right.\\\nonumber
			&\left.+(c_h^n \mu_h^{n+1}-f(c_h^n))(\phi_h^{n+1}-\phi_h^n)\right) ~d\boldsymbol{x}\\\nonumber
			=&\frac{1}{\tau} \sum_{K \in \mathcal{K}_h}\int_{K} \mu_h^{n+1}\left(\phi_h^{n+1} c_h^{n+1}-\phi_h^n c_h^n\right) d \boldsymbol{x}\\\nonumber
			=&-\sum_{e \in \mathcal{E}_h^I}\int_{e} [\mu^{n+1}_h] c^{n*}_h\boldsymbol{u}_{f,h}^{n+1}\cdot\boldsymbol{n}_e +  [\mu^{n+1}_h]^2 ~d 	\boldsymbol{x}\\\nonumber
			=&-\left(\sum_{K \in \mathcal{K}_h}\int_{K} \lambda^{-1}(\phi_{h}^{n})|\boldsymbol{u}_{f,h}^{n+1}|^2  ~d \boldsymbol{x} +\sum_{e \in \mathcal{E}^I_h}\int_{e} [\mu^{n+1}_{h}]^{2} ~d \boldsymbol{x}\right).
		\end{align}	
		As it is proved in \cite{Chencmame2024}, we can deduce that 
		\begin{align}\label{eq-2-3130}
			&\frac{1}{2}\sum_{K \in \mathcal{K}_h}D_{\tau}(\mathbf{\sigma}_e(\boldsymbol{u}_{s,h}^{n+1}), \varepsilon(\boldsymbol{u}_{s,h}^{n+1}))_K + \sum_{e \in \mathcal{E}_h^I} \frac{\varsigma_1}{2h_e}D_{\tau}\langle[\boldsymbol{u}_{s,h}^{n+1}], [ \boldsymbol{u}_{s,h}^{n+1}]\rangle_e\\\nonumber
			& -\sum_{e \in \mathcal{E}_h^I}\langle\{\mathbf{\sigma}_e(\boldsymbol{u}_{s,h}^{n+1}) \boldsymbol{n}_e\}, [\boldsymbol{u}_{s,h}^{n+1}]\rangle_e +\sum_{e \in \mathcal{E}_h^I}\langle\{\mathbf{\sigma}_e(\boldsymbol{u}_{s,h}^{n}) \boldsymbol{n}_e\}, [\boldsymbol{u}_{s,h}^{n}]\rangle_e\\\nonumber
			& -\sum_{K \in \mathcal{K}_h}(D_{\tau} \phi^{n+1}_h,p^{n+1}_h)_K\\\nonumber
			\leq&  \  0. 
		\end{align}
		Combining \eqref{eq-2-2-13}-\eqref{eq-2-3130}, we get
		$$D_{\tau} E^{n+1}_{h}  \leq 0.$$
		The proof is completed.
	\end{proof}
	\section{Numerical examples}
	In this section, we present several numerical experiments to validate our theoretical analysis and demonstrate the superior performance of the numerical schemes.
	In all the numerical experiments, we consider the flow of methane in porous media, with its physical properties shown in Table \ref{table1}. In numerical simulation, we take $\delta_1 = \delta_2 = \delta $. For a fixed  $\delta$, the adaptive time step and stabilization terms are selected according to (\ref{eq-tau}) and (\ref{eq-theta}). In all the numerical examples, the viscosity is set to be $\eta = 10^{-5}~ \mathrm{Pa\cdot s}$. For the two-dimensional example, we select the computational domain as $\Omega = [0, L]^2$ where $L$ = 100 m, and use a quasi-uniform triangular mesh with  $20,000$ elements. For the three-dimensional numerical experiments, we employed a tetrahedral mesh consisting of 162,000 elements. Figure \ref{fig-mesh} illustrates the specific mesh used in our simulations.

	\begin{figure}[htbp]
		\centering
		\includegraphics[width=5.5cm, height=5.5cm,trim=5cm 0cm 5cm 0cm,clip]{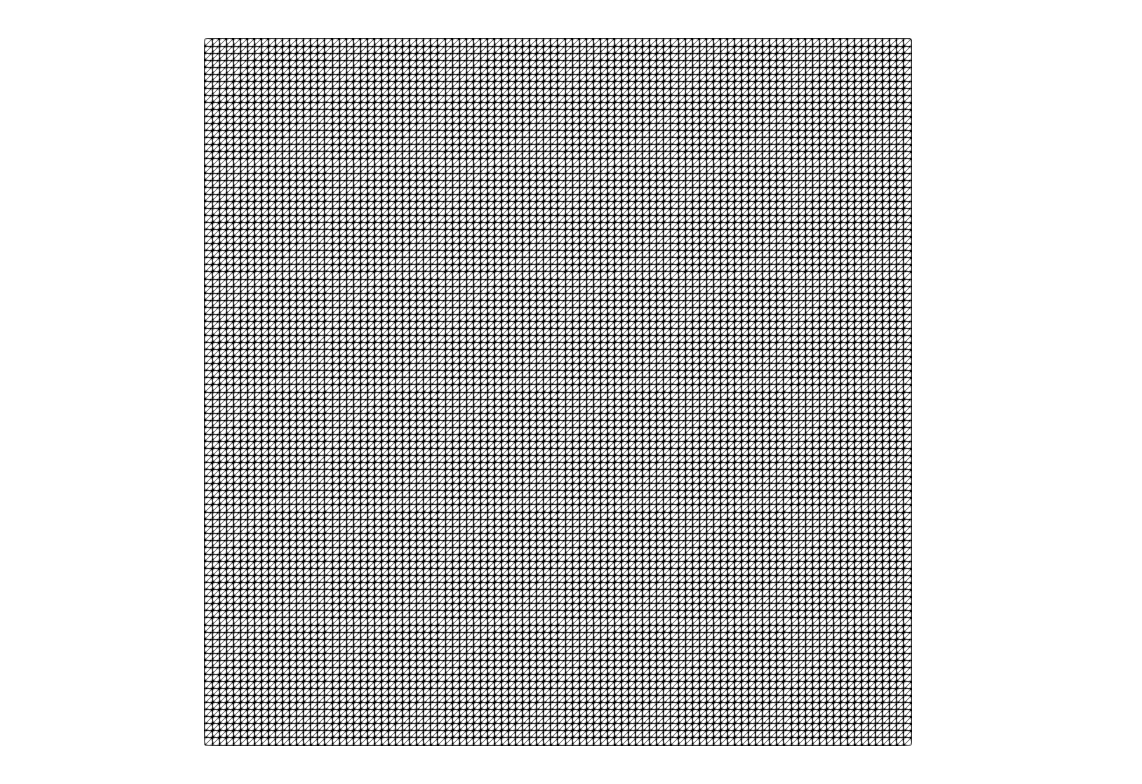}
		\includegraphics[width=5.5cm, height=5.5cm,trim=5cm 0cm 5cm 0cm,clip]{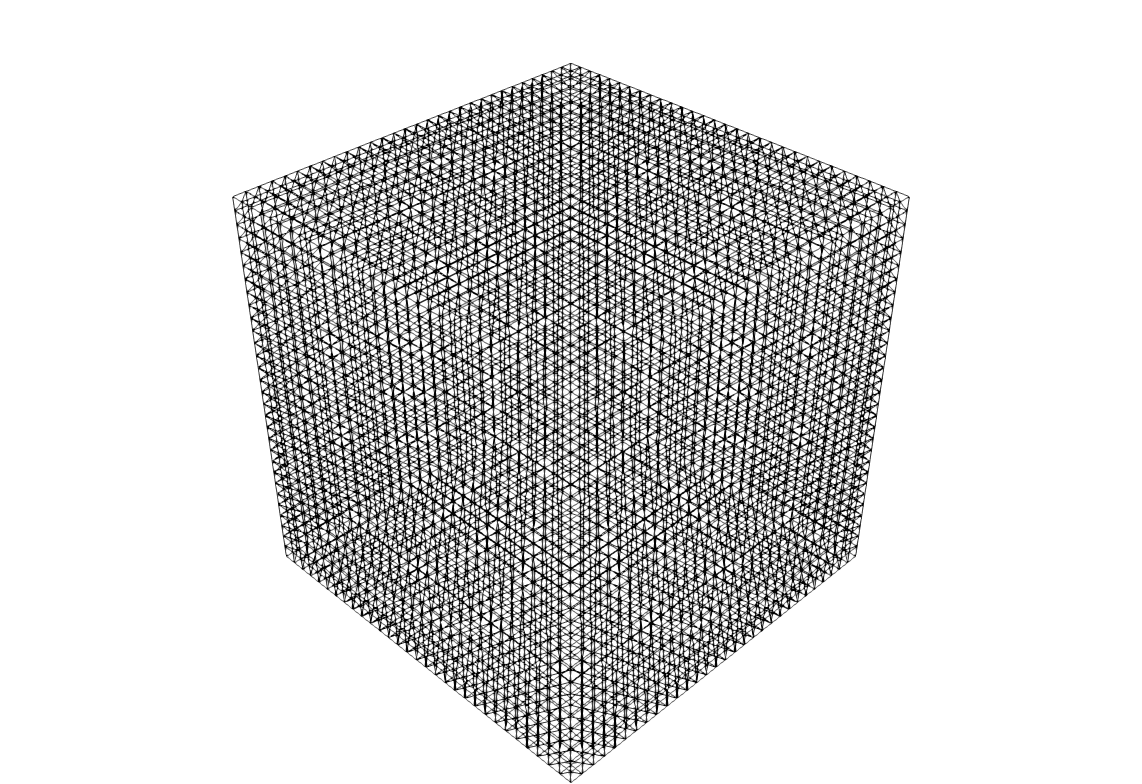}
		\caption{Mesh:  Left: Two-dimensional mesh. Right: Three-dimensional mesh.}\label{fig-mesh}
	\end{figure}

	\begin{table}[htbp]
		\centering
		\caption{Physical properties of methane.}\label{table1}
		\begin{tabular}{ccccc}
			\hline$P_c(\mathrm{bar})$ & $T_c(\mathrm{K})$ & Acentric factor & $M_w(\mathrm{g} / \mathrm{mole})$ & Temperature$(\mathrm{K})$ \\
			\hline 45.99 & 190.56 & 0.011 & 16.04 & $330$ \\
			\hline
		\end{tabular}
	\end{table}

\subsection{Example 1}

To evaluate the accuracy of the proposed numerical scheme, we conduct both temporal and spatial convergence tests. 
Since the nonlinear porous media gas flow model does not possess an analytical solution, the reference solution is obtained using the finest temporal or spatial discretization.

For the temporal refinement test, the computational domain $[0,1]\times[0,1]$ is discretized using a fixed $50\times 50$ mesh. 
The time step sizes are chosen as:
$$
\tau \in 
\left\{
9.75\times10^{-5},\ 
4.875\times10^{-5},\ 
2.4375\times10^{-5},\ 
1.21875\times10^{-5},\ 
6.09375\times10^{-6},\ 
1.5235375\times10^{-6}
\right\},
$$
and the smallest time step $\tau_{\min} = 1.5235375\times10^{-6}$ is taken as the reference temporal solution. 
The numerical errors of the molar density at a prescribed final time are computed by comparing the solutions obtained with larger time steps against the reference one.

For the spatial refinement test, the time step is fixed to $\tau = 10^{-2}.$ The spatial mesh resolutions are selected as $N \in \{10,\ 20,\ 40,\ 80,\ 160,\ 640\},$ and the finest mesh $N_{\max} = 640$ is adopted as the reference spatial solution. 

Table~\ref{table2} presents the temporal convergence results of the molar density. As the time step size $\tau$ decreases, the $L^2$-error exhibits a clear monotonic decay. The estimated convergence rates between successive refinements remain close to a first-order trend, and the fitted slope is approximately 1.45,
which demonstrates that the proposed scheme achieves the expected temporal accuracy.
Table~\ref{table3} shows the spatial convergence behavior under mesh refinement. As the mesh is refined from $h=1/10$ to $h=1/160$, the numerical errors decrease steadily. The computed rates are close to first order, and the fitted slope is approximately 1.01.
indicating that the method maintains stable and consistent spatial accuracy.
The corresponding log–log convergence plots are displayed in Figure~\ref{fig0-convergence}, where the linear decay trends further validate the temporal and spatial convergence properties of the proposed scheme. Overall, these results demonstrate that the method possesses reliable accuracy with respect to both time and space discretizations.

\begin{table}[h!]
\centering
\begin{tabular}{c c c}
\hline
Time step $\tau$ & $L^2$-Error & Rate \\
\hline
9.75e-5 & 1.330829e+00 & 1.4958 \\
4.875e-5 & 4.718754e-01 & 1.4706 \\
2.4375e-5 & 1.702673e-01 & 1.4263 \\
1.21875e-5 & 6.335198e-02 & 1.4323 \\
6.09375e-6 & 2.347375e-02 & - \\
\hline
\multicolumn{3}{c}{Fitted slope: $1.45$} \\
\hline
\end{tabular}
\caption{The time-convergence results of the molar density $c$}\label{table2}
\end{table}

\begin{table}[h!]
\centering
\begin{tabular}{c c c}
\hline
Mesh size $h$ & $L^2$-Error & Rate \\
\hline
1/10 & 9.61e-01 & 1.0099 \\
1/20 & 4.77e-01  & 0.9700  \\
1/40 & 2.44e-01 & 1.0430  \\
1/80 & 1.18e-01  & 1.0353  \\
1/160 & 5.77e-02 & - \\
\hline
\multicolumn{3}{c}{Fitted slope: $1.01$} \\
\hline
\end{tabular}
\caption{The spatial convergence results of the molar density $c$}\label{table3}
\end{table}
	\begin{figure}[htbp]
		\centering
		\includegraphics[width=5.5cm, height=4cm,trim = 0.5cm 0cm 0cm 0cm,clip]{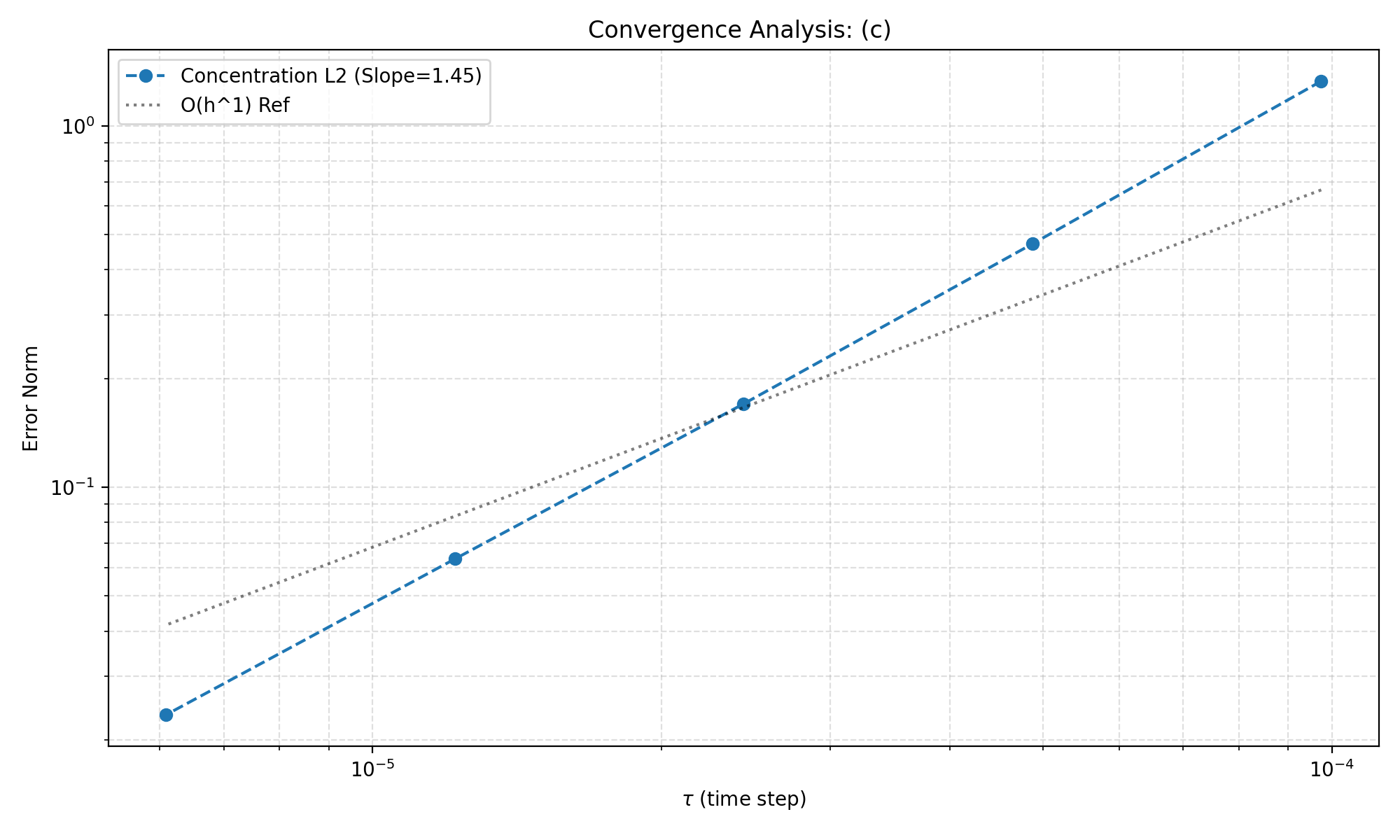}
		\includegraphics[width=5.5cm, height=4cm,trim = 0.5cm 0cm 0cm 0cm,clip]{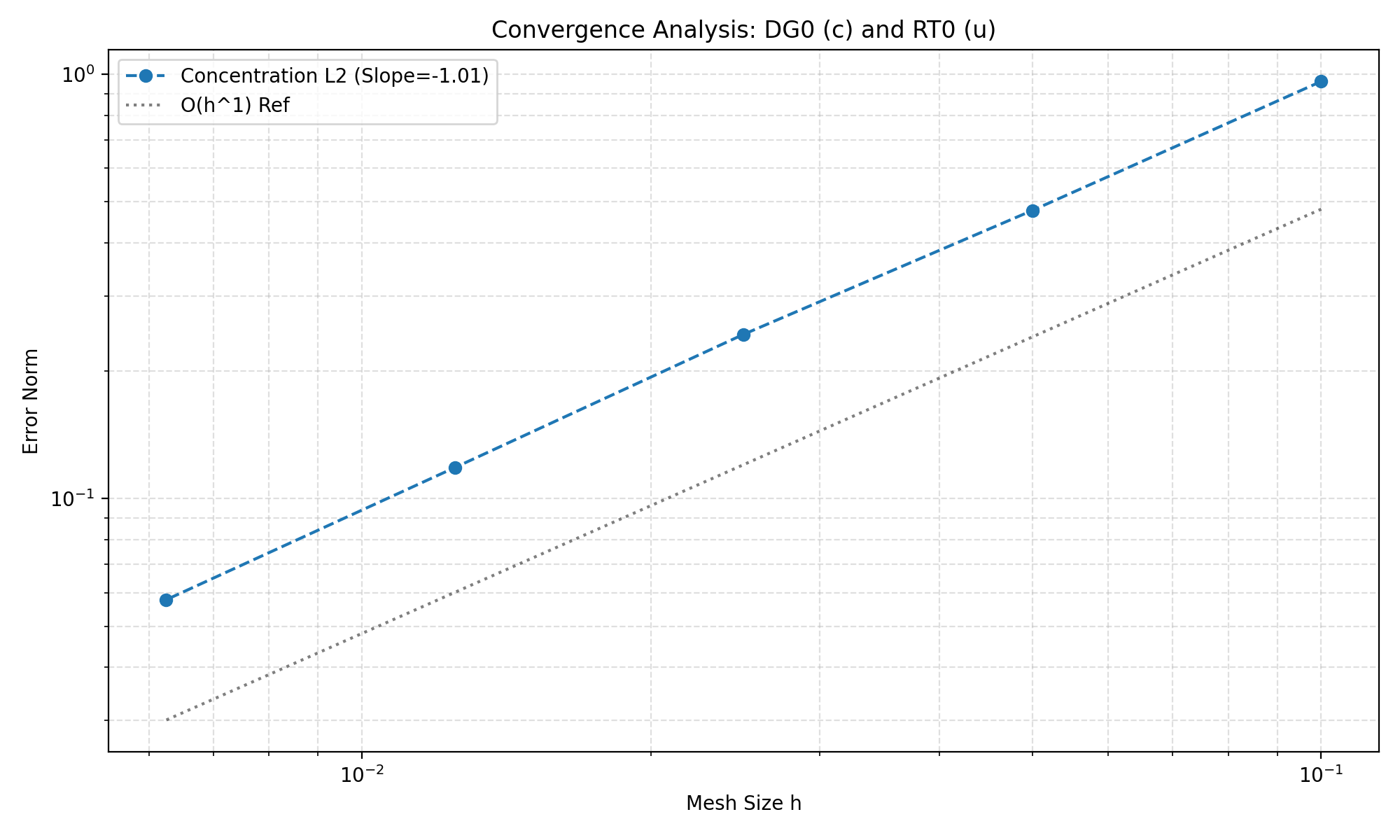}
		\caption{Example 1: Convergence plots of the molar density $c$ in space and time. Left: Spatial, Right: Temporal.} \label{fig0-convergence}
	\end{figure}

	\subsection{Example 2}
	In this example, we test a closed system to verify that the proposed numerical scheme ensures properties such as mass conservation and energy dissipation. The Perlin noise method is adopted to generate spatially varying random permeability fields, aiming to mimic realistic geological scenarios. The initial molar density is obtained using the following spatial random distribution method,
	$$
	c^0 = c_0+\operatorname{rand}(\boldsymbol{x}) \cdot\left(c_1-c_0\right),\\
	$$
	where $\operatorname{rand}(\boldsymbol{x})$ is a function for generating random numbers within the range [0,1], $c_0 = 100$ mol/$m^3$, $c_1 =300$ mol/$m^3$. In this test, we choose $N =10^{11}$ Pa, $\gamma =10^{11}$ Pa , $\eta = 10^{8}$ Pa, $\delta = 0.2$.

    To evaluate the influence of the time-step size on the performance of the nonlinear solver, 
    the iteration tolerance was set to $e = 10^{-11}$,
    and the nonlinear iterations were terminated once the residual dropped below this threshold. Five different time-step sizes were tested, and for each $\Delta t$, five consecutive time levels were computed. The number of nonlinear iterations required at each time step is summarized  in Table~\ref{table4}. These results show that the nonlinear solver is almost insensitive with the time-step size and remains robust for different choices of $\tau$.

	Figure \ref{fig1-initial} shows the initial molar density and permeability distribution. Figure \ref{fig1-mass} demonstrates that the proposed scheme preserves the dissipation property, conservation of the total moles and boundedness of the molar density. The graph on the left-hand side of Figure \ref{fig1-time} depicts the adaptive values of stabilization parameter and the graph on the right-hand side of Figure \ref{fig1-time} depicts the adaptive values of time step, which demonstrates the time step increases gradually as the system reaches steady state until the maximum set step $\tau_{max} = 1000$ is reached.
	
	In Figure \ref{fig1-c}, Figure \ref{fig1-ch}  and Figure \ref{fig1-pres}, we illustrate the distributions of molar density, chemical potential and pressure at different times. It can be seen that the molar density gradually reaches the equilibrium state driven by the chemical potential energy. 
\begin{table}[h]
    \centering
    \caption{Number of nonlinear iterations with different time steps}\label{table4}
    \label{tab:iterations_dt}
    \begin{tabular}{cccccccc}
    \hline
    $\tau$ & step\_1 & step\_2 & step\_3 & step\_4 & step\_5 \\
    \hline
    0.010000 & 7 & 7 & 6 & 6 & 6  \\
    0.005000 & 7 & 7 & 6 & 6 & 6  \\
    0.002500 & 7 & 7 & 6 & 6 & 6  \\
    0.001250 & 7 & 7 & 6 & 6 & 6  \\
    0.000625 & 7 & 7 & 6 & 6 & 6  \\
    \hline
    \end{tabular}
\end{table}
	\begin{figure}[htbp]
		\centering
		\includegraphics[width=5.5cm, height=4cm,trim = 0.5cm 0cm 0cm 0cm,clip]{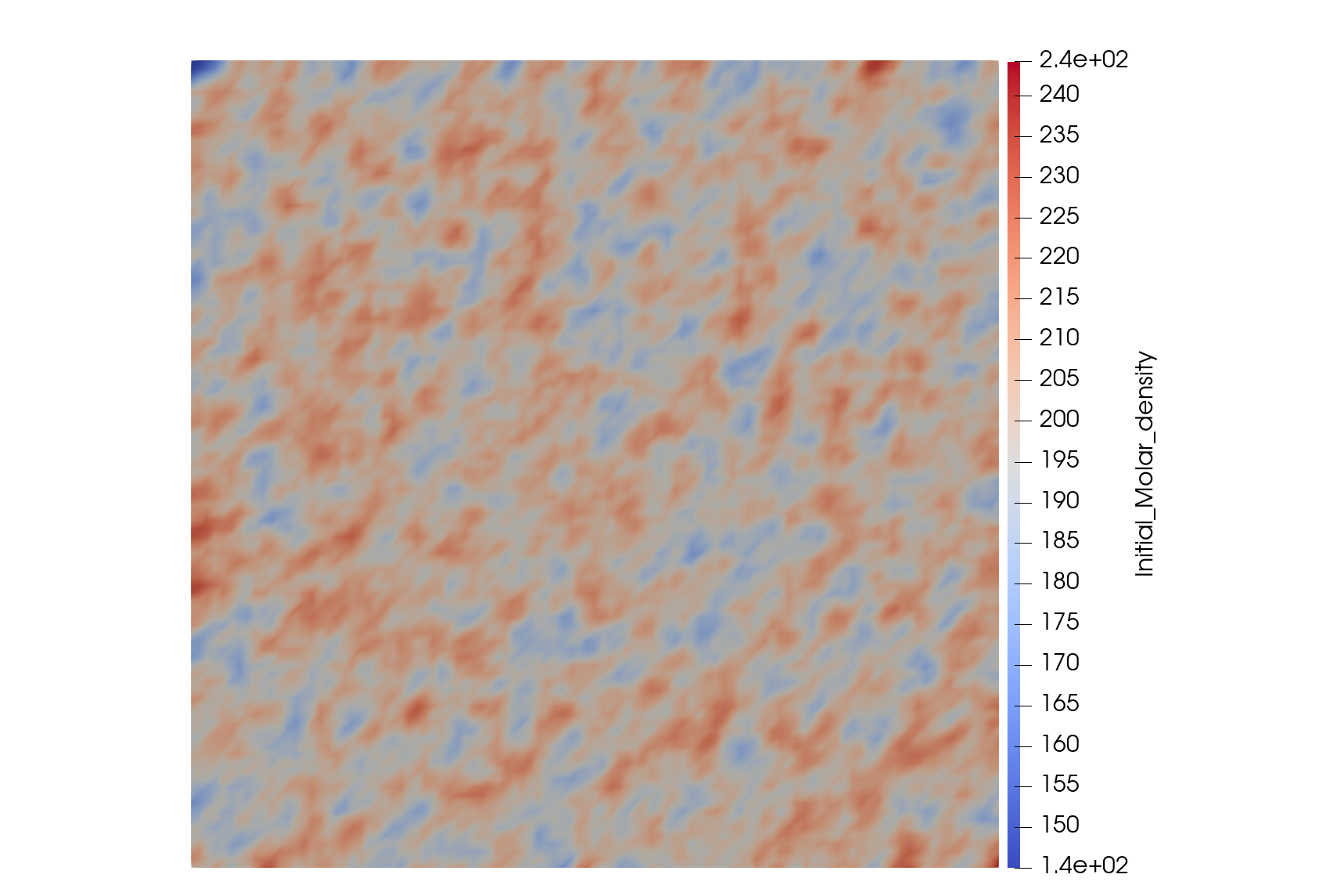}
		\includegraphics[width=5.5cm, height=4cm,trim = 0.5cm 0cm 0cm 0cm,clip]{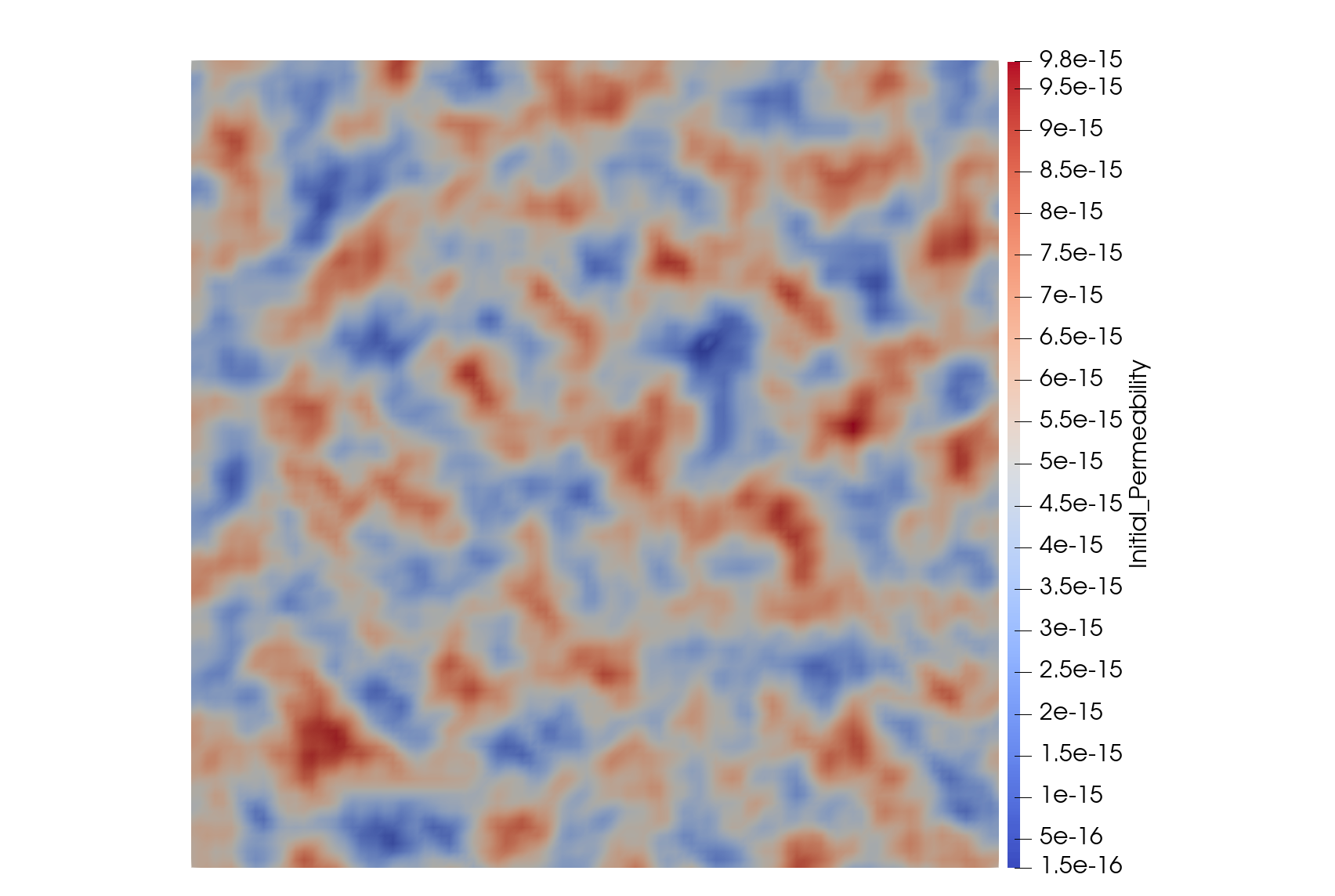}
		\caption{Example 2:  Distributions of initial molar density and permeability. Left: Initial molar density. Right: Initial permeability.}\label{fig1-initial}
	\end{figure}
	
	\begin{figure}[htbp]
		\centering
		\includegraphics[width=5.0cm, height=4.5cm]{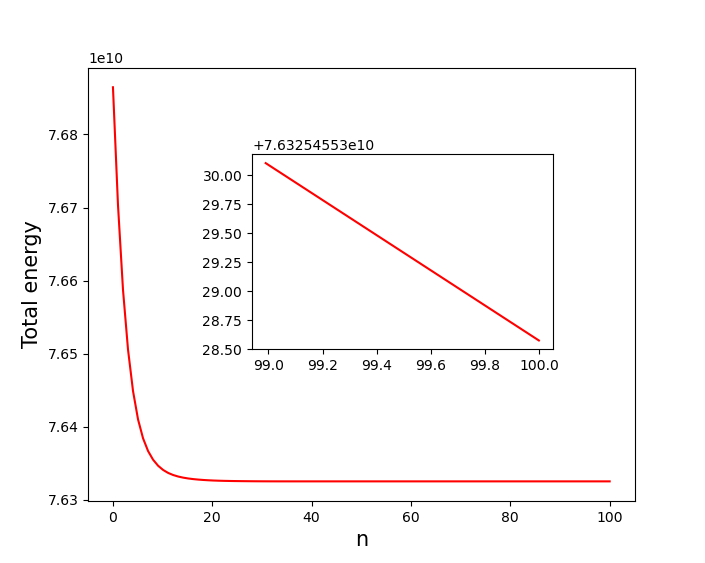}
		\includegraphics[width=5.0cm, height=4.5cm]{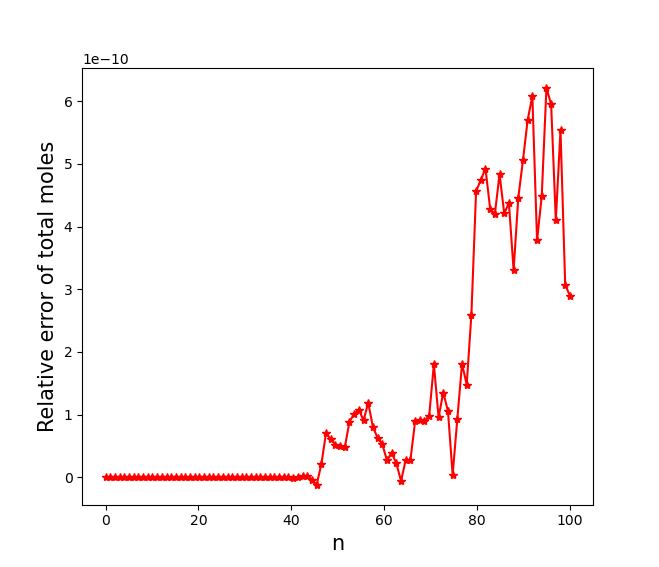}
		\includegraphics[width=5.0cm, height=4.5cm]{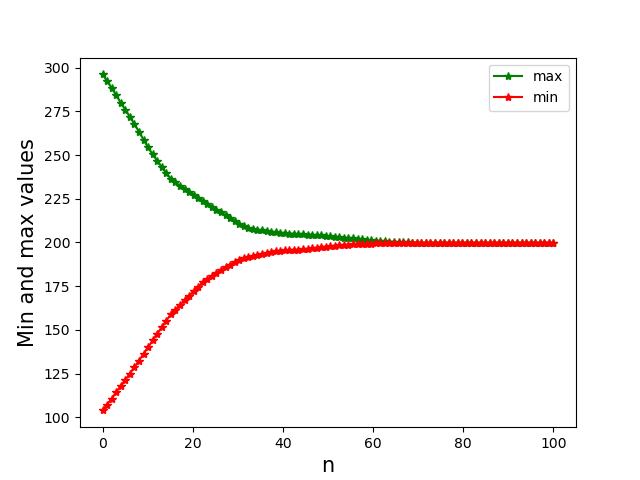}
		\caption{Example 2:  Left: Distributions of energy at different time steps. Middle: Mass conservation  at different time steps. Right: Minimum and maximum values of molar density.}\label{fig1-mass}
	\end{figure}
	\begin{figure}[htbp]
		\centering
		\includegraphics[width=5.0cm, height=4.5cm]{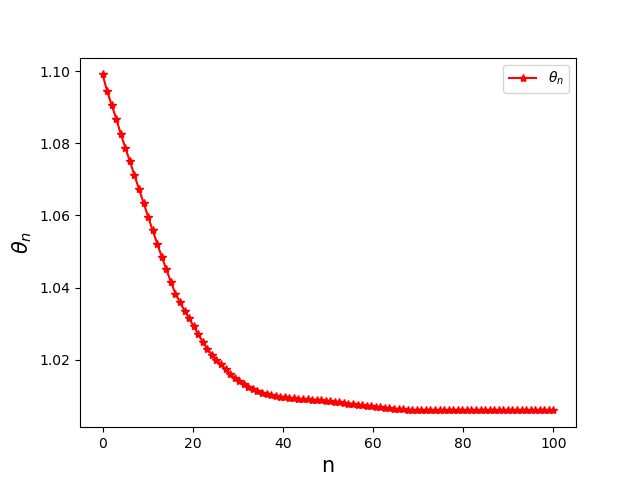}
		\includegraphics[width=5.0cm, height=4.5cm]{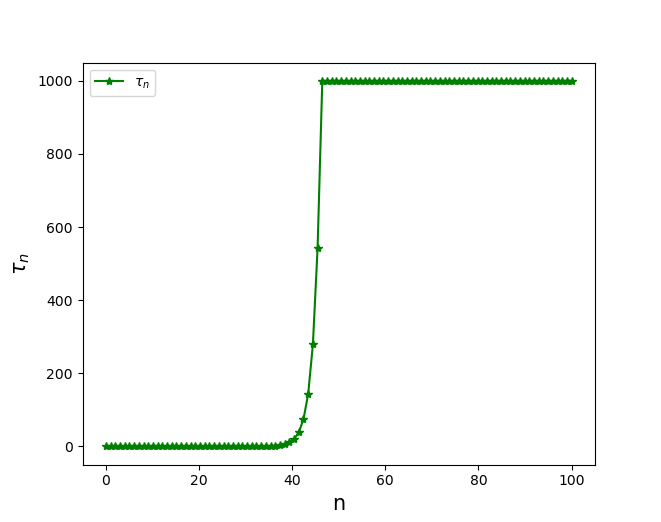}
		\caption{Example 2:   Left: Adaptive values of the stabilization parameter at different time steps. Right: Adaptive values of the time step size.}\label{fig1-time}
	\end{figure}
	\begin{figure}[htbp]
		\centering
		\includegraphics[width=5.5cm, height=4cm,trim=0.5cm 0cm 0cm 0cm,clip]{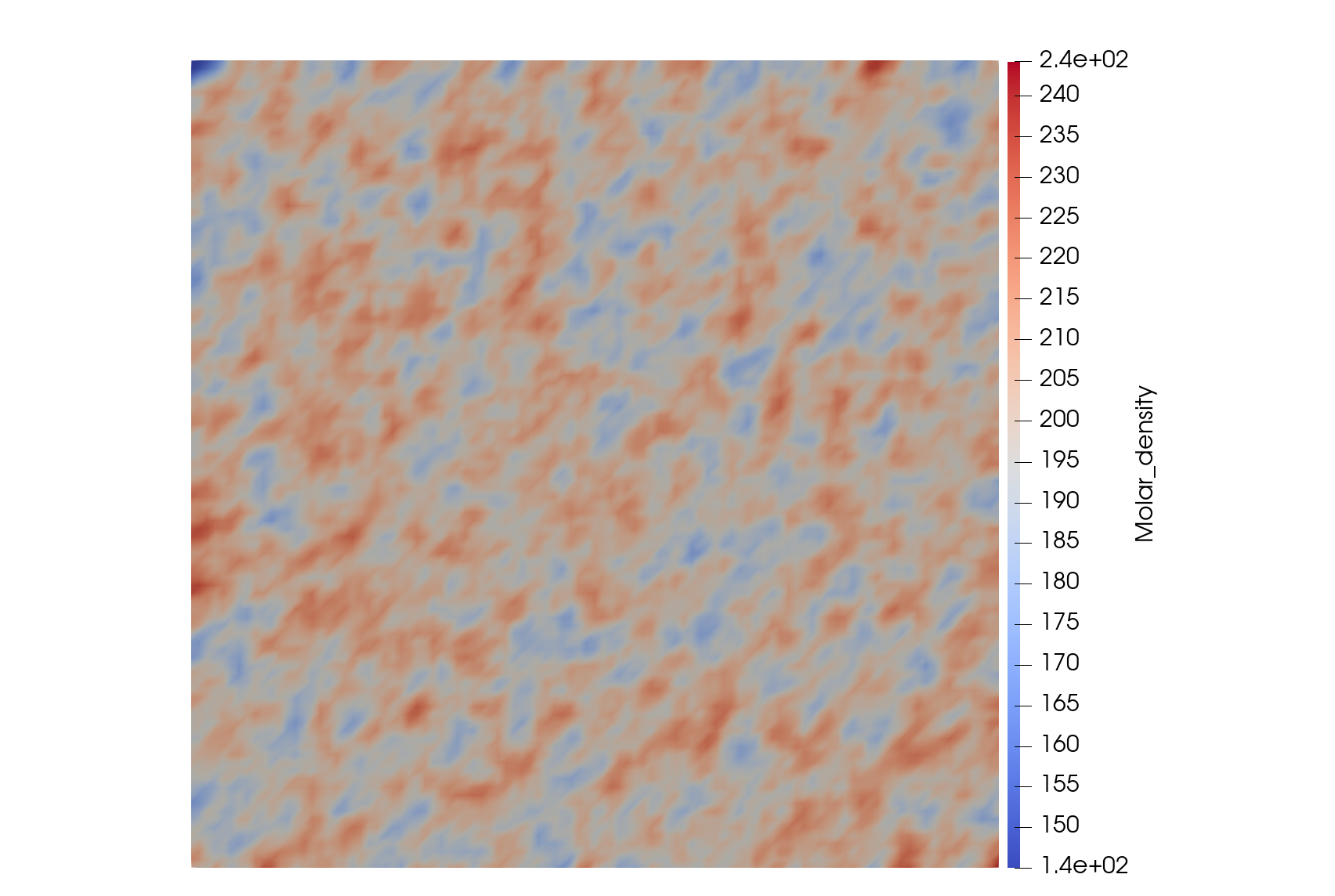}
		\includegraphics[width=5.5cm, height=4cm,trim=0.5cm 0cm 0cm 0cm,clip]{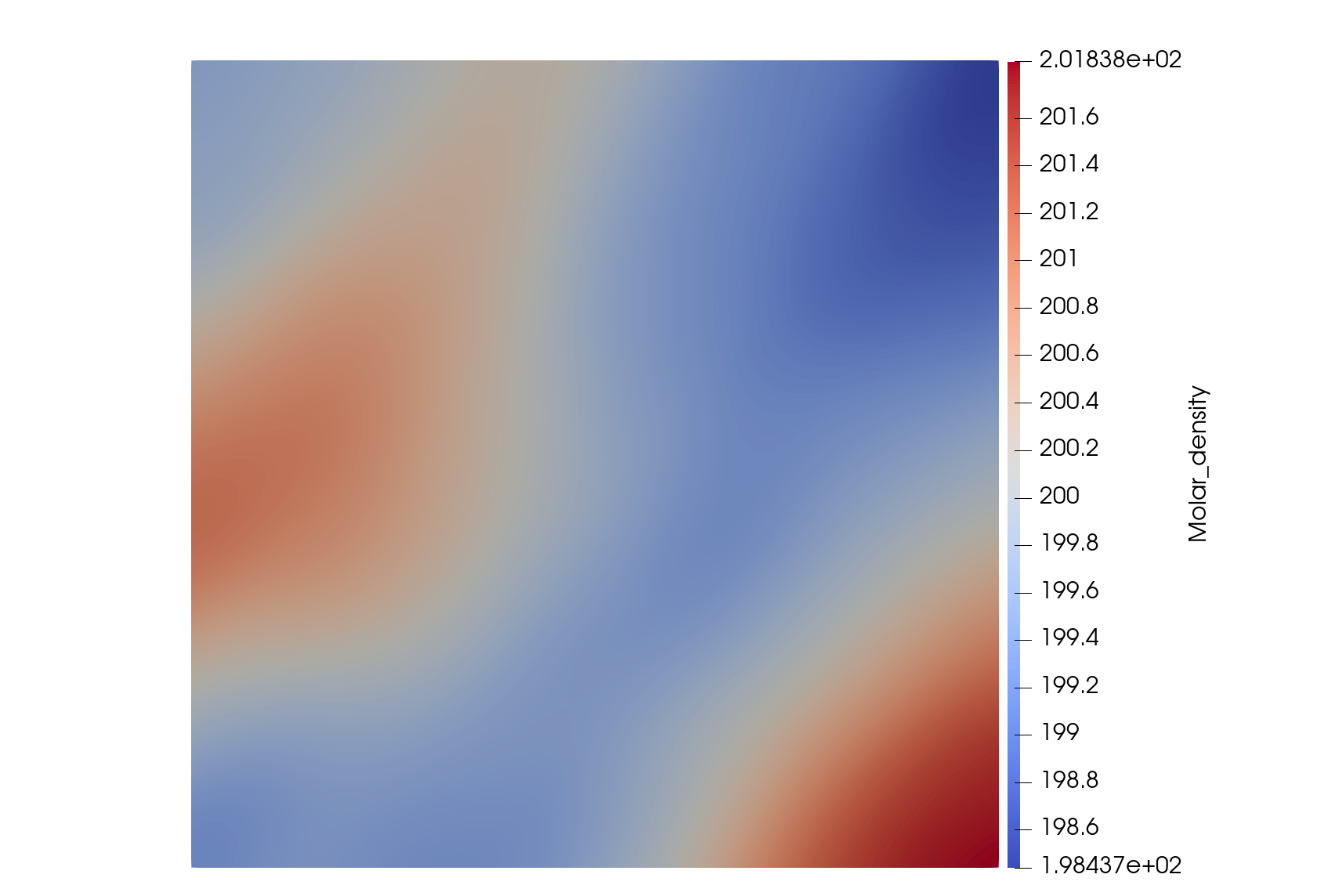}
		
		\includegraphics[width=5.5cm, height=4cm,trim=0.5cm 0cm 0cm 0cm,clip]{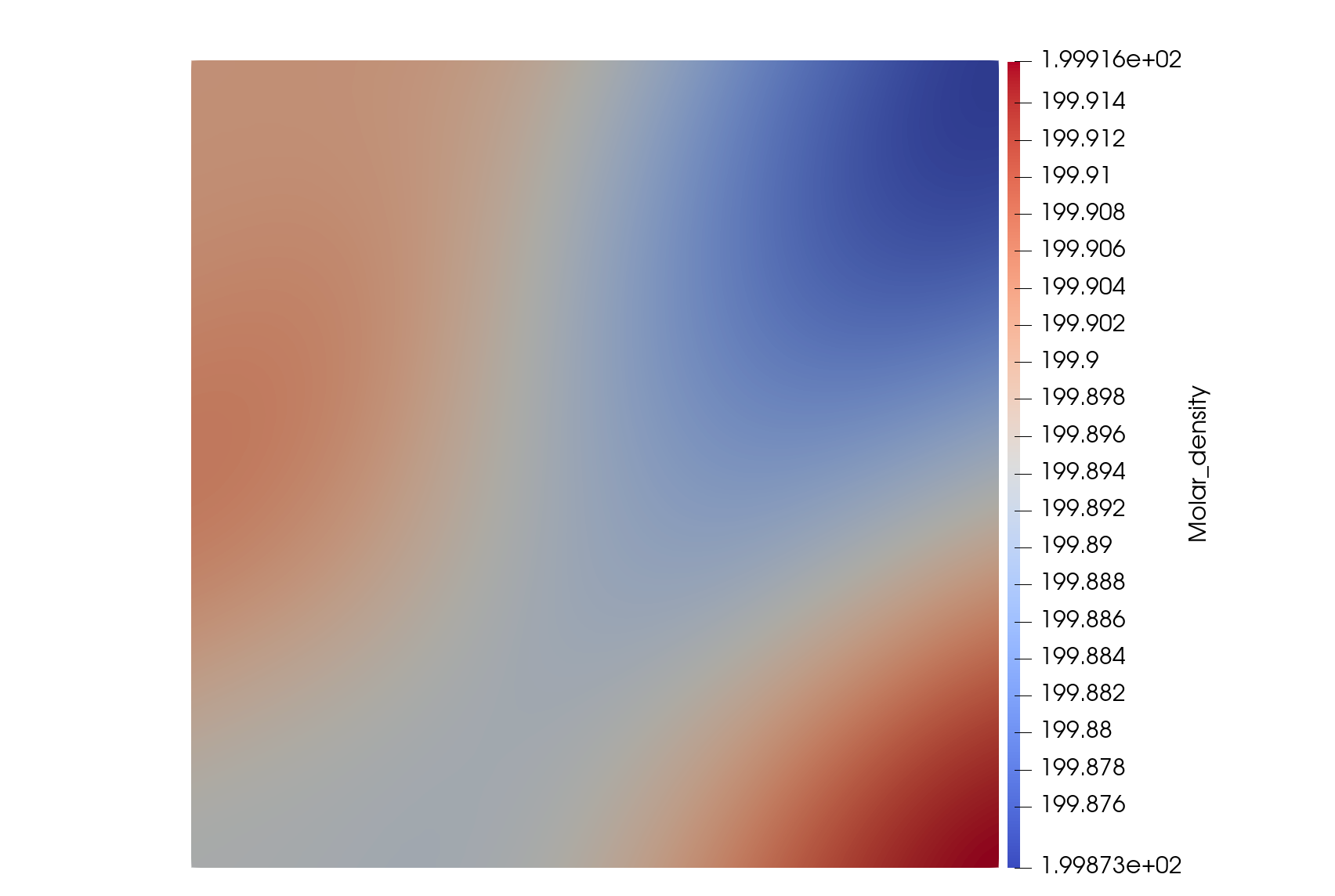}
		\includegraphics[width=5.5cm, height=4cm,trim=0.5cm 0cm 0cm 0cm,clip]{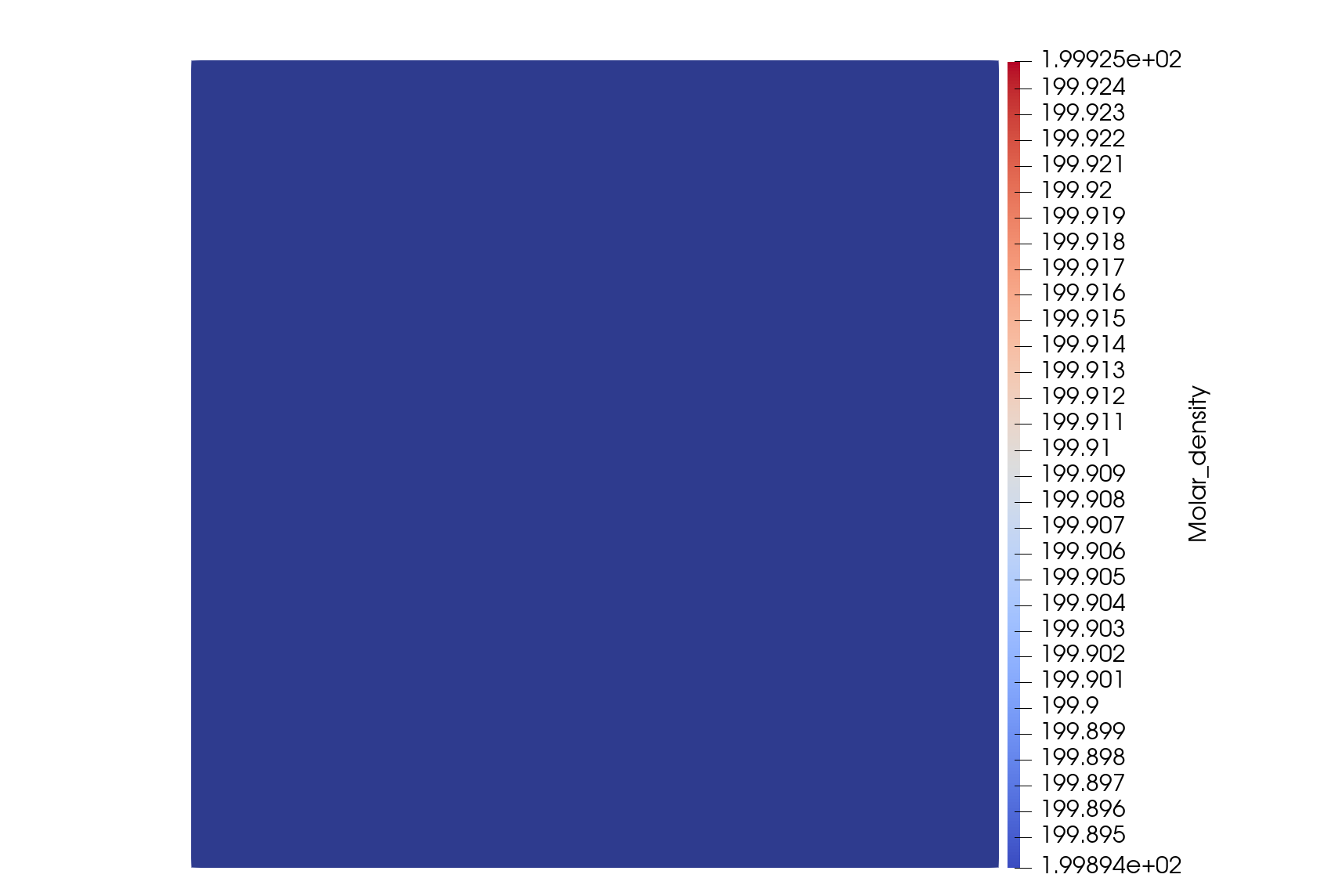}
		\caption{Distributions of molar density at different times in Example 2. Top-left: $n = 10$. Top-right: $n = 30$. Bottom-left: $n = 40$. Bottom-right: $n = 100$.}\label{fig1-c}
	\end{figure}
	
	\begin{figure}[htbp]
		\centering
		\includegraphics[width=5.5cm, height=4cm,trim=0.5cm 0cm 0cm 0cm,clip]{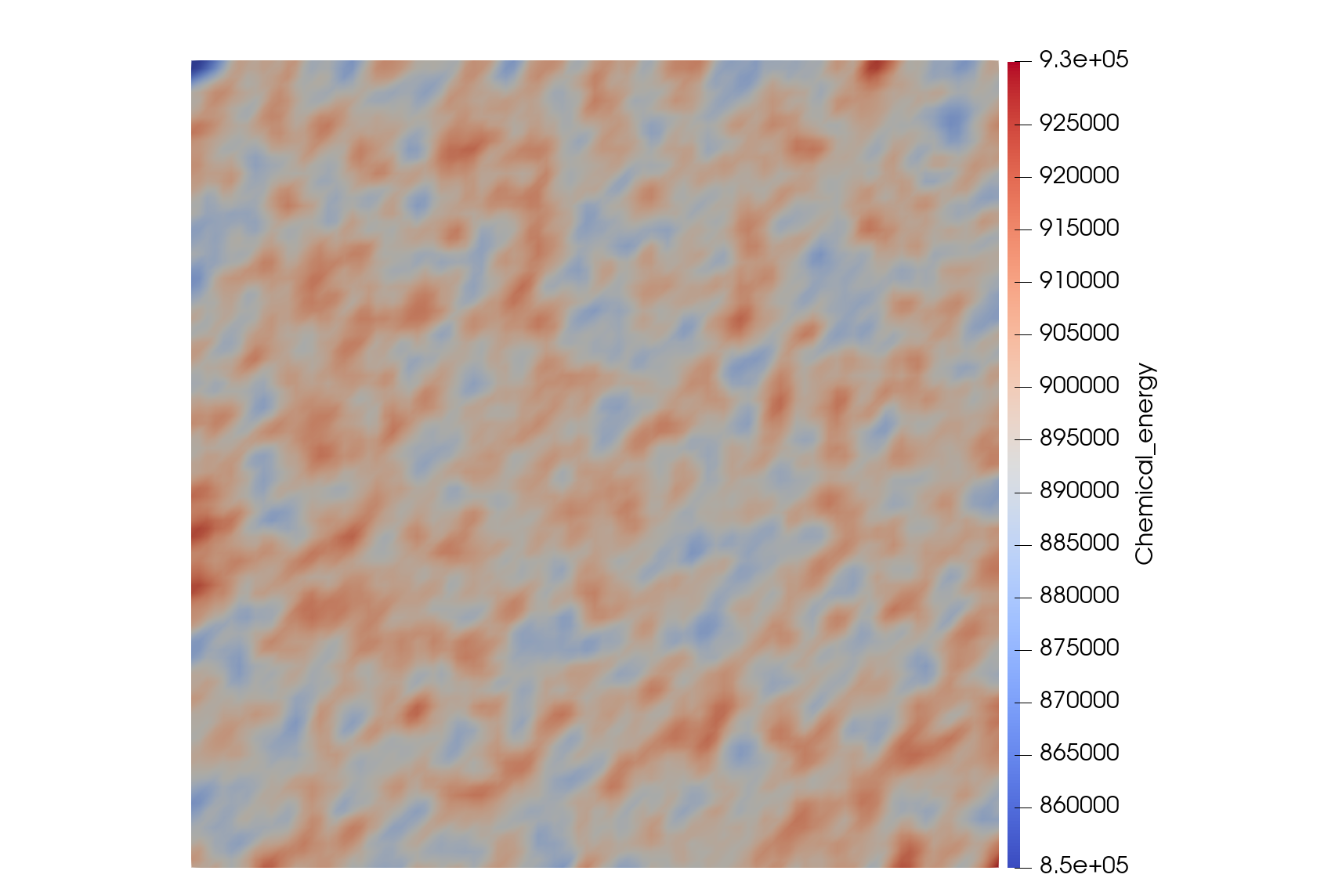}
		\includegraphics[width=5.5cm, height=4cm,trim=0.5cm 0cm 0cm 0cm,clip]{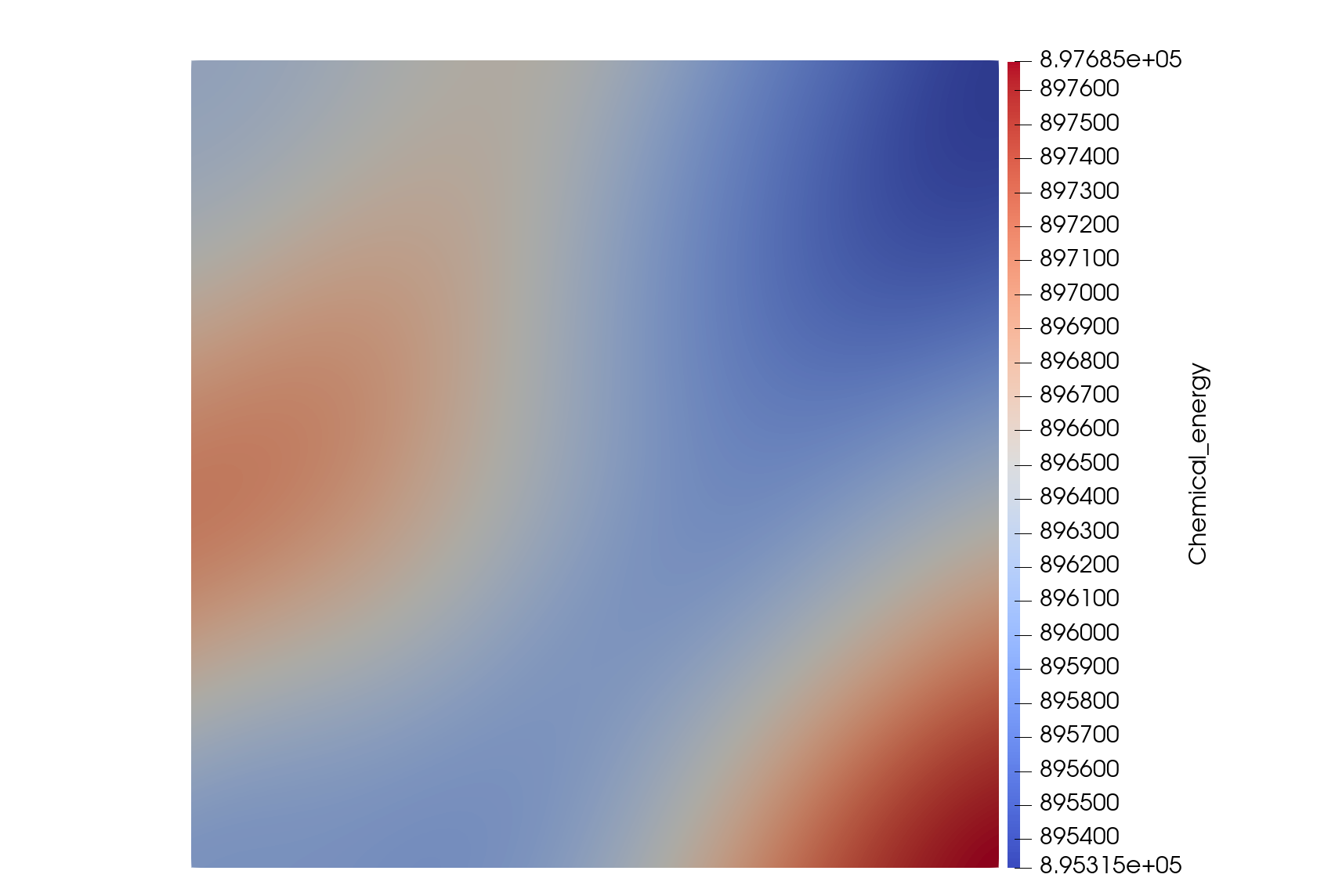}
		
		\includegraphics[width=5.5cm, height=4cm,trim=0.5cm 0cm 0cm 0cm,clip]{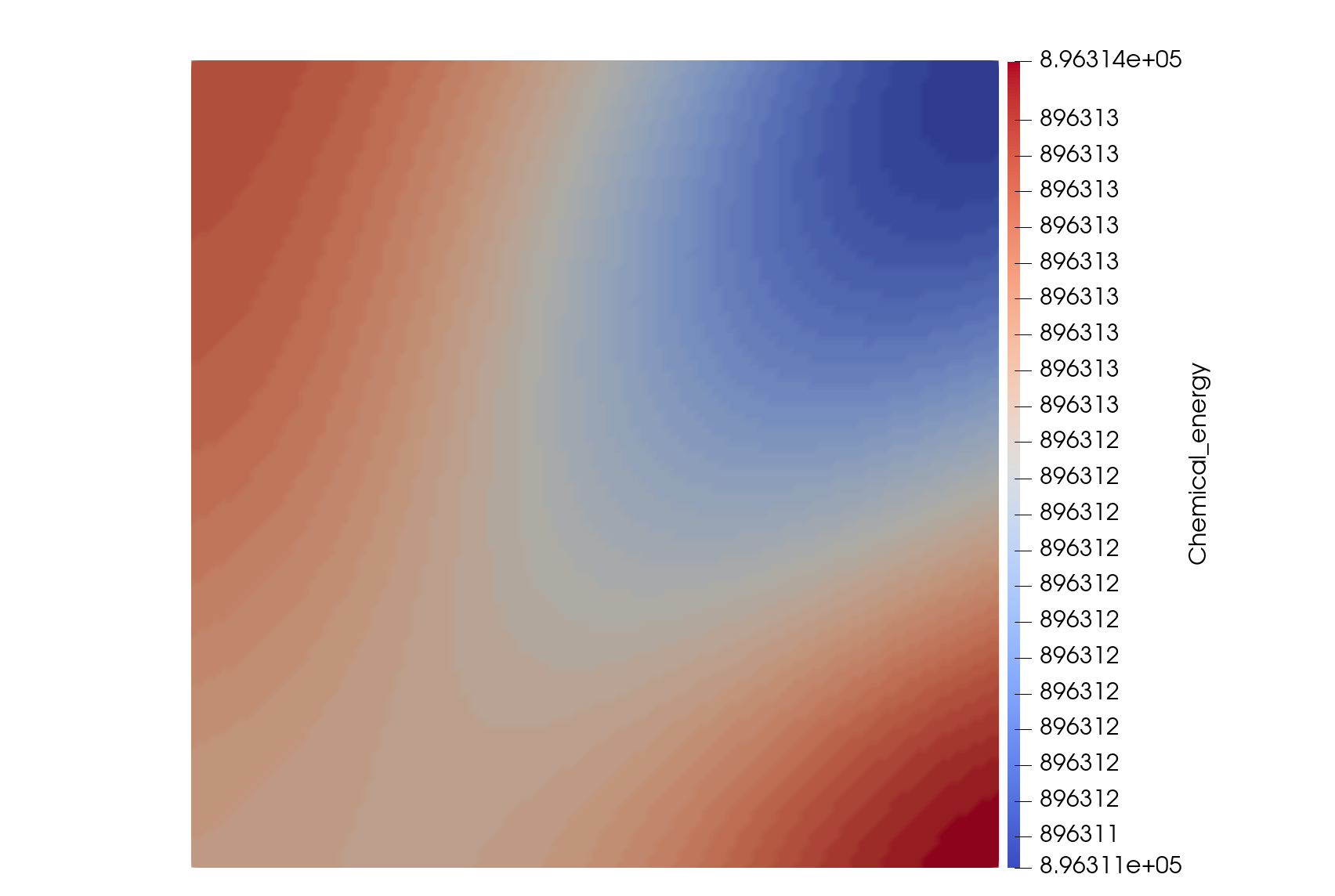}
		\includegraphics[width=5.5cm, height=4cm,trim=0.5cm 0cm 0cm 0cm,clip]{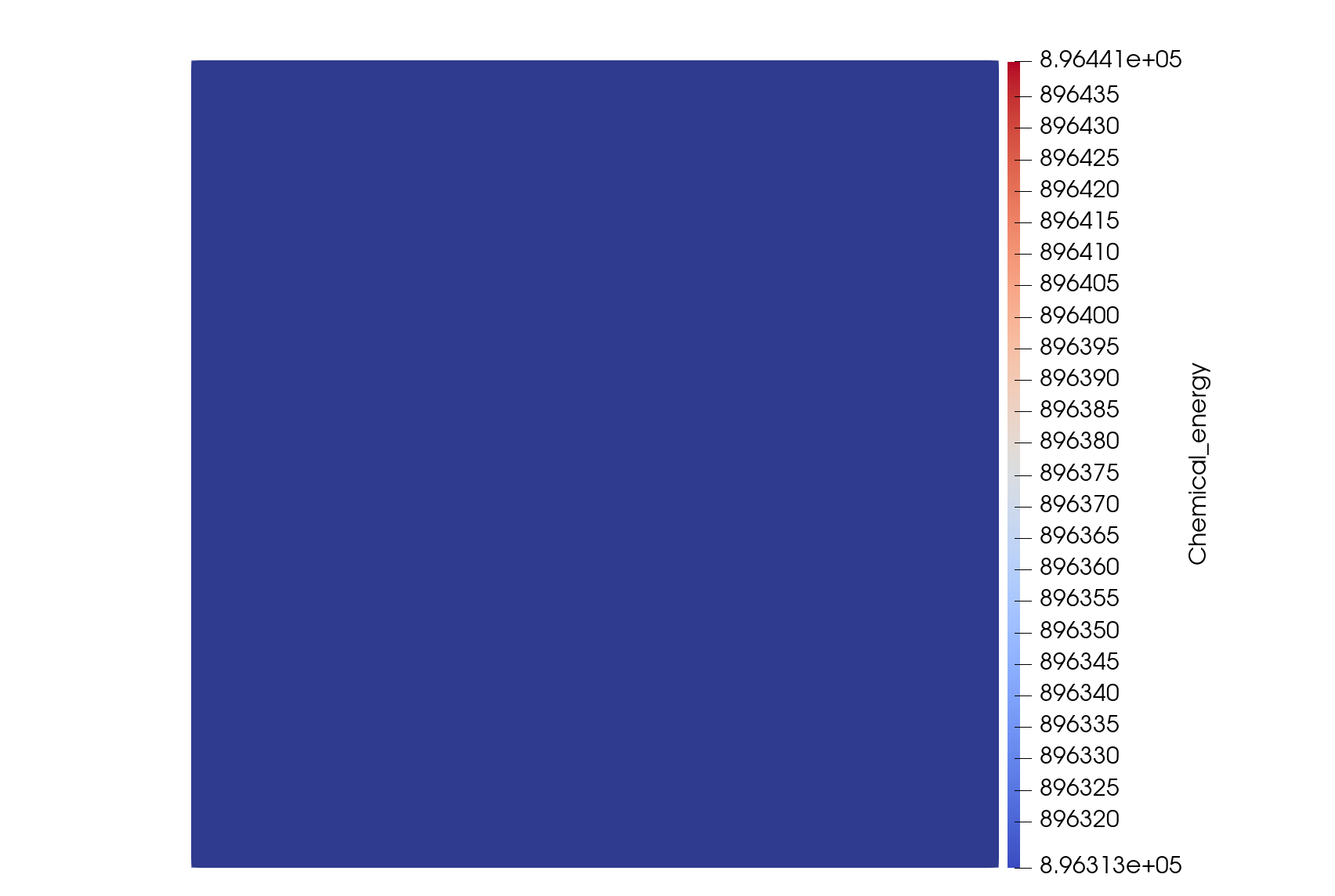}
		\caption{Distributions of chemical potential at different times in Example 2. Top-left: $n = 10$. Top-right: $n = 30$. Bottom-left: $n = 40$. Bottom-right: $n = 100$.}\label{fig1-ch}
	\end{figure}
	
	\begin{figure}[htbp]
		\centering
		\includegraphics[width=5.5cm, height=4cm,trim=0.5cm 0cm 0cm 0cm,clip]{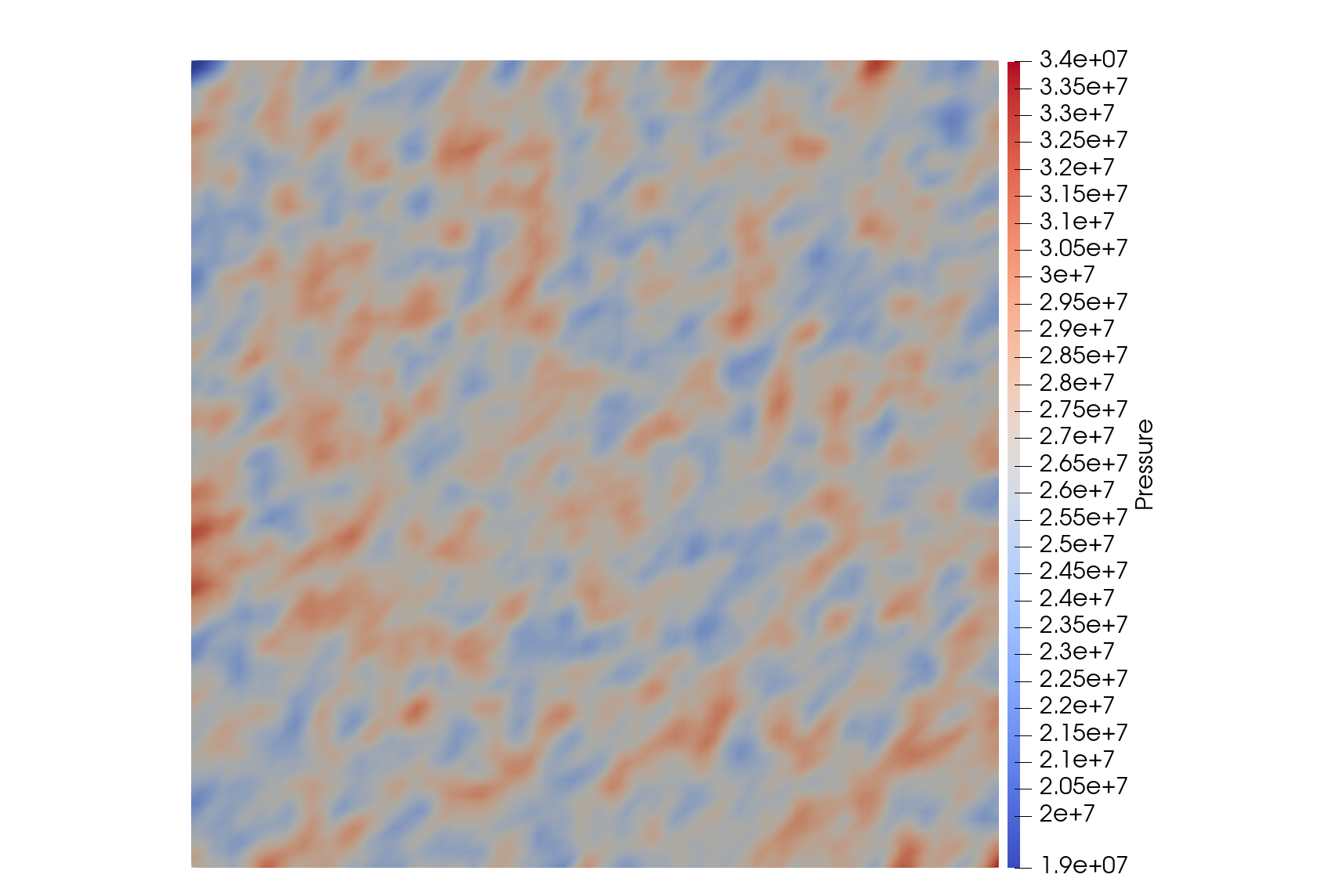}
		\includegraphics[width=5.5cm, height=4cm,trim=0.5cm 0cm 0cm 0cm,clip]{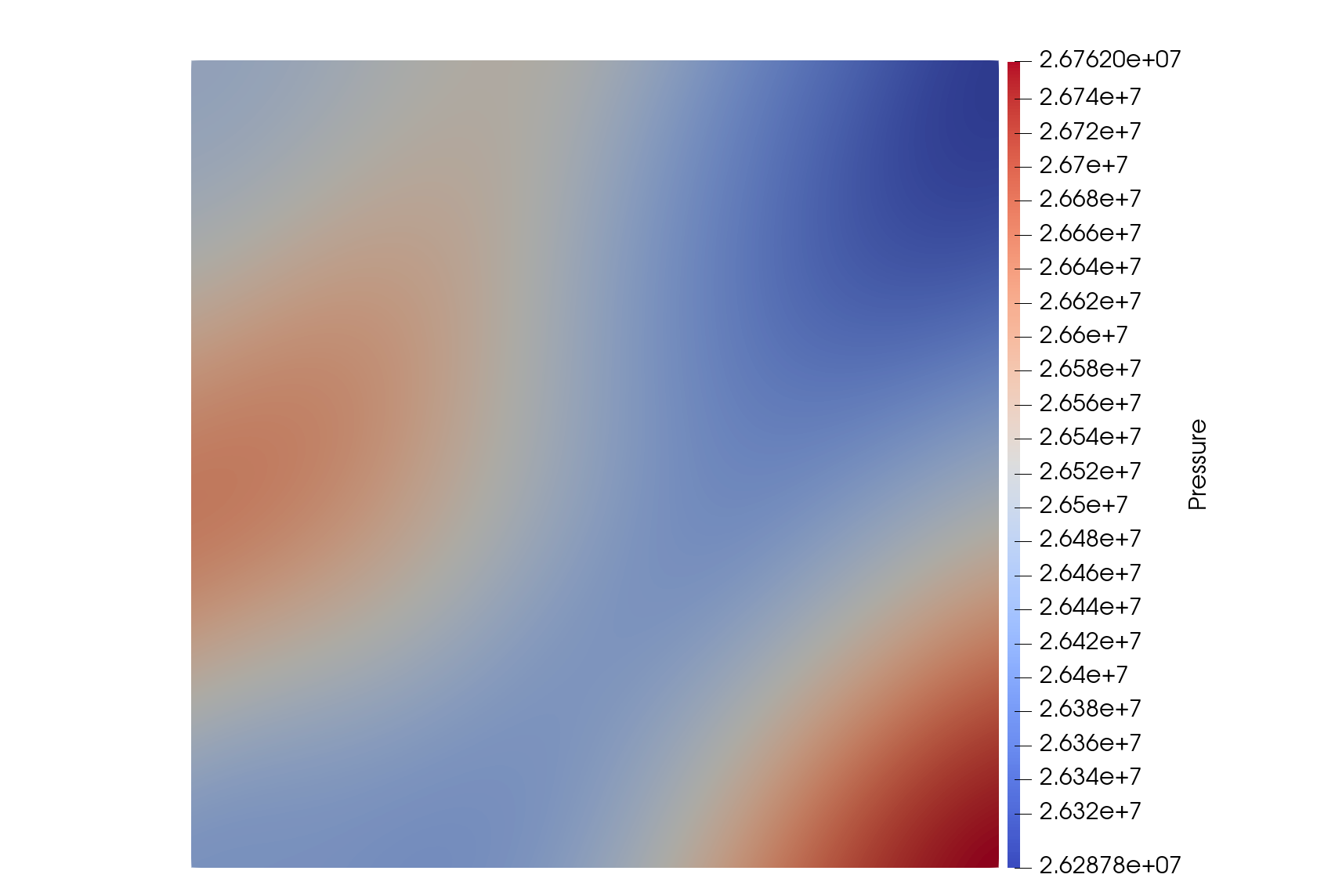}
		
		\includegraphics[width=5.5cm, height=4cm,trim=0.5cm 0cm 0cm 0cm,clip]{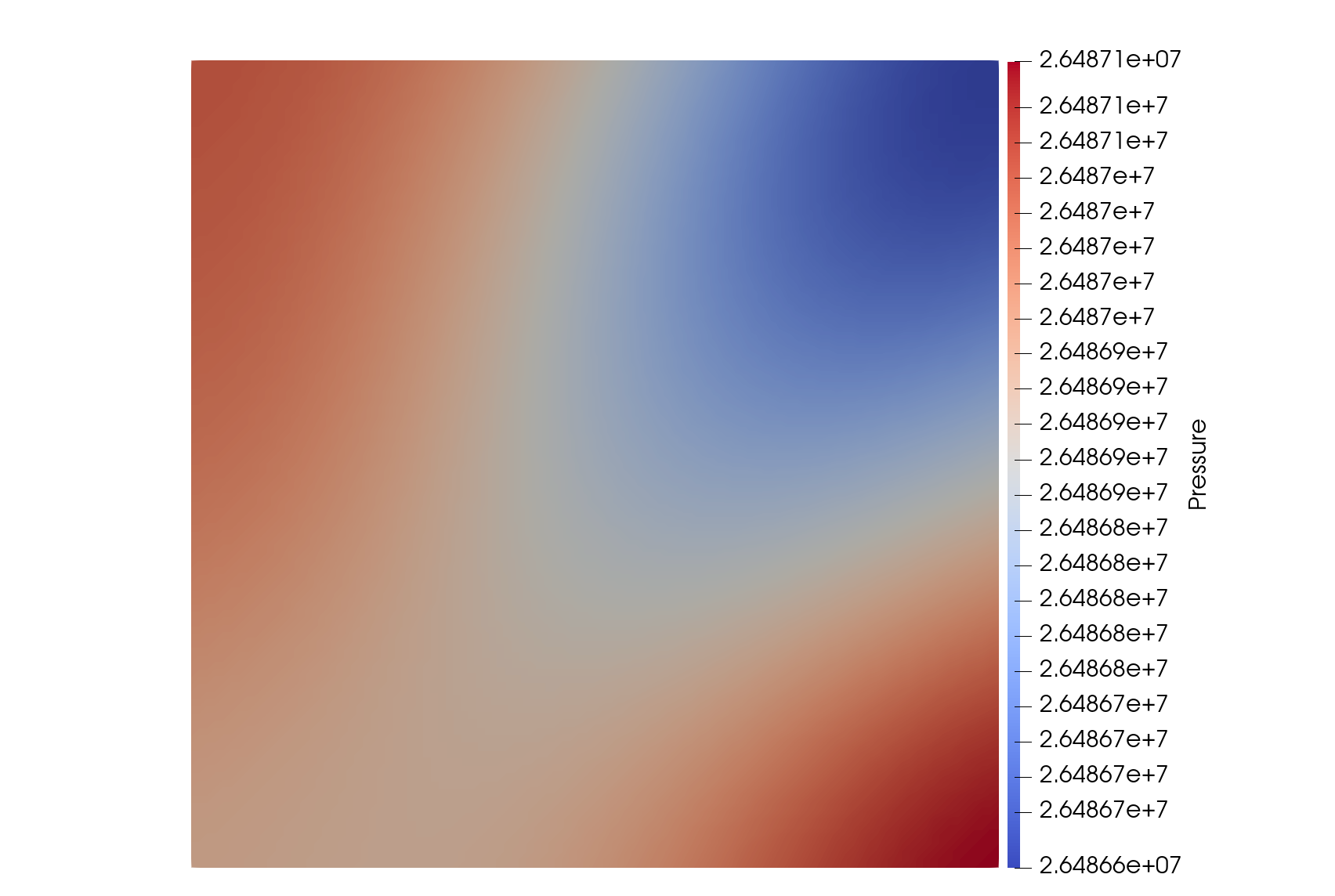}
		\includegraphics[width=5.5cm, height=4cm,trim=0.5cm 0cm 0cm 0cm,clip]{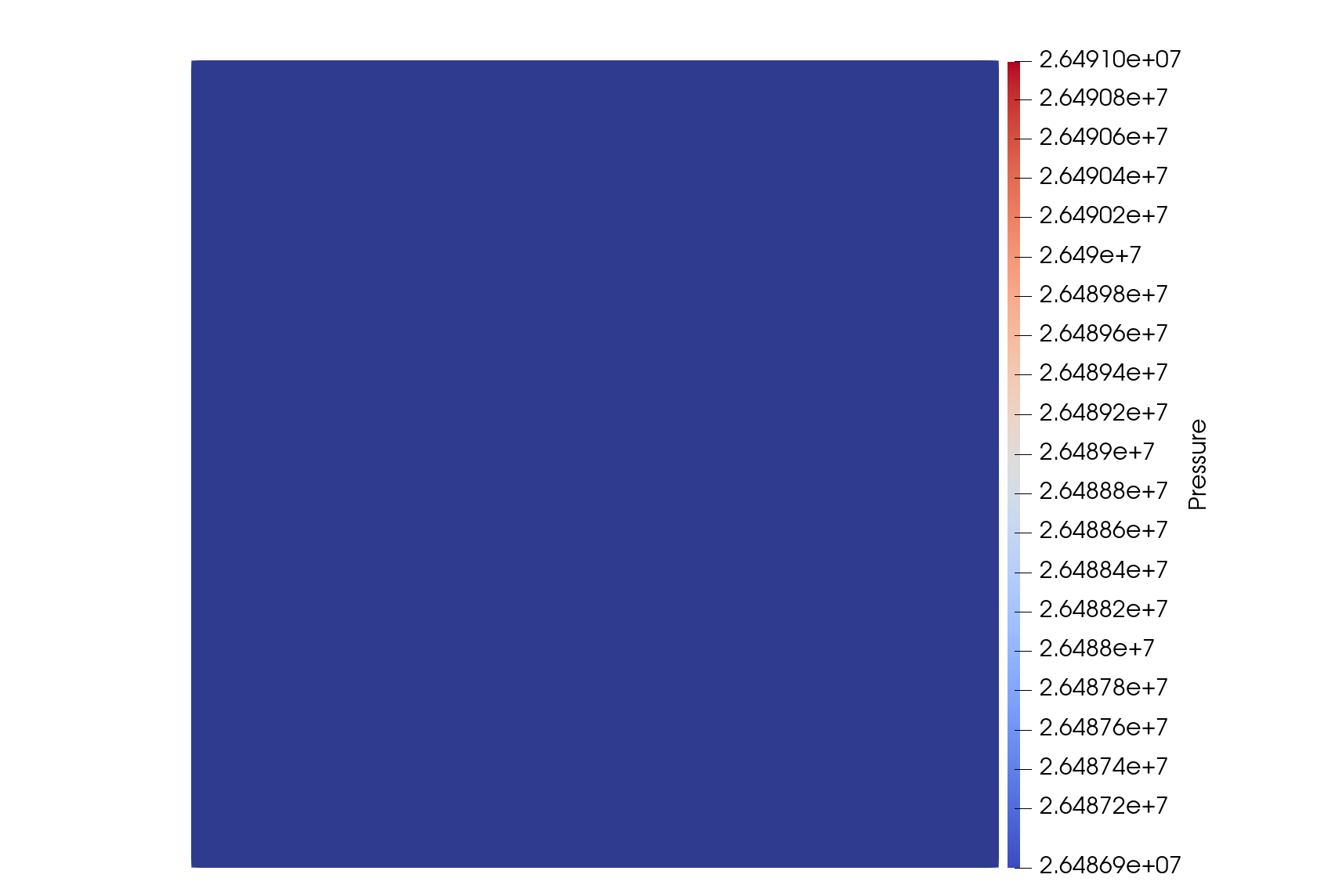}
		\caption{Distributions of pressure at different times in Example 2. Top-left: $n = 10$. Top-right: $n = 30$. Bottom-left: $n = 40$. Bottom-right: $n = 100$.}\label{fig1-pres}
	\end{figure}
	\subsection{Example 3}
	In this example, we consider an injection problem where Dirichlet boundary conditions for molar density are applied at the left boundary of the domain, with a prescribed boundary value of \(1000 \, \mathrm{mol/m^3}\). For this test, the following parameters are chosen: \(\delta = 0.8\), \(N = 10^{11} \, \mathrm{Pa}\), \(\gamma = 10^{11} \, \mathrm{Pa}\), and \(\eta = 10^{8} \, \mathrm{Pa}\). The permeability and porosity fields are assigned a random distribution, as illustrated in Figure \ref{fig2-permeability}, to simulate heterogeneous material properties commonly encountered in practical applications.
	
	Figures \ref{fig2-c}, \ref{fig2-ch}, and \ref{fig2-pres} depict the spatial variations of molar density, chemical potential, and pressure, respectively. The results demonstrate that the Dirichlet boundary condition imposed on the left boundary induces a chemical potential gradient, which drives the injection of molar density from the left to the right side of the domain. This behavior is consistent with the physical mechanism of advection under a concentration gradient. 
	
	Figure \ref{fig2-time} shows the evolution of the stabilization parameter over time steps. It is observed that the stabilization parameter remains nearly constant throughout the time iteration process. This behavior can be attributed to the fixed concentration boundary condition applied at the left end of the domain.
	
	These results highlight the effectiveness of the proposed method in capturing the coupled dynamics of molar density, chemical potential, and pressure in a heterogeneous medium, while maintaining numerical stability and accuracy. The consistent performance of the stabilization parameter underscores the reliability of the approach for long-time simulations.
	\begin{figure}[htbp]
		\centering
		\includegraphics[width=5.5cm, height=4cm,trim=0.5cm 0cm 0cm 0cm,clip]{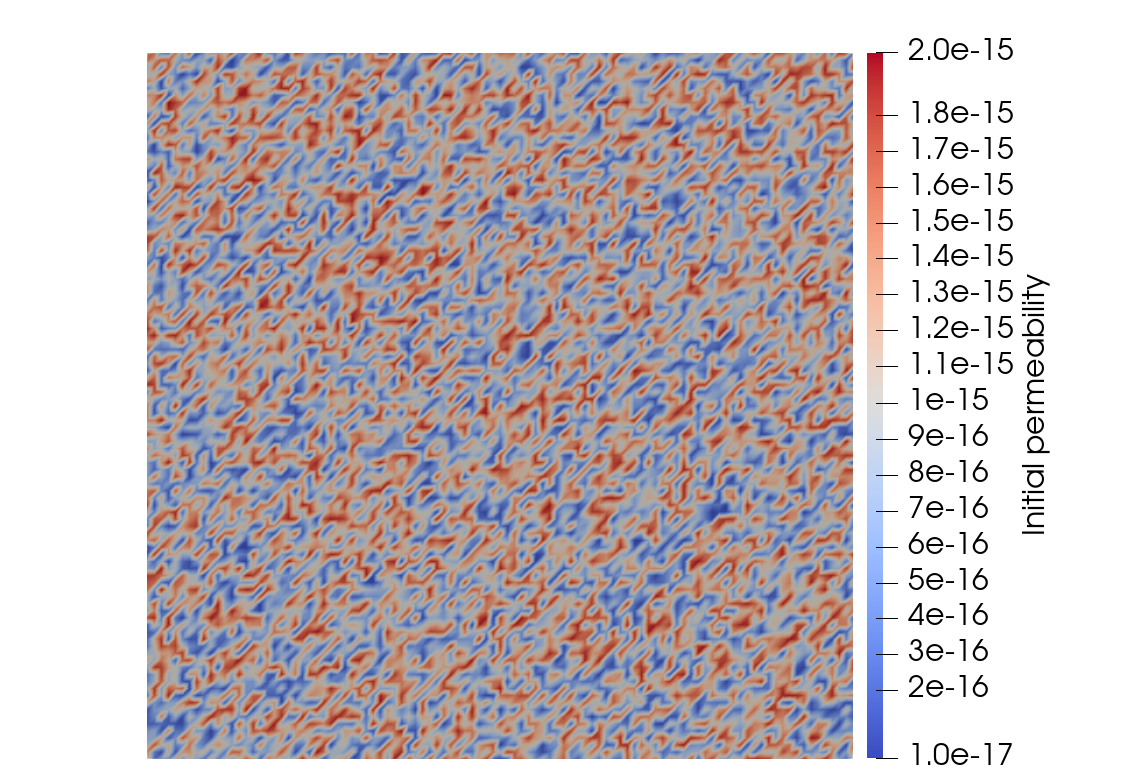}
		\includegraphics[width=5.5cm, height=4cm,trim=0.5cm 0cm 0cm 0cm,clip]{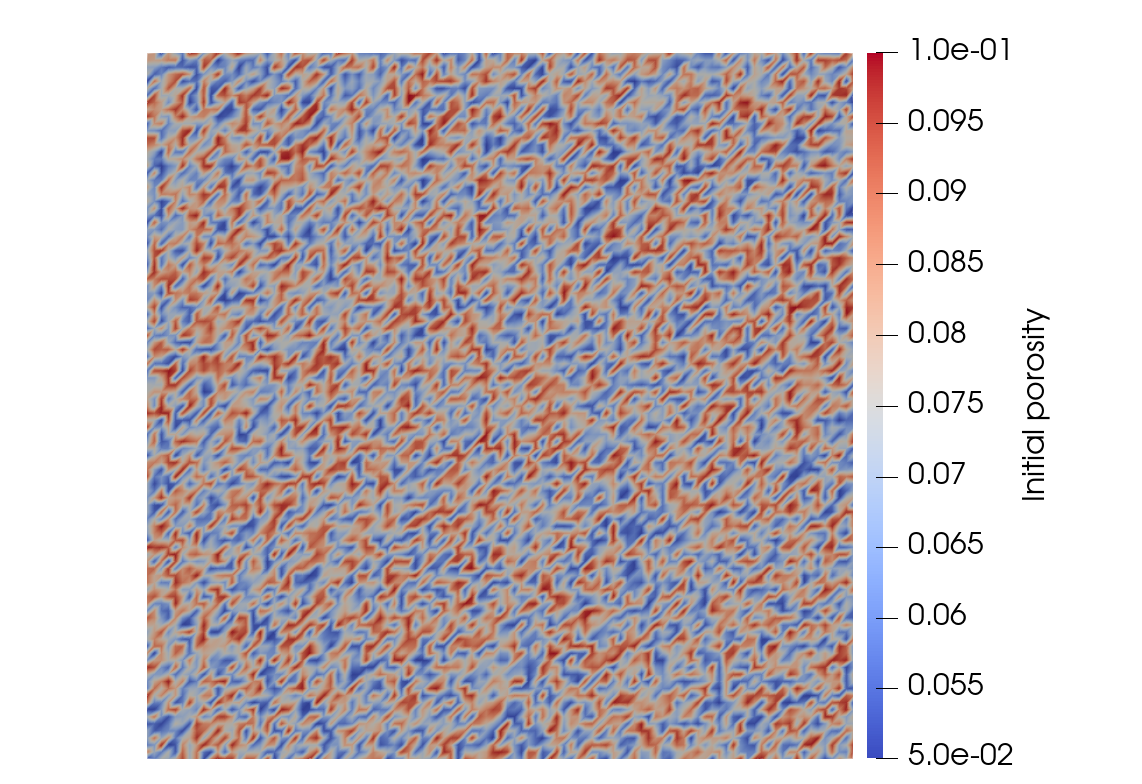}
		\caption{Example 3:   Left: The initial distribution of permeability. Right: The initial distribution of porosity.}\label{fig2-permeability}
	\end{figure}
	\begin{figure}[htbp]
		\centering
		\includegraphics[width=5.0cm, height=4.5cm]{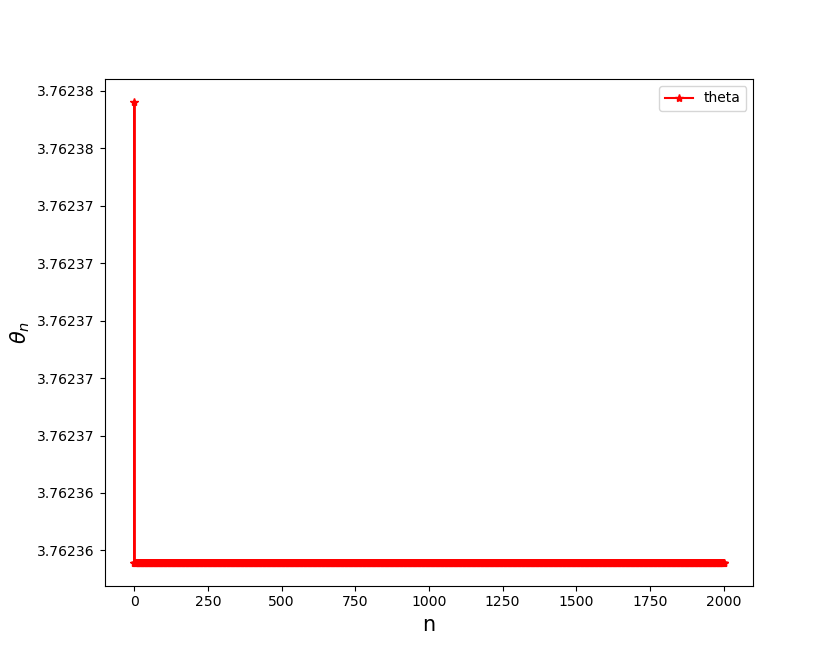}
		\includegraphics[width=5.0cm, height=4.5cm]{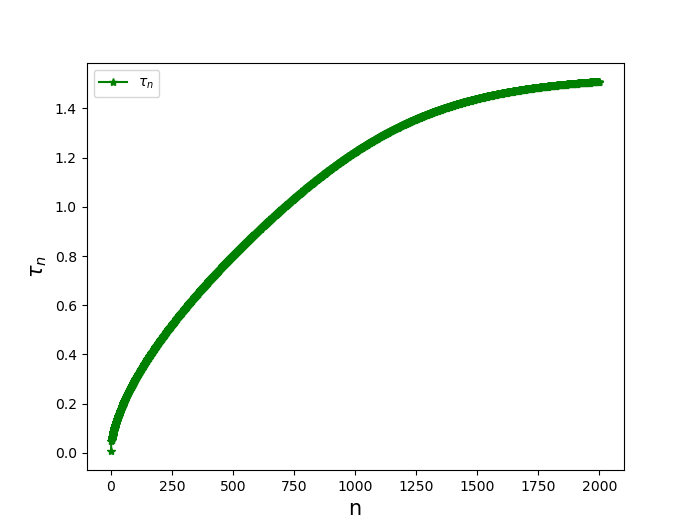}
		\caption{Example 3:   Left: Adaptive values of the stabilization parameter at different time steps. Right: Adaptive values of the time step size.}\label{fig2-time}
	\end{figure}
	\begin{figure}[htbp]
		\centering
		\includegraphics[width=5.5cm, height=4cm,trim=0.5cm 0cm 0cm 0cm,clip]{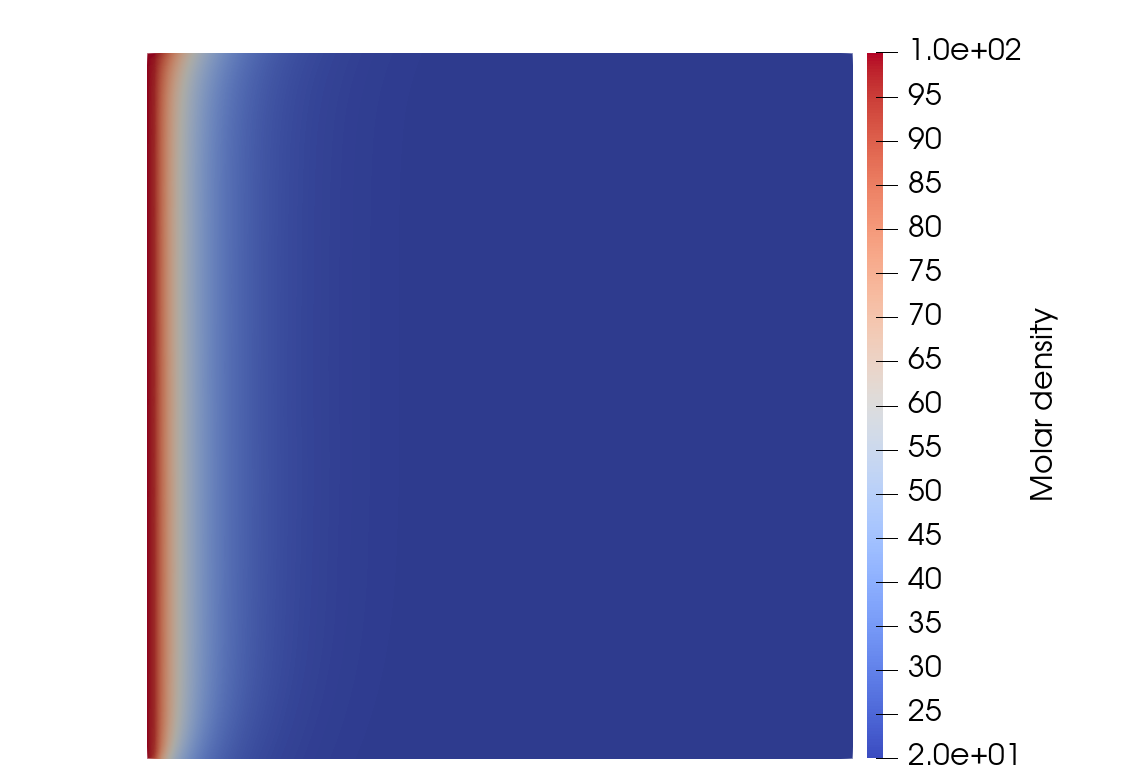}
		\includegraphics[width=5.5cm, height=4cm,trim=0.5cm 0cm 0cm 0cm,clip]{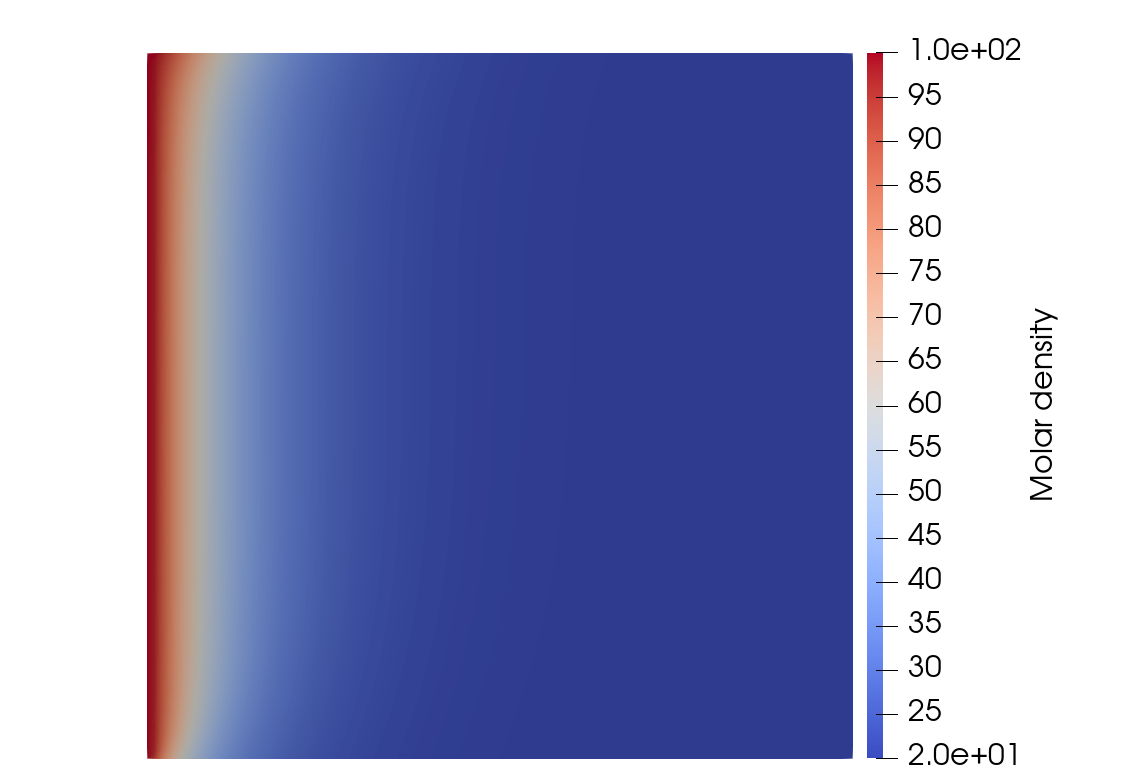}
		
		\includegraphics[width=5.5cm, height=4cm,trim=0.5cm 0cm 0cm 0cm,clip]{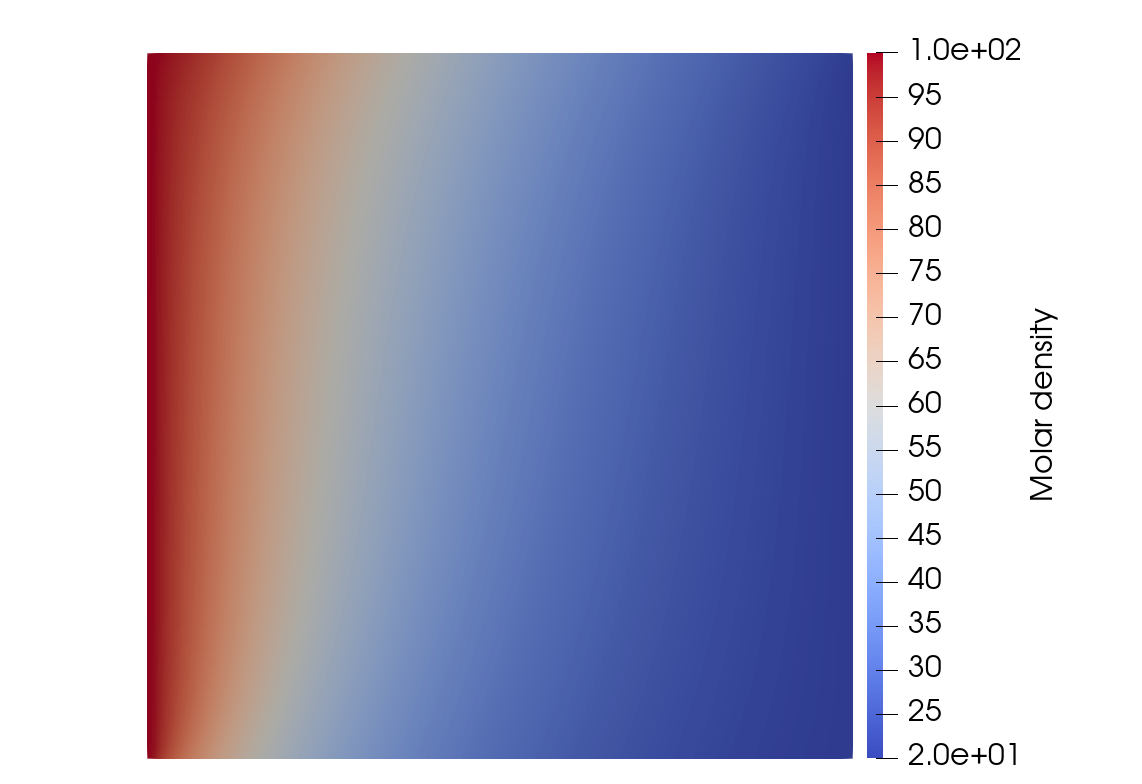}
		\includegraphics[width=5.5cm, height=4cm,trim=0.5cm 0cm 0cm 0cm,clip]{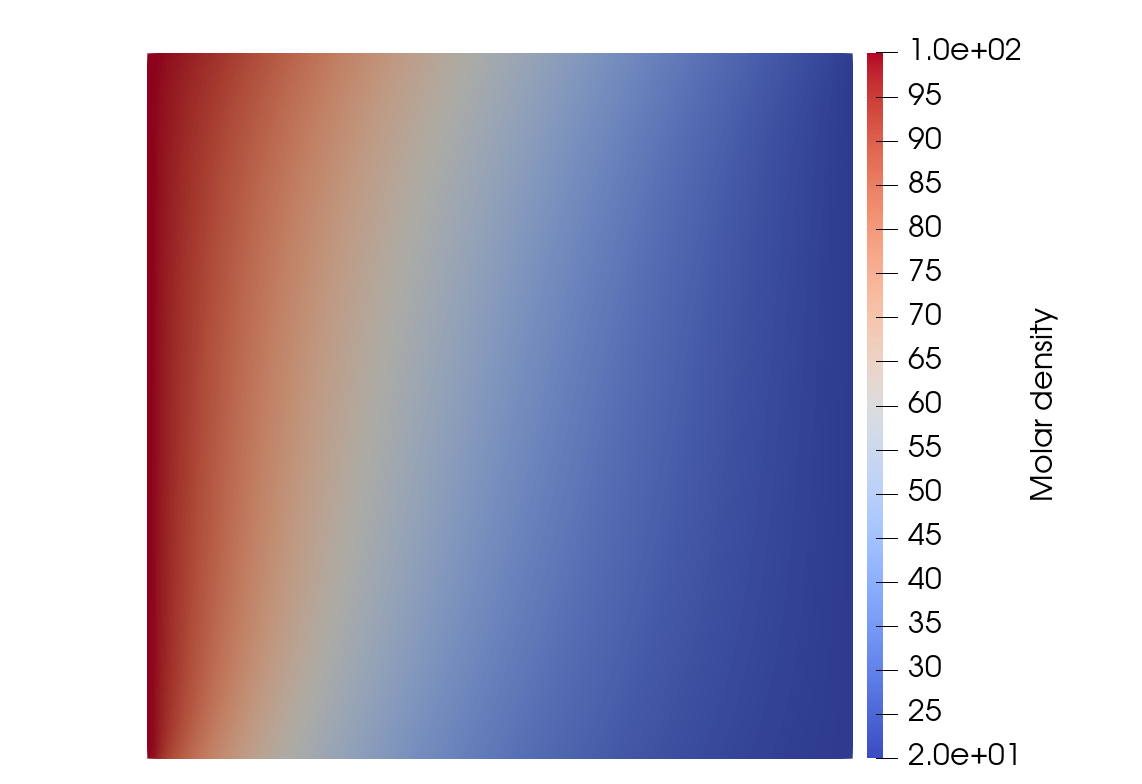}
		\caption{Distributions of molar density at different times in Example 3. Top-left: $n = 100$. Top-right: $n = 200$. Bottom-left: $n = 1000$. Bottom-right: $n = 2000$.}\label{fig2-c}
	\end{figure}
	
	\begin{figure}[htbp]
		\centering
		\includegraphics[width=5.5cm, height=4cm,trim=0.5cm 0cm 0cm 0cm,clip]{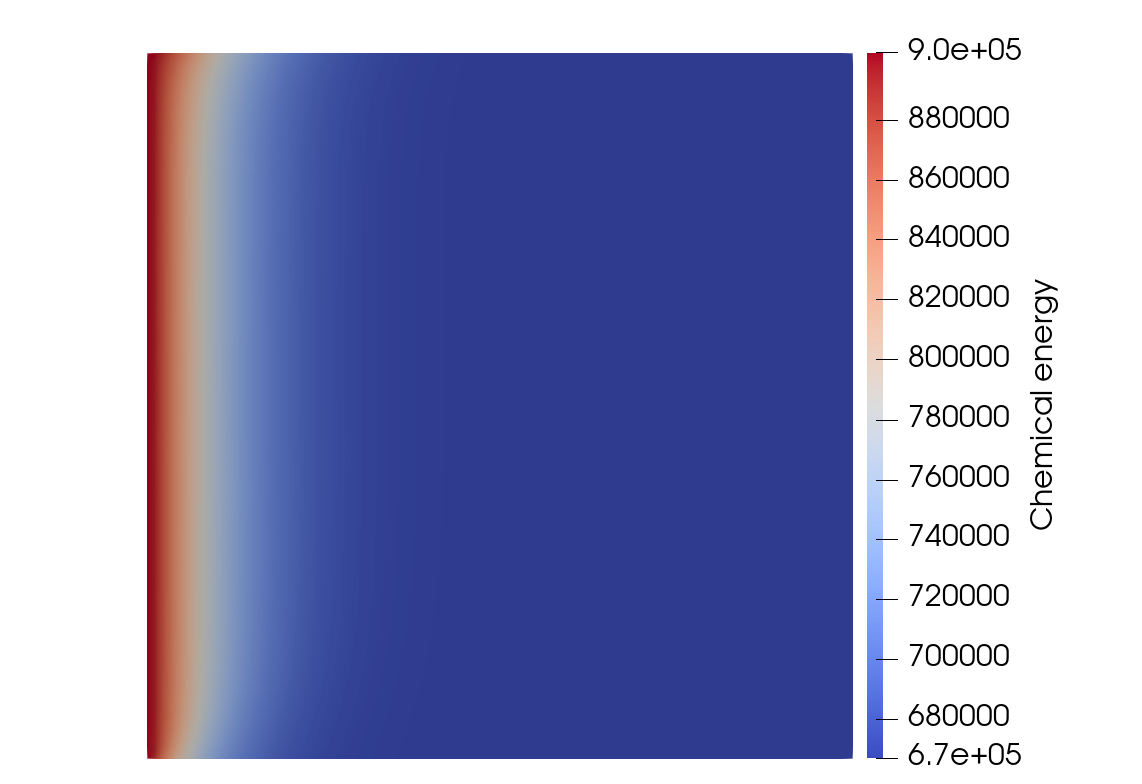}
		\includegraphics[width=5.5cm, height=4cm,trim=0.5cm 0cm 0cm 0cm,clip]{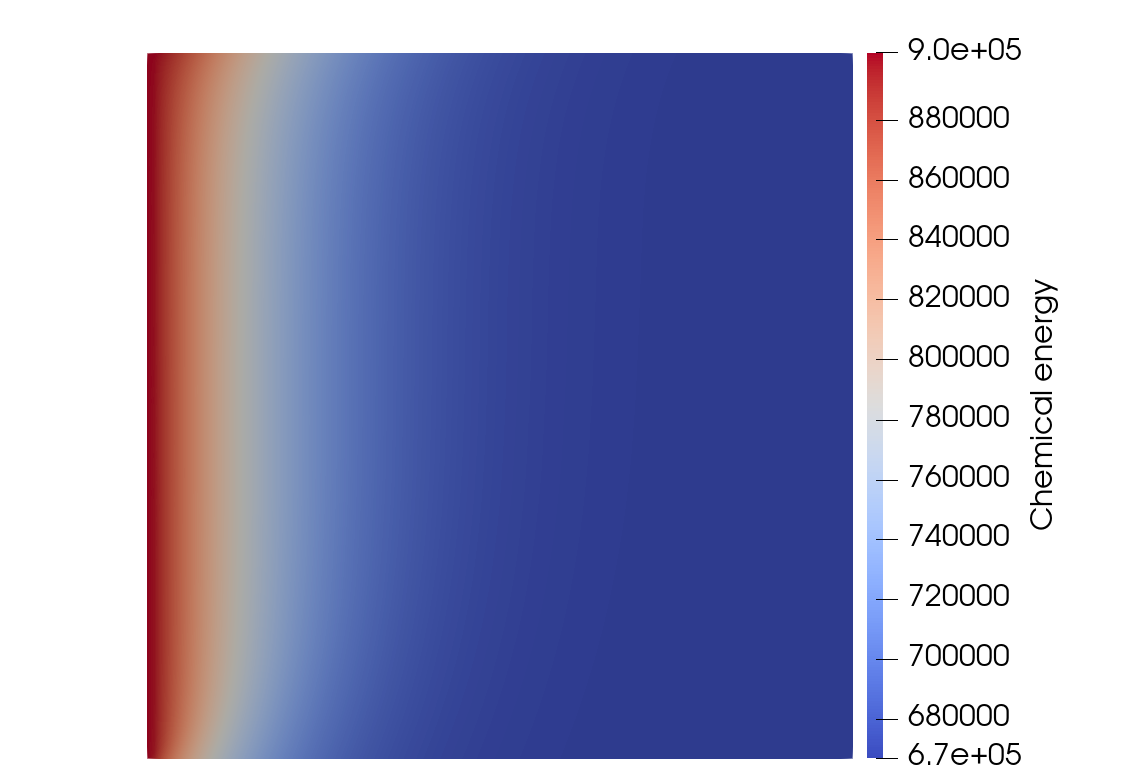}
		
		\includegraphics[width=5.5cm, height=4cm,trim=0.5cm 0cm 0cm 0cm,clip]{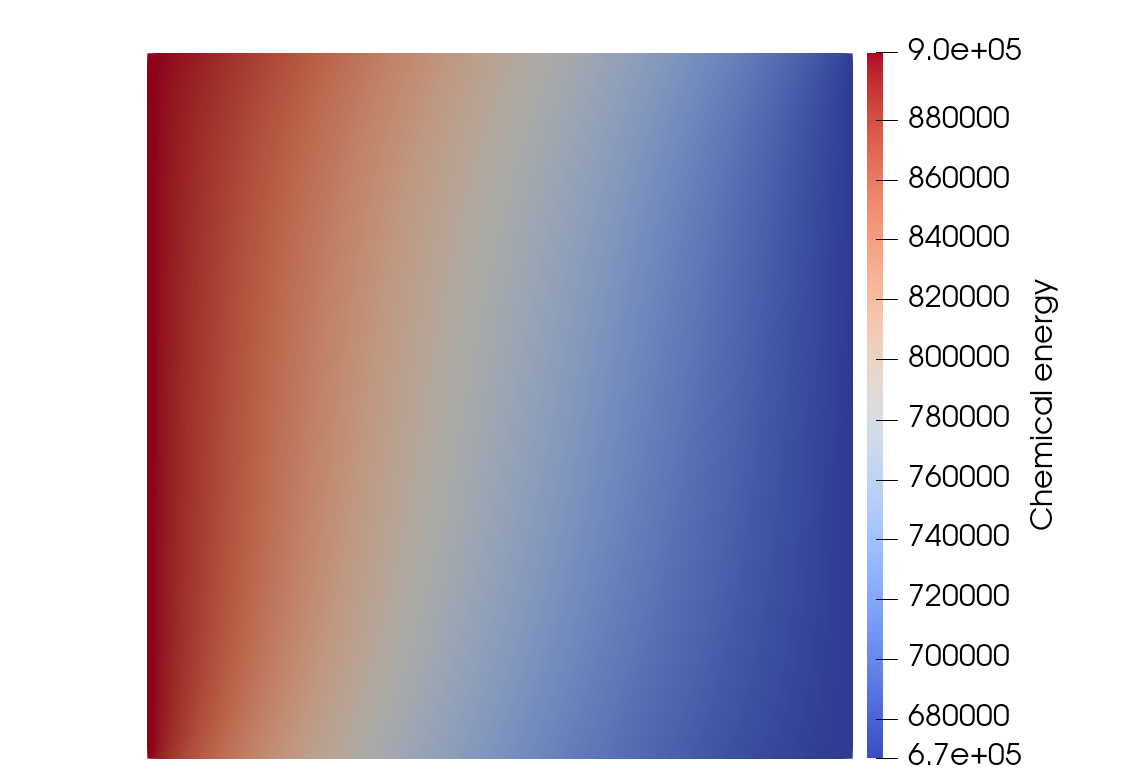}
		\includegraphics[width=5.5cm, height=4cm,trim=0.5cm 0cm 0cm 0cm,clip]{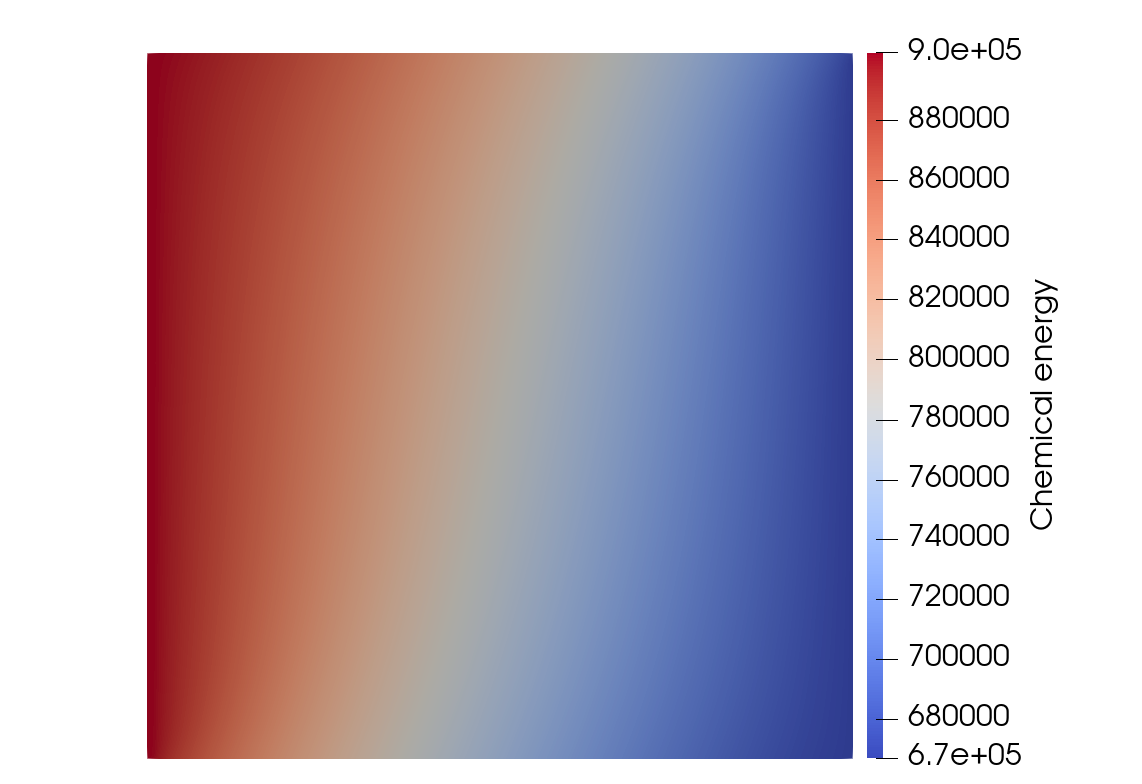}
		\caption{Distributions of chemical potential at different times in Example 3. Top-left: $n = 100$. Top-right: $n = 200$. Bottom-left: $n = 1000$. Bottom-right: $n = 2000$.}\label{fig2-ch}
	\end{figure}
	
	\begin{figure}[htbp]
		\centering
		\includegraphics[width=5.5cm, height=4cm,trim=0.5cm 0cm 0cm 0cm,clip]{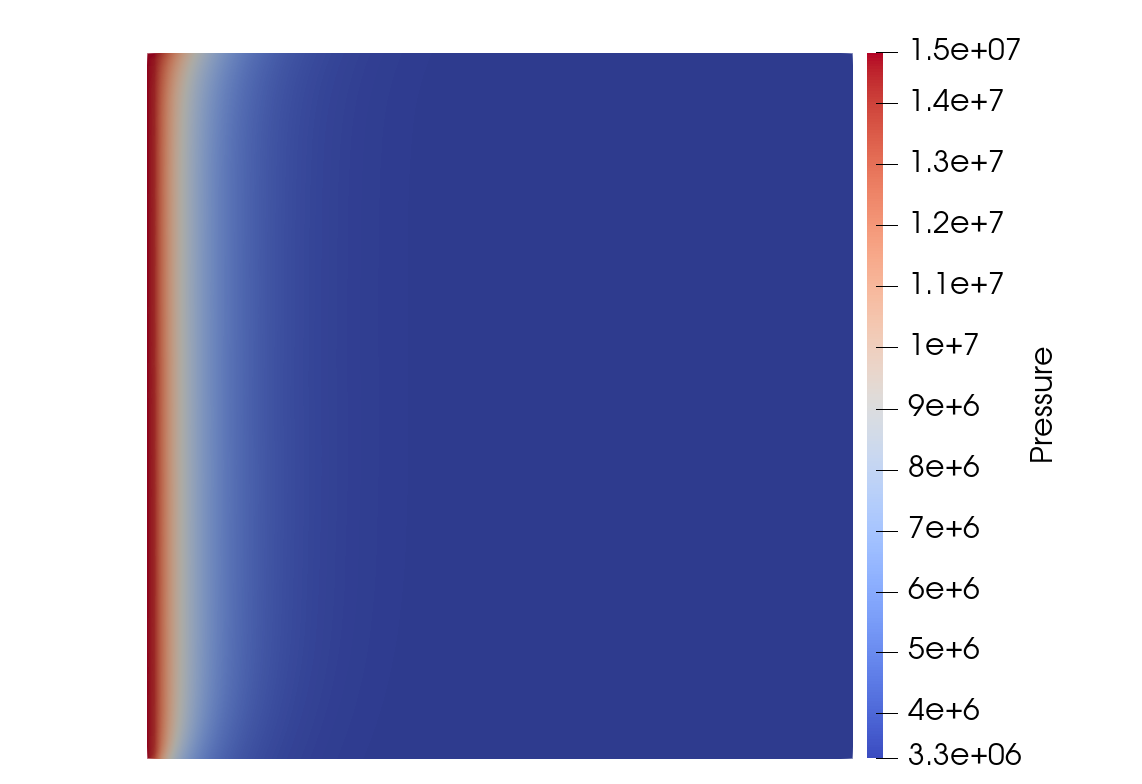}
		\includegraphics[width=5.5cm, height=4cm,trim=0.5cm 0cm 0cm 0cm,clip]{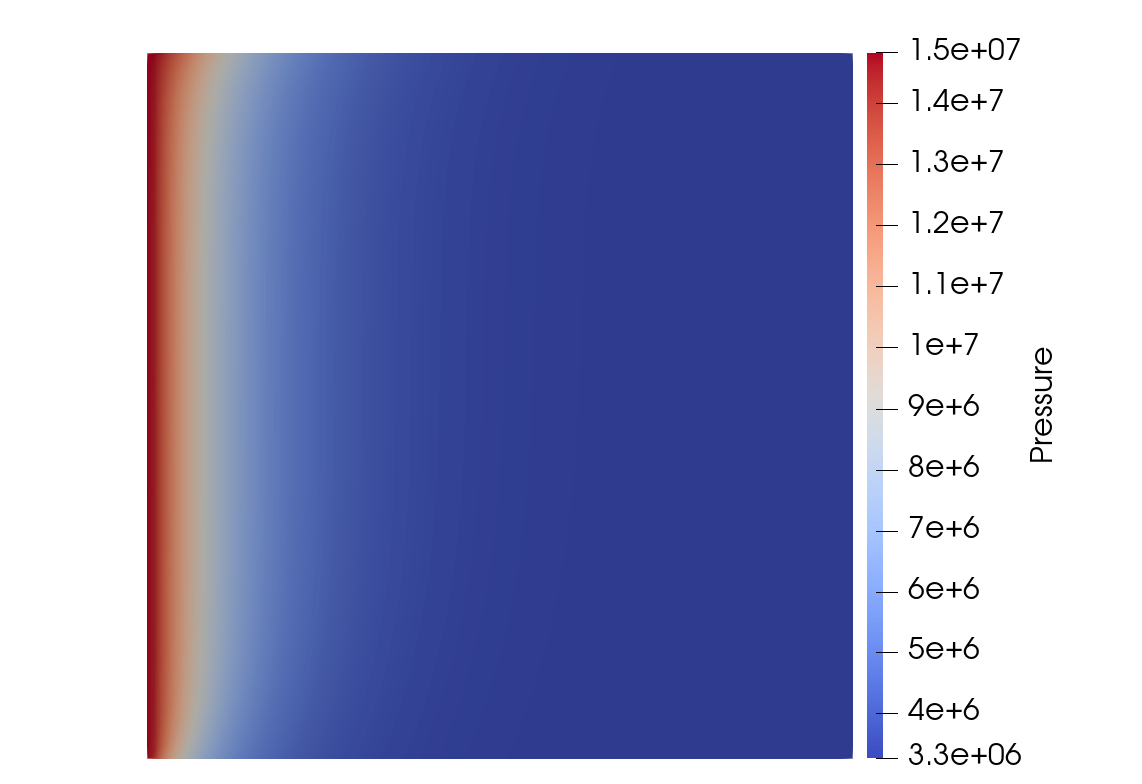}
		
		\includegraphics[width=5.5cm, height=4cm,trim=0.5cm 0cm 0cm 0cm,clip]{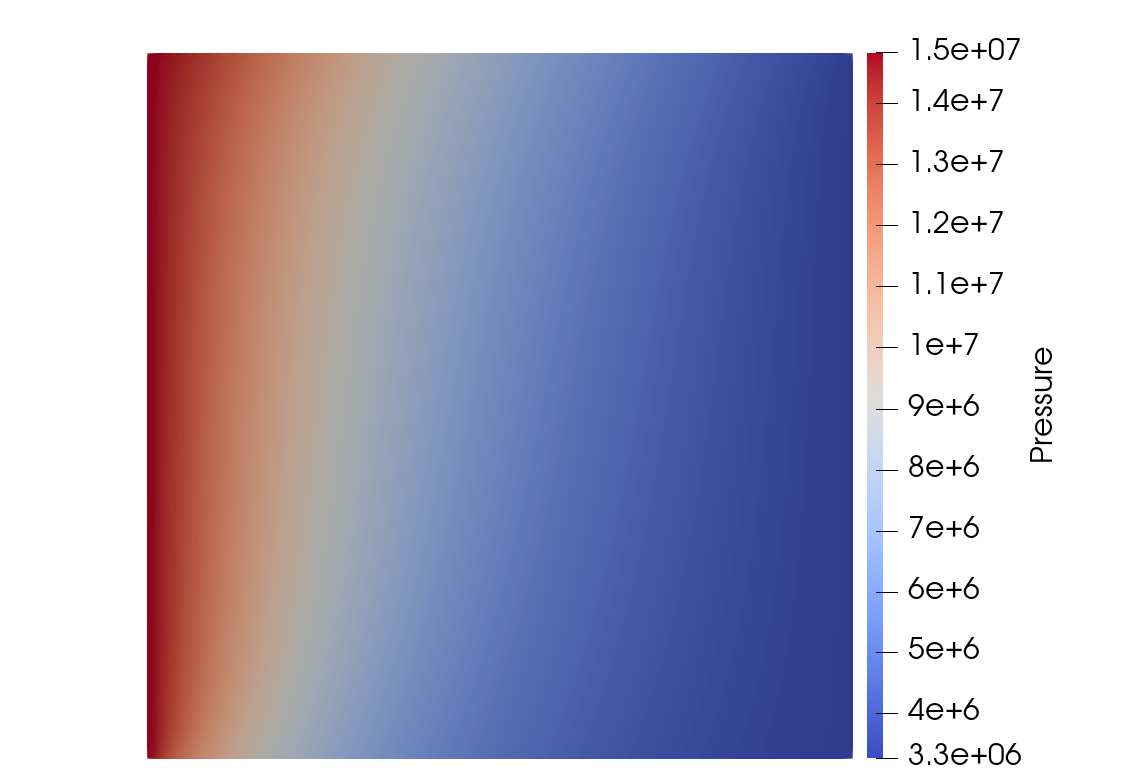}
		\includegraphics[width=5.5cm, height=4cm,trim=0.5cm 0cm 0cm 0cm,clip]{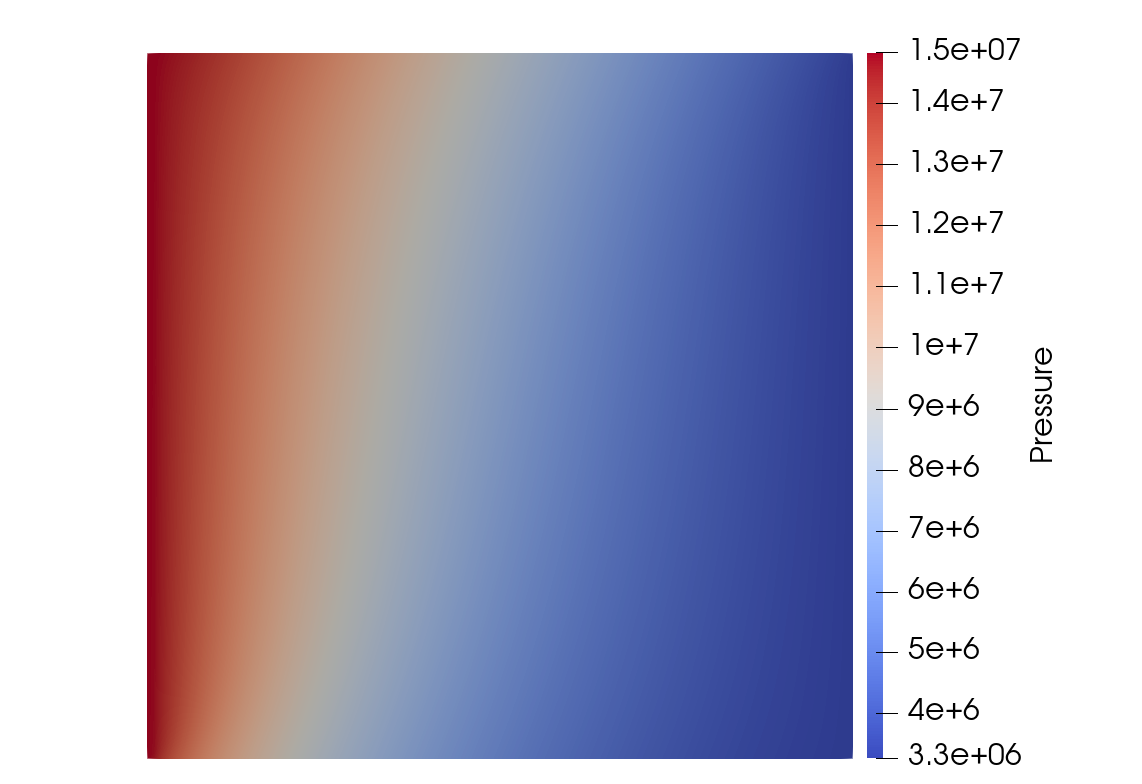}
		\caption{Distributions of pressure at different times in Example 3. Top-left: $n = 100$. Top-right: $n = 200$. Bottom-left:$ n = 1000$. Bottom-right: $n = 2000$.}\label{fig2-pres}
	\end{figure}
	\subsection{Example 4}
	In the third example, we demonstrate the robustness of the proposed method in three-dimensional (3D) space. The initial distribution of permeability is generated using the following randomized approach:
	$$
	\kappa^0 = \kappa_0+\operatorname{rand}(\boldsymbol{x}) \cdot\left(\kappa_1-\kappa_0\right),\\
	$$
	where $\kappa_0 = 0.5$ md and $\kappa_1 = 10$ md which represent the lower and upper bounds of the permeability, respectively. Here, 
	rand($\boldsymbol{x}$) denotes a random function that generates spatially varying values between 0 and 1. For this test, we set $\delta = 0.5$. The spatial discretization is performed using a uniform tetrahedral mesh with 162,000 elements, ensuring sufficient resolution for accurate 3D simulations. Figure \ref{fig3-mass} demonstrates that the proposed scheme preserves key physical properties, including the dissipation property, conservation of total moles, and boundedness of molar density. These results confirm the numerical stability and physical consistency of the method. Figure \ref{fig3-time} illustrates the adaptive values of the stabilization parameter and time steps, highlighting the dynamic adjustment of these parameters to maintain stability and accuracy throughout the simulation. Figure \ref{fig3-c} shows the evolution of molar density over time, capturing the transient behavior of the system in 3D space.
	
	The numerical results demonstrate that the proposed scheme performs effectively for three-dimensional simulations, accurately capturing the complex dynamics of the system while maintaining stability and computational efficiency. The ability to handle spatially varying permeability fields and adaptively adjust stabilization parameters further underscores the robustness and versatility of the method for realistic 3D applications. 
	\begin{figure}[htbp]
		\centering
		\includegraphics[width=5.0cm, height=4.5cm]{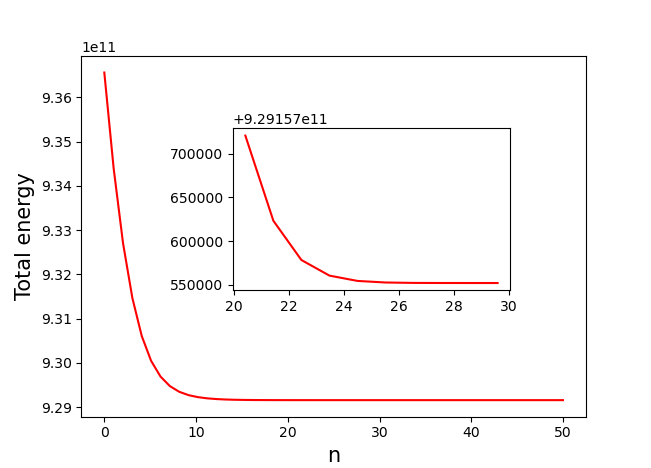}
		\includegraphics[width=5.0cm, height=4.5cm]{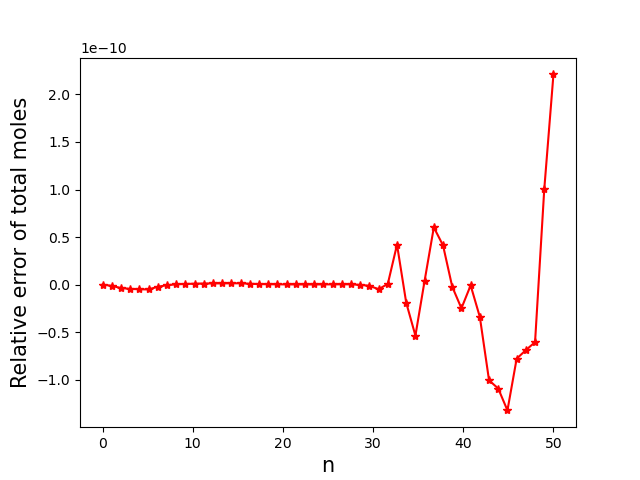}
		\includegraphics[width=5.0cm, height=4.5cm]{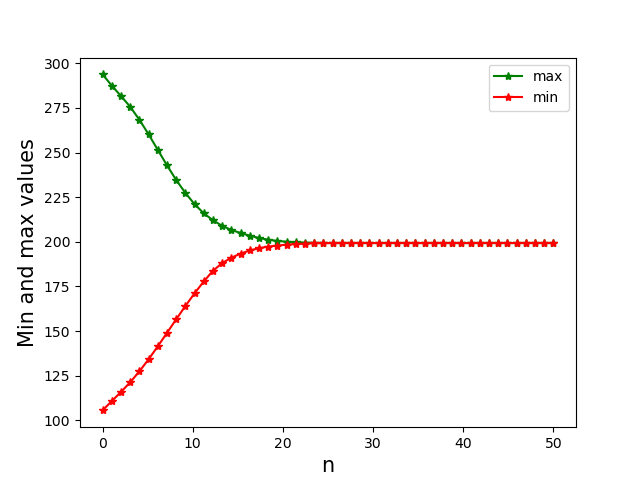}
		\caption{Example 4:  Left: Distributions of energy at different time steps. Middle: Mass conservation at different time steps. Right: Minimum and maximum values of molar density.}\label{fig3-mass}
	\end{figure}
	\begin{figure}[htbp]
		\centering
		\includegraphics[width=5.0cm, height=4.5cm]{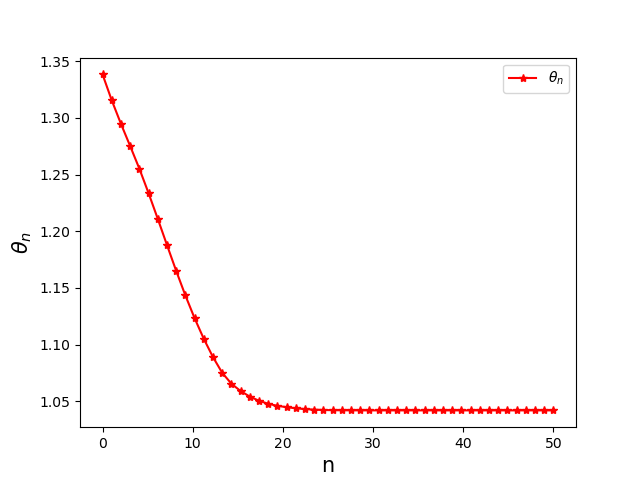}
		\includegraphics[width=5.0cm, height=4.5cm]{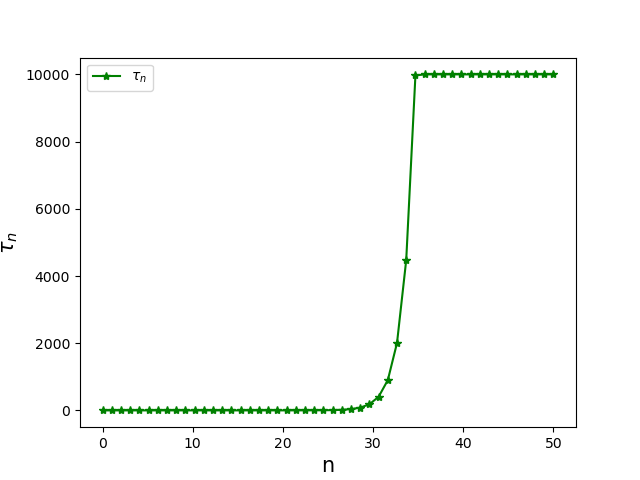}
		\caption{Example 4: Left: Adaptive values of the stabilization parameter at different time steps. Right: Adaptive values of the time step size.}\label{fig3-time}
	\end{figure}
	\begin{figure}[htbp]
		\centering
		\includegraphics[width=5.5cm, height=4cm,trim=0.5cm 0cm 0cm 0cm,clip]{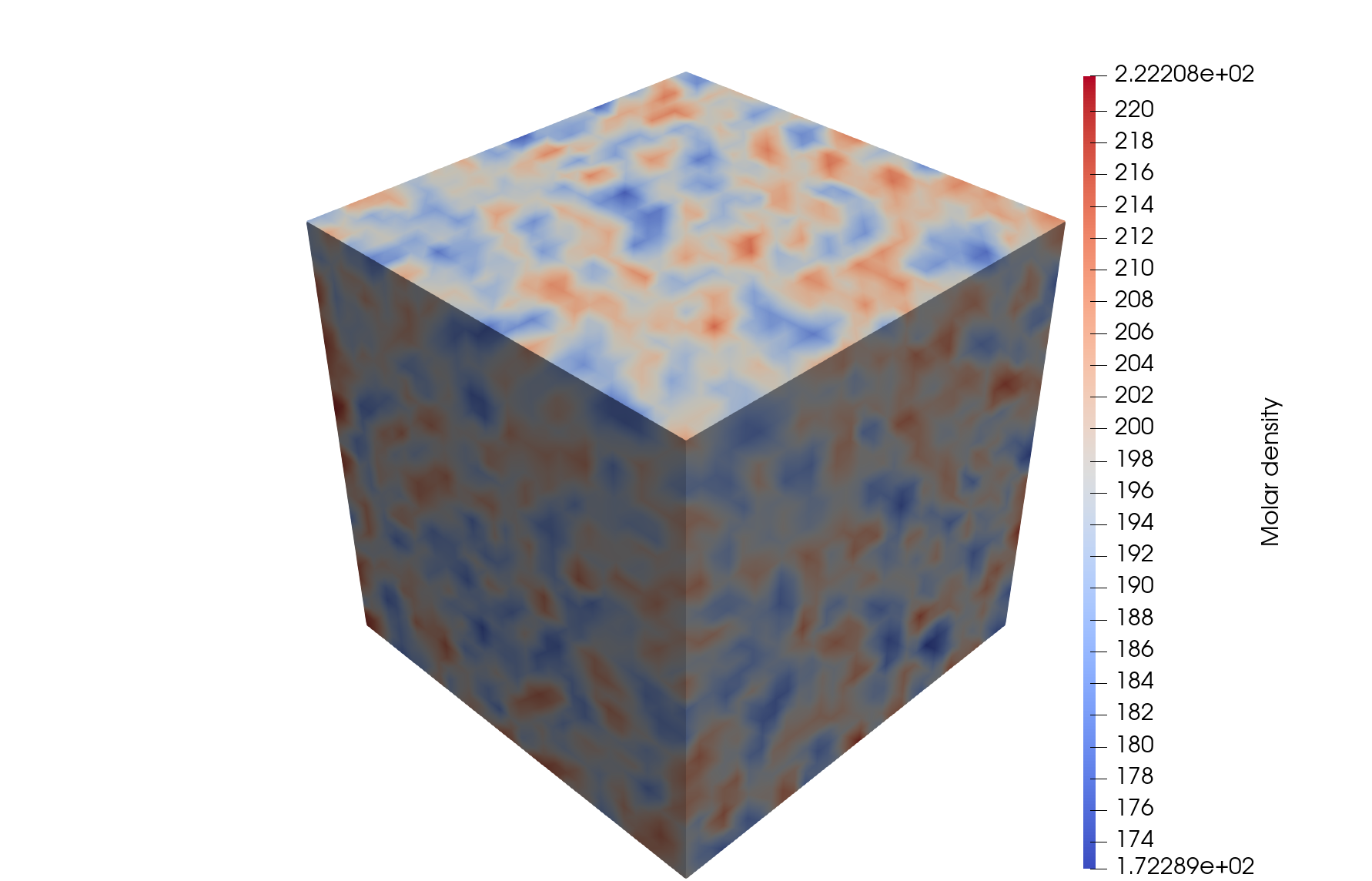}
		\includegraphics[width=5.5cm, height=4cm,trim=0.5cm 0cm 0cm 0cm,clip]{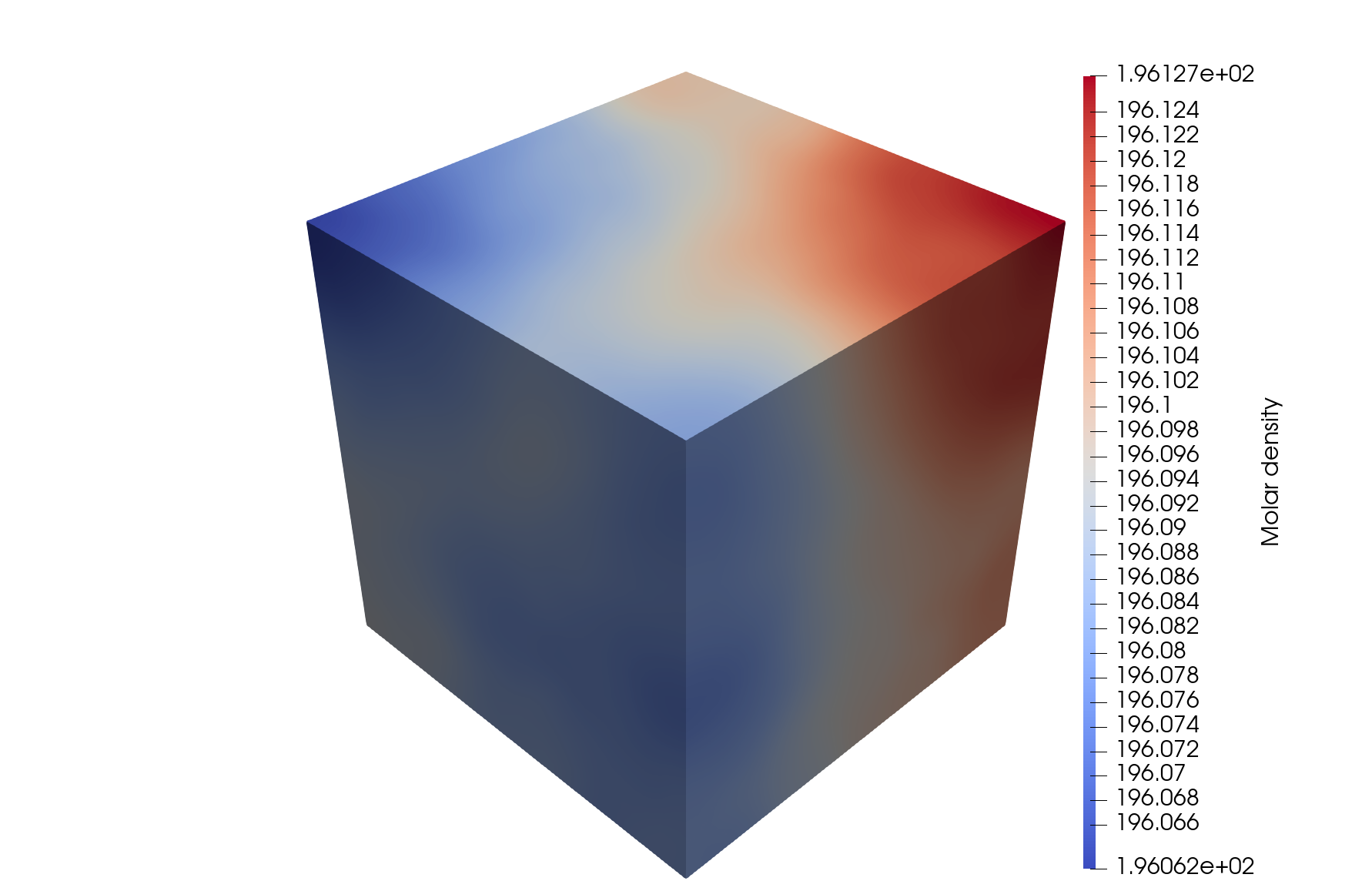}
		\includegraphics[width=5.5cm, height=4cm,trim=0.5cm 0cm 0cm 0cm,clip]{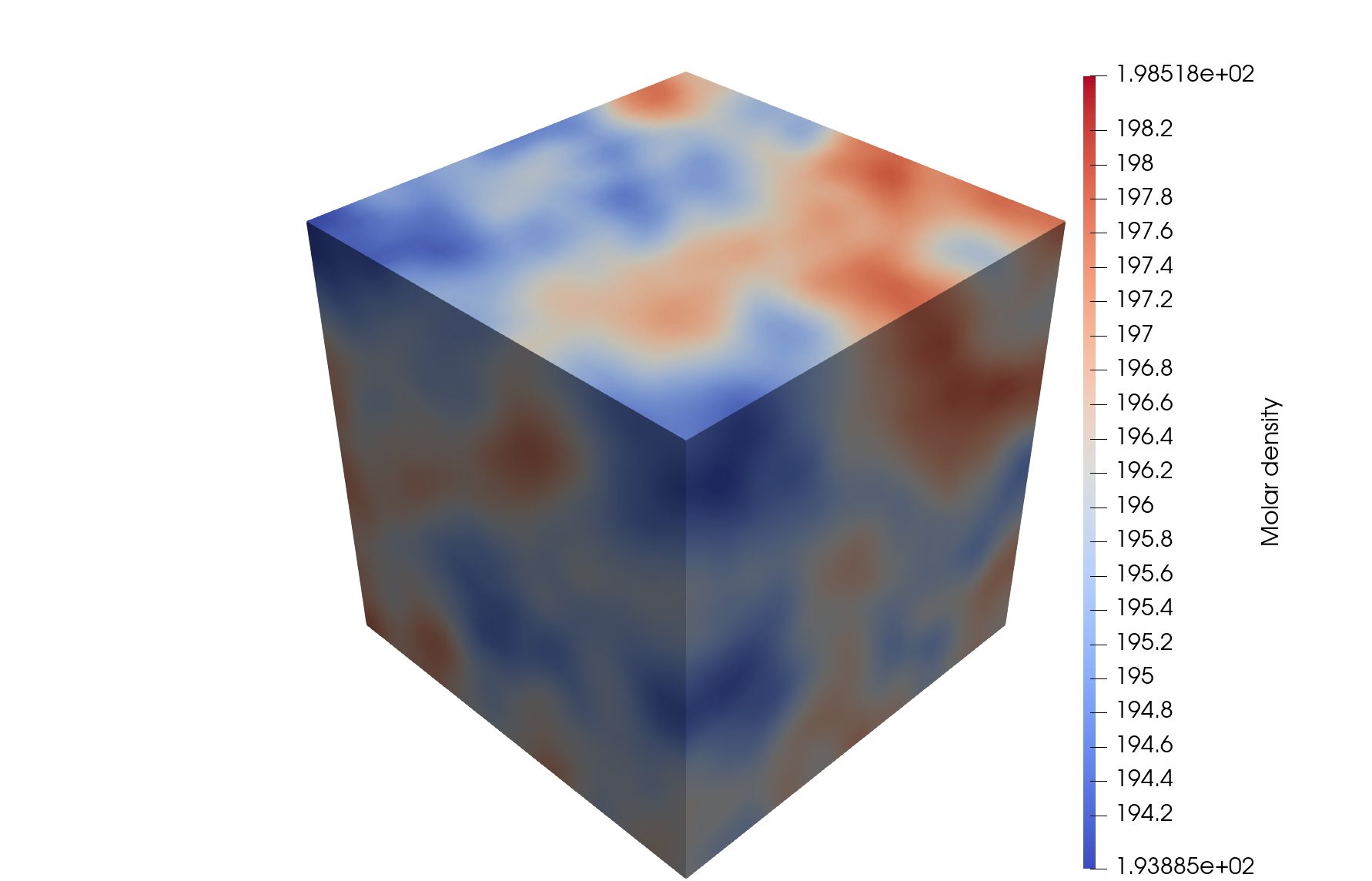}
		\includegraphics[width=5.5cm, height=4cm,trim=0.5cm 0cm 0cm 0cm,clip]{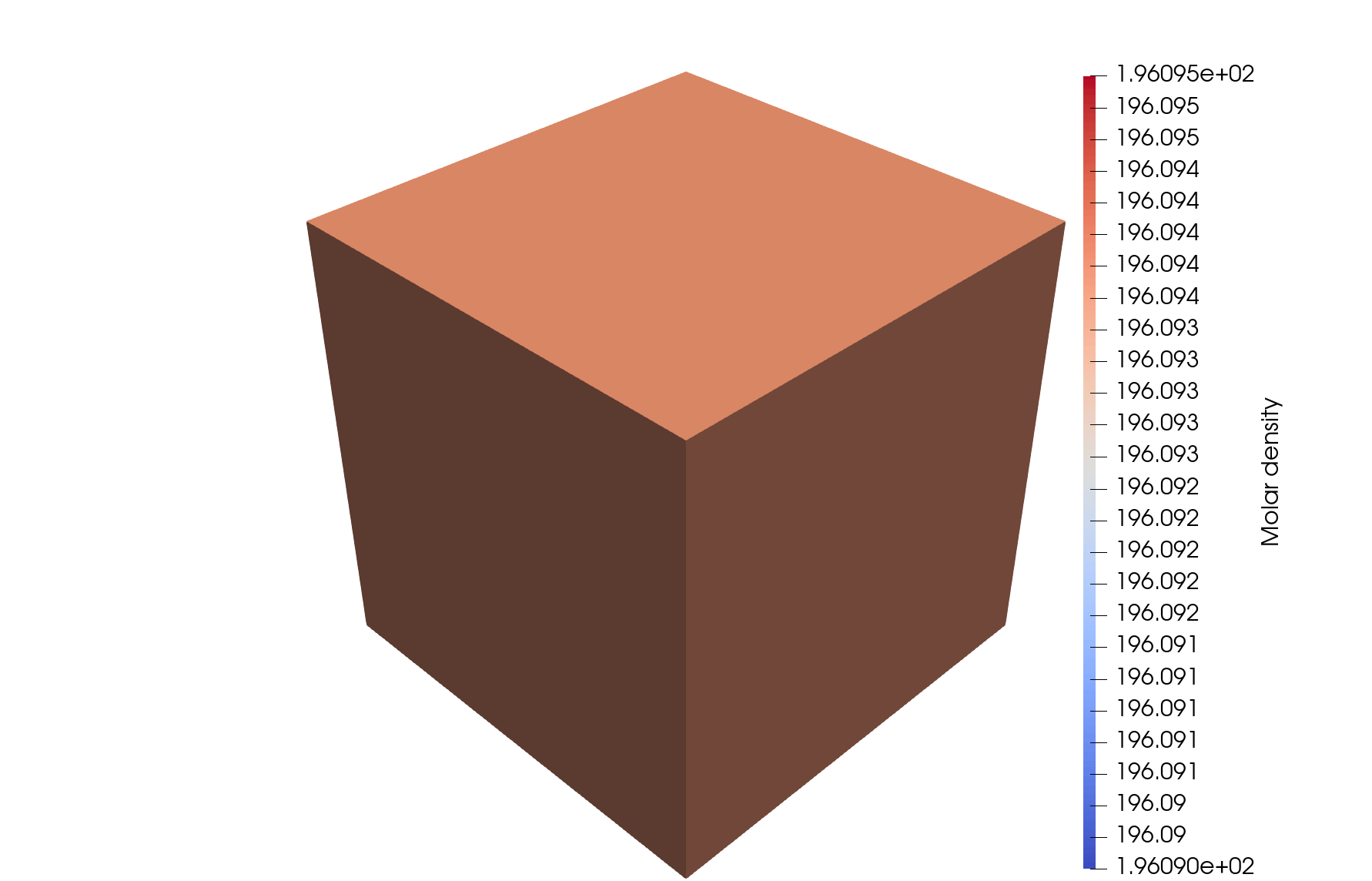}
		\caption{Distributions of molar density at different times in Example 4. Top-left: $n = 10$. Top-right: $n = 20$. Bottom-left:$ n = 30$. Bottom-right: $n = 50$.}\label{fig3-c}
	\end{figure}

	\section{Conclusion}
	
	In this work, we have developed and analyzed a linear, energy-stable numerical method for simulating thermodynamically consistent gas flow in poroelastic media. The proposed approach effectively addresses the inherent nonlinearity and strong coupling between gas dynamics and rock deformation while rigorously preserving key physical properties, including mass conservation, energy dissipation, and bounded molar density. By introducing a stabilization strategy that retains the original energy structure, we achieve a linear iterative scheme with provable convergence. The adaptive stabilization parameter and time-stepping strategy further enhance computational efficiency and robustness of the algorithm, particularly in highly dynamic or nonlinear regimes. The mixed finite element formulation, combined with upwind discretization, ensures stable and physically consistent spatial resolution. Extensive numerical experiments demonstrate the accuracy, stability, and efficiency of the proposed method, confirming its potential as a reliable tool for simulating complex gas-poroelastic systems. Future work may explore its extension to multi-component gas mixtures or multi-phase flow problems in geomechanics and subsurface energy applications.
	
	\section*{Acknowledgement}
	The work of Huangxin Chen is supported by National Key Research and Development Project of China (Grant No. 2023YFA1011702) and National Natural Science Foundation of China (Grant No. 12471345, 12122115). The work of Shuyu Sun is supported by National Key Research and Development Project of China (Grant No. 2023YFA1011701) and National Natural Science Foundation of China (Grant No. 12571466).

\end{document}